\documentclass[11pt]{article}
\usepackage[utf8]{inputenc}
\usepackage{palatino}
\usepackage{amssymb}
\usepackage{amsfonts}
\usepackage{amsmath}
\usepackage{amsthm}
\usepackage{latexsym}
\usepackage{color}
\usepackage{epsfig}
\usepackage{hyperref}
\usepackage{bm}
\usepackage{comment}

\newcommand{\ignore}[1]{}
\newenvironment{proofof}[1]{\par{\noindent \bf Proof of #1:}}{\qed\par}

\makeatletter
\renewcommand{\subsubsection}{\@startsection{subsubsection}{3}{0pt}{-12pt}{-5pt}{\normalsize\bf}}
\makeatother
\newcommand{\Inner}{\mathrm{In}}
\newcommand{\Outer}{\mathrm{Out}}
\newcommand{\num}{\mathrm{num}}

\newcommand{\blue}[1]{{{#1}}}

\newtheorem{claim}{Claim}
\newcommand{\grade}[1]{{{#1}}}

\newtheorem{proposition}[claim]{Proposition}

\newtheorem{lemma}[claim]{Lemma}
\newtheorem*{theorem*}{Theorem}
\newtheorem*{lemma*}{Lemma}

\newtheorem{theorem}[claim]{Theorem}
\newtheorem{obs}[claim]{Observation}
\newtheorem{cond}[claim]{Condition}
\newtheorem{definition}{Definition}
\newtheorem{corollary}[claim]{Corollary}

\newtheorem{question}[claim]{Question}
\newtheorem{fact}[claim]{Fact}
\newtheorem*{definition*}{Definition}

\newcommand{\dtv}{d_{\mathrm TV}}

\newcommand{\R}{\mathbb{R}}

\newcommand{\bP}{{\bf P}}
\newcommand{\bX}{{\bf X}}
\newcommand{\bY}{{\bf Y}}
\newcommand{\bQ}{{\bf Q}}

\newcommand{\bR}{{\bf R}}

\newcommand{\E}{{\bf E}}
\newcommand{\fh}{{f_{\mathsf{junta}}}}
\newcommand{\fhv}{{f_{\mathsf{junta,v}}}}

\newcommand{\Var}{\mathrm{Var}}

\newcommand{\Stab}{\mathsf{Stab}}
\newcommand{\inr}[2]{\langle #1, #2 \rangle}

\begin{document}
 \title{Noise Stability is Computable and approximately Low-Dimensional}
\author{Anindya De\\ 
Northwestern University \\{\tt de.anindya@gmail.com} \and Elchanan Mossel \\ MIT \\ {\tt elmos@mit.edu} \and Joe Neeman \\ UT Austin \\ {\tt neeman@iam.uni-bonn.de}}
 \maketitle

\begin{abstract}
Questions of noise stability play an important role in hardness of approximation in computer science as well as in the theory of voting. 
In many applications, the goal is to find an optimizer of noise stability among all possible partitions of $\R^n$ for $n \geq 1$ to $k$ parts with given Gaussian measures $\mu_1,\ldots,\mu_k$. We call a partition $\epsilon$-optimal, if its noise stability is optimal up to an additive $\epsilon$. 
In this paper, we give an explicit, computable function $n(\epsilon)$
such that an $\epsilon$-optimal partition exists in $\R^{n(\epsilon)}$.
This result has implications for the computability of certain problems
in non-interactive simulation, which are addressed in a subsequent work.

 \end{abstract}
\newpage

\section{Introduction} 

\paragraph{Isoperimetric Theory and Noise Stability.}
Isoperimetric problems have been studied in mathematics since antiquity.
The solution to the isoperimetric problem in the two dimensional Euclidean plane was known in ancient Greece. 
The study of isoperimetric problems is central in modern mathematics and theoretical computer science. Some central examples include the study of expanders and mixing of Markov chains. 

Our interest in this work is in a central modern isoperimetric problem, i.e, the problem of noise stability. This problem, originally studied by Borell in relation to the study of the heat equation in mathematical physics \cite{Borell:85}, has emerged as a central problem in theoretical computer science~\cite{KKMO:07}, as well as in combinatorics and in voting theory~~\cite{Kalai:02} as we elaborate below. 


The fundamental question in this area is to find a partition of Gaussian space (with prescribed measures) which maximizes the noise stability of the partition. We equip $\mathbb{R}^n$ with the standard Gaussian measure, denoted by $\gamma_n$. The Ornstein-Uhlenbeck operator $P_t$ is defined for $t \in [0, \infty)$ and $f: \mathbb{R}^n \rightarrow \mathbb{R}$ by
$$
(P_t f)(x) = \int_{y \in \mathbb{R}^n} f(e^{-t} \cdot x + \sqrt{1- e^{-2t}} \cdot y) d \gamma_n(y).
$$
The resulting notion of noise stability is defined by $\mathsf{Stab}_t(f) := \mathbf{E}[f P_t f]$, where the expectation is with respect to the $\gamma_n$. If $f$ denotes the indicator function of a set (call it $\mathcal{A}_f$), then $\mathsf{Stab}_t(f)$ is the probability that two $e^{-t}$-correlated Gaussians $x,y$ both fall in $\mathcal{A}_f$. In the limit $t \rightarrow 0$, noise stability captures the Gaussian surface area of the set $\mathcal{A}_f$~\cite{Ledoux:94}. In particular, we have the following relation for any set: 
\[
\mathsf{GAS}(\mathcal{A}_f) = \lim_{t \rightarrow 0}  \ \frac{\sqrt{2\pi}}{\arccos (e^{-t})} \cdot \mathbf{E}[ f \cdot (f - P_t f)],
\]
where $\mathsf{GAS}(\mathcal{A}_f)$ denotes the Gaussian surface area of the set $\mathcal{A}_f$.

\paragraph{Halfspaces are most stable sets.}
For both noise stability and surface area, halfspaces are known to be the
optimal sets~\cite{Borell:75, ST:78, Borell:85, Bobkov:97}. We will state this
fact in a slightly convoluted way, in order to more easily generalize it.  Let
$\Delta_k$ denote the probability simplex in $\R^k$ (i.e. the convex hull
formed by the standard unit vector $\{\mathbf{e}_1, \ldots, \mathbf{e}_k\}$).
Any function with range $[k]$ naturally embeds in $\Delta_k$ by identifying $i
\in [k]$ with $\mathbf{e}_i$.
Moreover, we extend the Ornstein-Uhlenbeck operator to act on vector valued
functions in the obvious way: if $f = (f_1, \dots, f_k): \R^n \to \R^k$
then $P_t f = (P_t f_1, \dots, P_t f_k)$.
Finally, say that $f = (f_1, f_2): \R^n \to \Delta_2$ is a \emph{halfspace} if
there exist $a, b \in \R^n$ such that $f_1(x) = 1_{\{\langle x - a, b\rangle \le 0\}}$.
\begin{theorem}[Borell]~\label{thm:Borell}
For any $g: \mathbb{R}^n \rightarrow \Delta_2$, a halfspace $f: \mathbb{R}^n \rightarrow \Delta_2$ with $\mathbf{E}[f] = \mathbf{E}[g]$ satisfies 
$$
\mathbf{E}[\langle f, P_t f \rangle] \ge \mathbf{E}[\langle g, P_t g \rangle]. 
$$
\end{theorem}

Theorem~\ref{thm:Borell} has many applications in computational complexity.
Most famously, it can be combined with an
\emph{invariance principle}~\cite{MOO10}
in order to prove a number of tight hardness-of-approximation results
under the unique games conjecture~\cite{KKMO04}.


\paragraph{More parts?}
It is straightforward to extend the notion of noise stability to partitions
with many parts. Namely, a partition of $\R^n$ can be described by $f:
\mathbb{R}^n \rightarrow [k]$. Identifying $[k]$ with $\{\mathbf{e}_1, \dots,
\mathbf{e}_k\}$ as above, we define the noise stability of such a partition by
$\mathbf{E}[\langle f, P_t f \rangle]$.
One may then ask for an analogue of Theorem~\ref{thm:Borell}, but where
$\Delta_2$ is replaced by $\Delta_k$ for some $k \ge 3$.

Even the three-part version of Theorem~\ref{thm:Borell} turns out to be
rather harder than the two-part one. We will try to explain why, by analogy
with isoperimetric problems. First of all, the Euclidean, isoperimetric
analogue for $k=3$ is known as the ``double bubble'' problem.
The well-known  ``double bubble conjecture'' states that the minimum
total surface area of two bodies separating and enclosing two given (Lebesgue)
volumes is achieved by two spheres meeting at $120^\circ$. After being open
for more than a century, this problem was settled rather recently in
$\R^2$, $\R^3$, and $\R^4$~\cite{HMRR:00, HMRR:02, RHLS:03}.
For the Gaussian space,~\cite{Corwin} showed that for some small but positive
constant $c>0$, the Gaussian surface area of three partitions is minimized by
the ``standard simplex partition" as long as the measures of all the three
parts is within $1/3 \pm c$. 

Since halfspaces both maximize noise stability and minimize Gaussian surface
area, and since the standard simplex partition is known to minimize multi-part
Gaussian surface area in certain cases, it seems natural to guess that the
standard simplex partition also maximizes multi-part noise stability.  This was
explicitly conjectured in~\cite{KKMO04}, in the special case that all of the
parts in the partition have equal measure.  However, a somewhat surprising (at
least to the authors) recent work~\cite{heilman2014} showed that the standard
simplex partition fails to maximize multi-part noise stability \emph{unless}
all of the parts have equal measure.  On the other hand, there is also some
support for the conjecture in the equal-measure case: Heilman~\cite{Heilman:14}
showed that the conjecture is true if the dimension is bounded by a function of
$t$. 

\paragraph{Approximate noise stability of multipartitions?}
In light of the uncertainty about optimal partitioning for $k \ge 3$, one can
ask a more modest question. Given $k \ge 3$, $t>0$, and prescribed measures for
the $k$ parts, let $\alpha_n$ be the optimal noise stability that can
be obtained in $\R^n$ under these constraints. Since the Gaussian measure
is a product measure, $\alpha_n$ is clearly non-increasing in $n$.
Since it is bounded below by zero, it has a limit as $n \to \infty$.

Our main result is that there is an explicitly computable
$n_0=n_0(k,t,\epsilon)$ such that $\alpha_{n_0} \ge \alpha_m - \epsilon$ for
all $m \in \mathbb{N}$. Although the bound on $n_0$ that we give is not
particularly good, the key point is that it is explicitly computable.
As a consequence, up to error $\epsilon$, the noise stable partition is
also explicitly computable.
We conclude the introduction by an open question: 
\begin{question}
Does there exists $n_0$ such that $\alpha_{n_0} = \alpha_n$ for $n > n_0$? 
\end{question}
Our current techniques are not suitable for addressing the question above. 

\subsection{Acknowledgments} 
E.M. acknowledges the support of grant N00014-16-1-2227 from Office of Naval Research and of 
NSF award CCF 1320105.  A.D. acknowledges the support of a start-up grant from Northwestern University.
The authors started this research at the workshop on ``Information theory and concentration phenomena" at the IMA, Minneapolis in 2015. The authors would like to thank the organizers of this workshop and the IMA for their hospitality. 

\section{Main theorem and overview of proof technique} \label{sec:theorem}
In order to state the main theorem, we first need to recall the notion of a polynomial threshold function. A function  $f: \mathbb{R}^n \rightarrow \{0,1\}$ is said to be a degree-$d$ PTF if there exists a polynomial $p: \mathbb{R}^n \rightarrow \mathbb{R}$ of degree $d$ such that $f(x) = 1$ if and only if $p(x) > 0$. We will need a $k$-ary generalization of this definition. We note that there are several possible ways to generalize the notion of PTFs to $k$-ary PTFs and our particular choice is dictated by the convenience of using the relevant results from \cite{DS14}. 

\begin{definition}
A function $f: \mathbb{R}^n \rightarrow [k]$ is said to be a multivariate PTF is there exists polynomials $p_1, \ldots, p_k : \mathbb{R}^n \rightarrow \mathbb{R}$ such that 
$$
f(x) = \begin{cases}
j  & \textrm{if} \  p_j(x)>0 \ \textrm{ and  for } i \not =j, \  p_{i}(x) \le 0 \\
1 & \textrm{otherwise}\end{cases}
$$
In this case, we denote $f = \mathsf{PTF}(p_1, \ldots, p_k)$. Further, $f$ is said to be a degree-$d$ multivariate PTF if   $p_1, \ldots, p_k$ are  degree $d$ polynomials.
\end{definition}
We now state the main theorem of this paper. We set the convention, that unless explicitly mentioned otherwise, the underlying distribution is $\gamma_n$, the standard $n$-dimensional Gaussian measure. 
Likewise, given any random variable $X$ over $\mathbb{R}^k$, $\mathbf{E}[X]$ denotes its (vector-valued) expectation and $\mathsf{Var}(X)$ denotes its covariance matrix.  
\begin{theorem}~\label{thm:main-stab}
Let $f: \mathbb{R}^n \rightarrow [k]$ such that $\mathbf{E}[f] = \mu \in \mathbb{R}^k$. Then, given any $t>0, \epsilon>0$, there exists an explicitly computable $n_0 = n_0(t, k, \epsilon)$ and $d= d(t, k, \epsilon)$ such that there is a degree-$d$ PTF $g: \mathbb{R}^{n_0} \rightarrow [k]$ such that 
\begin{enumerate}
\item $\Vert \mathbf{E}[f] - \mathbf{E}[g] \Vert_1 \le \epsilon$. 
\item $\mathbf{E}[\langle g, P_t g \rangle] \ge \mathbf{E}[\langle f, P_t f \rangle] - \epsilon$.
\end{enumerate}

\end{theorem}
Note that the above theorem automatically implies that a function $g$ satisfying the above properties can be explicitly computed (up to some additional error $\epsilon$). This is because the set of degree-$d$ PTFs on $\mathbb{R}^{n_0}$ admits 
a finite sized explicitly enumerable $\epsilon$-cover.  

\paragraph{Proof Sketch}
The proof of Theorem~\ref{thm:main-stab} consists of the following main steps: 

\paragraph{From general partitions to PTF.} 
The first step in the proof is to show that given any $f: \mathbb{R}^n \rightarrow [k]$, there is a multivariate PTF $g': \mathbb{R}^n \rightarrow [k]$ which meets the two criteria in Theorem~\ref{thm:main-stab} and has degree $d= d(t, k, \epsilon)$, for some explicit function $d(t, k ,\epsilon)$.  
In other words, $g'$ satisfies $\Vert \mathbf{E}[f] - \mathbf{E}[g'] \Vert_1 \le \epsilon$ and  $\mathsf{Stab}_t(g') \ge \mathsf{Stab}_t(f')  - \epsilon$. 
 This is done in Section~\ref{sec:reduction}. Note that main difference between the desired conclusion of Theorem~\ref{thm:main-stab} and what is accomplished in this step is that the ambient dimension remains $n$ as opposed to a bounded dimension $n_0$. 
 
Why is this true? The basic intuition is that if $f$ is noise stable then it should have most of its Hermite expansion weight at low degree. 
Therefore we should be able to replace $f$ with the PTF where the polynomial is the truncated expansion of $f$. 
There are a number of challenges in formalizing this intuition: 1. We cannot rule out that a positive fraction of the weight of $f$ is at high degrees (perhaps as large as $n$).
2. It is not clear that the PTF obtained this way is noise stable nor that 3. It has the right expected value. 

Our analysis proceeds as follows. We would like to construct $g'$ from $f$ by ``rounding''
$P_t f$ for some small $t$. The advantage of $P_t f$ over $f$ is that $P_t f$ is guaranteed to have decaying tails. 
The rounding of $P_t f$ can be performed given some $a \in \mathbb{R}^n$ by considering the function
$g_a: \mathbb{R}^n \to [k]$ which takes the value $i$ whenever $i$ is the largest
coordinate of $P_t f - a$. It is not hard to prove that it is possible to choose $a$ such
that $\E[g_a] = \E[f]$; moreover, one can show that this function $g_a$ has better noise
stability than $f$ does. The main obstacle is that the function $g_a$ is not a PTF. 
Unfortunately the Hermite decay of $P_t f$ does not translate to Hermite decay of $g_a$. 
Instead we use smoothed analysis to show that for most $a$'s, $g_a$ has Hermite decay and can therefore be 
well approximated by PTF. The smoothed analysis argument uses the co-area formula and gradient bounds and draws on ideas from \cite{Neeman14, KNOW14}.

\paragraph{From PTF in dimension $n$ to a small PTF of bounded degree polynomials}

Given the function $g': \mathbb{R}^n \rightarrow [k]$ of degree $d= d(t, k, \epsilon)$, our next goal is show it is possible to obtain a PTF $g$ on some $n_0 = n_0 (t, k, \epsilon)$ variables such that (i) $\Vert \mathbf{E}[g] - \mathbf{E}[g'] \Vert_1 \le \epsilon$ and (ii) $|\langle g, P_t g \rangle - \langle g', P_t g' \rangle| \le \epsilon$. This part builds on and extends the theory and results of~\cite{DS14}. The key notion introduced in \cite{DS14} is that of an \emph{eigenregular} polynomial. Namely, a polynomial is said to be $\delta$-eigenregular if for the canonical tensor $\mathcal{A}_p$ associated with the polynomial, the ratio of  the maximum singular value to its Frobenius norm is at most $\delta$ (the tensor notions are explicitly defined later). 

  The key advantage of this definition is that as shown in \cite{DS14}, when $\delta \rightarrow 0$, the distribution of $p$ (under $\gamma_n$) converges to a normal.  In other words, eigenregular polynomials obey a central limit theorem.  In fact, given $k$ polynomials $p_1, \ldots, p_k$ which are $\delta$-eigenregular, they also obey a multidimensional central limit theorem.  
  
The \emph{regularity lemma} from \cite{DS14} implies that the polynomials $p_1,\ldots,p_k$ can be jointly expressed as bounded (in terms of $t,k$ and $\epsilon$)  size polynomials in 
eigenregular homogenous polynomials $\{\Inner(p_{s, q, \ell })\}$ where $1 \le s \le k$, $1 \le q \le d$ and $1 \le \ell \le \num(s,q)$. In other words, we may write 
 $p_s = \Outer(p_s) (\{\Inner(p_{s,q, \ell})\}_{q \in [d], \ell \in [\num(s,q)]})$, where $\num(s) = \sum_{q=1}^d \num(s,q)$ is bounded in terms of  $t,k$ and $\epsilon$ and all the Inner polynomials are $\delta$-eigenregular. 

\cite{DS14} used the statement above to conclude that the joint distribution of $p_1,\ldots,p_k$ can be approximated in a bounded dimension as we can replace each of the inner polynomials by a one dimensional Gaussian. For our application things are more delicate, as we are not only interested in the joint distribution of $p_1,\ldots,p_k$ but also in the noise stability of $p_1,\ldots,p_k$.  For this reason it is important for us to maintain the degrees of the inner polynomials (each of which is homogenous) and not replace them with Gaussians. 

\paragraph{A small PTF representation}
In the final step of the proof, we maintain $\Outer(p_s)$ and show how that polynomials 
$\{\Inner(p_{s,q, \ell})\}$ can be replaced by a collection of polynomials  $\{\Inner(r_{s,q, \ell})\}$ in bounded dimensions (in $t,k$ and $\epsilon$) thus completing the proof. 
The fact that a collection of homogenous polynomials can be replaced by polynomials in bounded dimensions is a tensor analogue of the fact that for any $k$  vectors in $\mathbb{R}^n$, there exist $k$ vectors in $\mathbb{R}^k$ with the same matrix of inner products. Once such polynomials are found, it is not hard to construct eigenregular polynomials from them by averaging the polynomials over independent copies of random variables.

\section{Applications} 

Given the wide applicability of Borell's isoperimetric result to combinatorics and theoretical computer science, we believe that Theorem~\ref{thm:main-stab} will also be widely applicable. We will now point out some applications of this theorem. First, by combining Theorem~\ref{thm:main-stab} with the invariance principle~\cite{MOO10}, we  
derive a weak $k$-ary analogue of ``Majority is Stablest". 
The analogue of the Ornstein-Uhlenbeck operator is the so-called \emph{Bonami-Beckner operator} defined as follows: for $\rho \in [-1,1]$ and $x \in \{-1,1\}^n$, let $\mathcal{D}_{\rho}(x)$ be the product distribution over $\{-1,1\}^n$ such that for $y \sim \mathcal{D}_{\rho}(x)$, for all $1 \le i \le n$, $\mathbf{E}[x_i y_i] =\rho$. 
 Then, for any $f: \{-1,1\}^n \rightarrow \mathbb{R}^k$, $T_{\rho} f(x) = \mathbf{E}_{y \sim D_{\rho}(x)} [f (y)]$. Likewise, for any $i \in [n]$ and $z \in \{-1,1\}^{n-1}$, let $f_{z,-i}: \{-1,1\} \rightarrow \mathbb{R}^k $ denote the function obtained by restricting all but the $i^{th}$ coordinate to $z$. Then, $\mathsf{Var}(f_{z,-i})  = \mathbf{E}_{z}[\Vert f_{z,-i} - \mathbf{E}_{z}[f_{z,-i}] \Vert_2^2]$. 
 Define the influence of the $i^{th}$ coordinate on $f$ by
$$
\mathsf{Inf}_i(f)= \mathbf{E}_{z \in \{-1,1\}^{n-1}} [\mathsf{Var}(f_{z,-i})]. 
$$

\begin{theorem}~\label{thm:Plurality}
Given any $k \in \mathbb{N}$ and $\rho \in [0,1],\epsilon>0$, $\exists n_0 = n_0(k,\epsilon,\rho), C = C(k,\epsilon,\rho)$ (which is explicitly computable) such that the following holds: For any $\boldsymbol{\mu} = (\mu_1, \ldots, \mu_k) \in \Delta_k$ and $n \geq n_0$, there is an explicitly computable $g = g_{\boldsymbol{\mu}, \rho, \epsilon}: \{-1,1\}^{n} \rightarrow [k]$ such that $\max_i \mathsf{Inf}_i(g) \le C/\sqrt{n}$ and $|\Pr_{x \in \{-1,1\}^{n_0}}[g(x) = i ] - \mu_i| \leq  C / \sqrt{n}$ and for any $f: \{-1,1\}^n \rightarrow [k]$ such that   $|\Pr_{x \in \{-1,1\}^n} [f(x) = i ] - \mu_i| \leq \epsilon$ and $\max_i \mathsf{Inf}_i(f) \le \epsilon$, 
\[
\mathbf{E}[\langle f, T_{\rho} f \rangle] \le \mathbf{E}[\langle g, T_{\rho} g \rangle] + 2 k \epsilon. 
\]
\end{theorem}
The proof of the theorem, which is omitted, is by now a standard reduction from a Gaussian noise stability result to a discrete one, where in one direction the invariance principle immediately bounds the discrete stability by the Gaussian stability. In the other direction, starting for an $\epsilon$-optimal Gaussian partition, each Gaussian is replaced by a normalized sum of independent variables to obtain a discrete partition.  
The same result holds for other domains, for example for $f,g : [k]^n \to [k]$. In particular, the case where $(\mu_1,\ldots,\mu_k) = (1/k,\ldots,1/k)$ implies that in a tied elections between $k$ alternatives, we can find an $\epsilon$-optimally robust noise stable voting rule, where the stability is with respect of each candidate randomizing their vote independently with probability $1-\rho$. 

\subsection{Relationship to rounding of SDPs}
To the best of our knowledge our results are independent of the the results of Raghavendra and Steurer~\cite{RaghavendraSteurer:09} who showed that for any CSP, there is an a rounding algorithm that is optimal up to $\epsilon$, whose running time is polynomial in the instance size and doubly exponential in $1/\epsilon$. 
It is natural to suspect that the two results are related, since there are well-known connections between SDP rounding and Gaussian noise stability.

However, the usual analysis relating rounding to Gaussian noise stability seems
to require that halfspaces maximize noise stability for all possible values of
the noise. In our results, this property does not necessarily hold.

In the other direction, it is tempting to try to cast the noise stability problem as an optimization problem on Gaussian graphs and then apply the results of Steurer and Raghavendra to obtain explicit bounds on the dimension where an almost optimal solution can be achieved. It is hard to implement this approach for two reasons: first, we do not know the SDP solution for the Gaussian graph; second, we are interested in the optimal solution and it is not clear what is the relation between the best integral solution and the SDP solution for the Gaussian graph. 

While we do not see how to formally relate the two works, connecting the two
(if possible) would surely yield important insights.

\subsection{Non-interactive correlation distillation}
Next, we talk about a basic problem in information theory and communication complexity which was recently considered in the work of Ghazi, Kamath and Sudan~\cite{GKS:16}.  Let there be two non-communicating players Alice and Bob who have access to independent samples from a joint distribution $\bP = (\bX,\bY)$ on the set $\mathcal{A} \times \mathcal{B}$. In other words, Alice (resp. Bob) have access to $(x_1, x_2, \ldots, )$ (resp. $(y_1, y_2, \ldots)$) such that $x_i \in \mathcal{A}$, $y_i \in \mathcal{B}$ and for each $i \in \mathbb{N}$, $(x_i, y_i)$ is distributed according to $\bP$, and the random variables $\{(x_i, y_i)\}_{i=1}^n$ are mutually independent. Let $\boldsymbol{\mu} = (\mu_1, \ldots, \mu_k) \in \Delta_k$ and $\boldsymbol{\nu} = (\nu_1, \ldots, \nu_k) \in \Delta_k$.  What is the maximum $\kappa \in [0,1]$ such that Alice and Bob can non-interactively jointly sample 
a distribution $\bQ$ on $[k] \times [k]$ such that the distribution of the marginals of Alice and Bob are $\boldsymbol{\mu}$ and $\boldsymbol{\nu}$ respectively and they sample the same output with probability $\kappa$? 

To formulate this problem more precisely, we introduce some notation.
Let $\bX^n = (\bX_1, \dots, \bX_n)$ and $\bY^n = (\bY_1, \dots, \bY_n)$, where $(\bX_i, \bY_i)$ are independently
drawn from $\bP$.
Now, 
note that a non-interactive protocol for Alice and Bob is equivalent to a pair $(f,g)$ where 
$f: \mathcal{A}^n \rightarrow [k]$ and $g: \mathcal{B}^n \rightarrow [k]$ (for some $n \in \mathbb{N}$). 
In this terminology, the question now becomes the following: given $\boldsymbol{\mu}$, $\boldsymbol{\nu}$, do there exist $n$ and $f: \mathcal{A}^n \rightarrow [k]$ and $g: \mathcal{B}^n \rightarrow [k]$ such that $f(\bX^n) \sim \boldsymbol{\mu}$, $g(\bY^n) \sim \boldsymbol{\nu}$, and $\Pr(f(\bX^n) = g(\bY^n)) = \kappa$?


Before we state the main result of \cite{GKS:16} and our extension, we consider a motivating example. Let $\mathcal{A}= \mathcal{B}= \mathbb{R}$ and let $\bP = (\bX, \bY)$ be two $\rho$-correlated standard Gaussians. Let $k=2$ and $\boldsymbol{\mu} = \boldsymbol{\nu}$. 
Then, Borell's isoperimetric theorem (Theorem~\ref{thm:Borell})
states that the maximum achievable $\kappa$ is given by 
\[
\kappa = \Pr_{\bX,\bY}[f(\bX) = f(\bY)],
\]
where $f : \mathbb{R} \rightarrow \Delta_2$ is a halfspace with measure $\boldsymbol{\mu}$. Thus, in the above case, $n=1$ suffices and $f=g$ is the halfspace whose measure is $\boldsymbol{\mu}$. We now state the main result of \cite{GKS:16}. For the result below, for probability distribution $\bP$, we let $|\bP|$ denote the size of some standard encoding of $\bP$. 

\begin{theorem}~\label{thm:GKS}[Ghazi-Kamath-Sudan]
Let $(\mathcal{A} \times \mathcal{B}, \bP)$ be a probability space, and let $\bX^n$ and $\bY^n$ be as in Theorem~\ref{thm:GKS}.
There is an algorithm running in time $O_{| \bP|, \delta}(1)$ such that
given $\boldsymbol{\mu}$ and $\boldsymbol{\nu}$ in $\Delta_2$ and a parameter $\kappa \in [0,1]$, it distinguishes between the following two cases: 
\begin{enumerate}
  \item There exist $n \in \mathbb{N}$ and
    $f: \mathcal{A}^n \rightarrow \{0, 1\}$ and
    $g: \mathcal{B}^n \rightarrow \{0, 1\}$ such that $f(\bX^n) \sim
    \boldsymbol{\mu}$, $g(\bY^n) \sim \boldsymbol{\nu}$, and $\Pr(f(\bX^n) =
    g(\bY^n)) \ge \kappa - \delta$.  In this case, there is an explicit $n_0 =
    n_0(|\bP|, \delta)$ such that we may choose $n \le n_0$. Further, the
    functions $f$ and $g$ are explicitly computable.  

  \item For any $n \in \mathbb{N}$ and $f: \mathcal{A}^n \rightarrow \{0, 1\}$
    and $g: \mathcal{B}^n \rightarrow \{0, 1\}$, if
    $g: \mathcal{B}^n \rightarrow \{0, 1\}$ satisfy $f(\bX^n) \sim
    \boldsymbol{\mu}$ and $g(\bY^n) \sim \boldsymbol{\nu}$ then $\Pr(f(\bX^n) =
    g(\bY^n)) \le \kappa - 8\delta$.
\end{enumerate}
\end{theorem}
In other words, the above theorem states that there is an algorithm which, given $\bP$, target marginals $\boldsymbol{\mu}$, $\boldsymbol{\nu}$, and correlation $\kappa$, can distinguish between two cases: (a) In the first case, it is possible for Alice and Bob to non-interactively simulate a distribution $\bR$ which has the correct marginals and achieves the correlation $\kappa$ up to $\delta$. In this case, 
there is an explicit bound on the number of copies of $\bP$ required and the algorithm also outputs the functions $f,g$ used for the non-interactive simulation. 
(b) In the second case, no non-interactive protocol between Alice and Bob can simulate a distribution $\bR$ such that it has the correct marginals and target correlation to error at most $8 \delta$. 

We remark that the requirement for the marginals to match \emph{exactly} is
unimportant. Indeed, any protocol where the marginals match approximatly
can be ``fixed'' to have exact marginals, with a small loss in the correlation.



The main restriction of Theorem~\ref{thm:GKS} is that the output of
the non-interactive protocol is a pair of bits. Theorem~\ref{thm:main-stab}
immediately implies the following modification of Theorem~\ref{thm:GKS}:
we may replace $\{0, 1\}$ by $[k]$ wherever it appears, provided that we
also assume that $\bP$ is a Gaussian measure.
In the best of both worlds, we would be able to replace $\{0, 1\}$ by $[k]$
without adding the assumption that $\bP$ is a Gaussian measure.
Using the methods we develop here, this also turns out to be possible; we
provide details in a follow-up work.

\section{Preliminaries} 
We will start by defining some technical preliminaries which will be useful for the rest of the paper. 
\begin{definition}
For $k \in \mathbb{N}$ and $1\le i \le k$, let $\mathbf{e}_i$ be the unit vector along coordinate $i$ and let $\Delta_k$ be the convex hull formed by $\{ \mathbf{e}_i \}_{1\le i \le k}$. 
\end{definition}
In this paper, we will be working on the space of functions $f: \mathbb{R}^n \rightarrow \mathbb{R}$ where the domain is equipped with the standard $n$ dimensional normal measure (denoted by $\gamma_n(\cdot)$). Unless explicitly mentioned otherwise, all the functions considered in this paper will be in $L^2(\gamma_n)$. A key property of such functions is that they admit the so-called Hermite expansion. Let us define a family of polynomials $H_q: \mathbb{R} \rightarrow \mathbb{R}$ (for $q \ge 0$) as 
$$
H_0(x) = 1 ; \ H_1(x) = x ; \  H_q(x) = \frac{(-1)^q}{\sqrt{q!}}  \cdot e^{x^2/2}  \cdot \frac{d^q}{dx^q} e^{-x^2/2}. 
$$
Let $\mathbb{Z}^{\ast}$ denote the subset of non-negative integers and $S \in \mathbb{Z}^{\ast n}$. Define $H_S: \mathbb{R}^n \rightarrow \mathbb{R}$ as
$$
H_S(z) = \prod_{i=1}^n H_{S_i}(z_i). 
$$
It is well known that the set $\{H_S\}_{S \in \mathbb{Z}^{\ast n}}$ forms an orthonormal basis for $L^2(\gamma_n)$. In other words, every $f \in L^2(\gamma_n)$
may be written as
$$
f = \sum_{S \in \mathbb{Z}^{\ast n}} \widehat{f}(S) \cdot H_S,
$$
where $\widehat{f}(S)$ are typically referred to as the \emph{Hermite coefficients} and expansion is referred to as the \emph{Hermite expansion}. The notion of Hermite expansion can be easily extended to $f: R^n \rightarrow \mathbb{R}^k$ as follows: Let $f = (f_1, \ldots, f_k)$ and let 
$$
f_i  =  \sum_{S \in \mathbb{Z}^{\ast n}} \widehat{f_i}(S) \cdot H_S. 
$$
Then, the Hermite expansion of $f$ is given by $\sum_{S \in \mathbb{Z}^{\ast n}} \widehat{f}(S) \cdot H_S$ where $\widehat{f}(S) = (\widehat{f_1}(S), \ldots, \widehat{f_k}(S))$. 
In this setting, we also have Parseval's identity:
\begin{equation}\label{eq:parseval}
\int \Vert f (x) \Vert_2^2  \ \gamma_n(x) dx = \sum_{S \in \mathbb{Z}^{\ast n}} \Vert \widehat{f}(S) \Vert_2^2
\end{equation}
For $f: \mathbb{R}^n \rightarrow \mathbb{R}^k$ and $d \in \mathbb{N}$, define
$f_{\le d} : \mathbb{R}^n \rightarrow \mathbb{R}^k$ by
$$
f_{\le d} (x) = \sum_{S: |S| \le d} \widehat{f}(S) \cdot H_S(x). 
$$
We will define $\mathsf{W}^{\le d} [f] = \Vert f_{\le d} \Vert_2^2$ and $\mathsf{W}^{> d} [f] = \sum_{|S|>d} \Vert \widehat{f}(S) \Vert_2^2$. 
\subsubsection*{Ornstein-Uhlenbeck operator} 
\begin{definition}
The Ornstein-Uhlenbeck  operator $P_t$ is defined for $t \in [0, \infty)$ such that for any $f: \mathbb{R}^n \rightarrow \mathbb{R}^k$, 
$$
(P_t f)(x) = \int_{y \in \mathbb{R}^n} f(e^{-t} \cdot x + \sqrt{1- e^{-2t}} \cdot y) d \gamma_n(y).
$$
\end{definition}
Note that if $f : \mathbb{R}^n \rightarrow \Delta_k$, then so is $P_t f$ for every $t>0$. 
A basic fact about  the Ornstein-Uhlenbeck operator is that the functions $\{H_S\}$ are eigenfunctions of this operator. We leave the proof of the next proposition to the reader. 
\begin{proposition}
For $S \in \mathbb{Z}^{\ast n}$,
$P_t H_S = e^{-t \cdot |S|} \cdot H_S$. 
\end{proposition}
We define the noise stability of $f: \mathbb{R}^n \rightarrow \Delta_k$ with noise rate $t \in [0, \infty)$ by
\[
\mathsf{Stab}_t(f)  = \langle f , P_t f \rangle. 
\]
\subsubsection*{Multivariate polynomial threshold functions}
We recall the definition of polynomial threshold functions from the introduction (primarily to define an associated quantity which shall be useful later). 
\begin{definition*}
A function $f: \mathbb{R}^n \rightarrow [k]$ is said to be a multivariate PTF is there exists polynomials $p_1, \ldots, p_k : \mathbb{R}^n \rightarrow \mathbb{R}$ such that 
$$
f(x) = \begin{cases}
j  & \textrm{if} \  p_j(x)>0 \ \textrm{ and  for } i \not =j, \  p_{i}(x) \le 0 \\
1 & \textrm{otherwise}\end{cases}
$$
We then denote $f$ as $f = \mathsf{PTF}(p_1, \ldots, p_k)$. Further, $f$ is said to be a degree-$d$ multivariate PTF if   $p_1, \ldots, p_k$ are  degree $d$ polynomials. For the rest of this paper, we will use PTFs to mean multivariate PTFs. The underlying $k$ will be clear from the context. Also, let us define $\mathsf{Collision} (f ) \subseteq \mathbb{R}^n$ defined as 
\[
\mathsf{Collision} (f) = \big\{x \in \mathbb{R}^n: |\{ j: p_j (x) > 0\}| \not =1 \big\}. 
\]
\end{definition*}

\subsubsection*{Convention} We also adopt the convention that for all subsequent statements, unless explicitly mentioned otherwise, the domain is equipped with the standard Gaussian measure over the relevant dimensions. {Also, at several places in the paper, we will establish bounds on one quantity in terms of others without the explicit mention of the dependence. Unless mentioned otherwise, these bounds can be made explicit by working through the proof but we do not do so in the interest of clarity. For example, if we state that a quantity $d= d(k,\epsilon)$, we mean that that $d$ can be bounded as an explicit function of $k$ and $\epsilon$ but we hide the explicit dependence.  }

\section{Reduction from arbitrary functions to PTFs}\label{sec:reduction}
In this section, we will prove that given any $k$-ary function with a given set of measures for each of the $k$-partitions, there is a multivariate PTF with nearly the same measures for the induced partitions which is a multivariate PTF and (up to an error $\epsilon$), no less noise stable at a fixed noise rate $t$.
This is the first step (``from general partitions to PTF'') of the proof
sketch in Section~\ref{sec:theorem}.

We will need the following concentration bound for low-degree polynomials, which may be
proved using the Gaussian hypercontractive inequality. See \cite{Janson:97} for a reference. 
\begin{theorem}\label{thm:hyper}
Let $p: \mathbb{R}^n \rightarrow \mathbb{R}$ be a degree-$d$ polynomial. Then, for any $t>0$, 
$$
\Pr_{x} \big[|p(x) - \mathbf{E}[p(x)]| \ge t \cdot \sqrt{\mathsf{Var}[p]}\big]  \leq d \cdot e^{-t^{2/d}}.
$$
\end{theorem}

For technical reasons, we will also require the PTFs to satisfy a property which we refer to as $(d, \delta)$-\emph{balanced} defined below. 
\begin{definition} A degree-$d$ multivariate PTF $f = \mathsf{PTF}(p^{(1)}, \ldots, p^{(k)})$ is said to be $(d, \delta)$-balanced if each $p^{(i)}$ has variance $1$ and $|\mathbf{E}[p^{(i)}]| \le \log^{d/2} (k \cdot d/\delta)$. 
\end{definition}
We ask the reader to observe that the first condition (namely, $\mathsf{Var}(p^{(i)}) =1$) can be achieved without loss of generality by simply scaling all the polynomials to have variance $1$. This scaling does not change the value of the PTF at any point $x$. While the condition on expectation is non-trivial, the next proposition says that any multivariate PTF can be assumed to be $(d,\delta)$-balanced while only changing the value of the PTF at $\delta$-fraction of places. 
\begin{proposition}\label{prop:balanced}
Let $f : \mathbb{R}^n \rightarrow [k]$ defined as $f = \mathsf{PTF}(p^{(1)}, \ldots, p^{(k)})$. Then, there is a $(d,\delta)$-balanced multivariate PTF $g: \mathbb{R}^n \rightarrow [k]$ defined as $g = \mathsf{PTF}(q^{(1)}, \ldots, q^{(k)})$ such that $\Pr_{x} [g(x) \not = f(x) ] \le \delta$.   Further, $\Pr_{x}[ x \in \mathsf{Collision}(g)] \le \Pr_{x} [x \in \mathsf{Collision}(f)] + \delta$. In fact, the polynomials $\{q^{(i)}\}$ are linear translations of the polynomials $\{p^{(i)}\}$.  
\end{proposition}
\begin{proof}
As we have observed, we can assume without loss of generality that the polynomials $p^{(i)}$ have variance $1$. We define the polynomials $q^{(i)}$ as follows. If $|\mathbf{E}[p^{(i)}] | \leq \log^{d/2}(d \cdot k / \delta)$, then we set $q^{(i)} = p^{(i)}$. Else, let $b_i  = \mathsf{sign} (\mathbf{E}[p^{(i)}])$ and we define 
\[
q^{(i)} = p^{(i)}  -\mathbf{E}[p^{(i)}]  + b_i \cdot \log^{d/2}(d \cdot k / \delta). 
\]
First, note that as claimed, $q^{(i)}$ are indeed affine translations of $p^{(i)}$. Next, note that by Theorem~\ref{thm:hyper}, whenever $q^{(i)} \not = p^{(i)}$,  the sign of $p^{(i)}(x)$ (for $x \sim \gamma_n$) is $\delta/k$-close to being fixed. Applying Theorem~\ref{thm:hyper} again, 
\[
\Pr_{x \sim \gamma_n} \big[ \mathsf{sign}(q^{(i)}(x)) \not = \mathsf{sign}(p^{(i)}(x)) \big] \le \delta/k. 
\]
From this, it immediately follows that $$\Pr_{x } [f(x) \not = g(x)] \le \delta \ \textrm{and} \Pr_{x} [ x \in \mathsf{Collision}(f) ] \le \Pr_{x}[x \in \mathsf{Collision}(g)] + \delta. $$
\end{proof}

The next theorem is the main result of this section namely, that given any function $f: \mathbb{R}^n \rightarrow [k]$, it is possible to obtain a multivariate $(d,\epsilon)$ balanced PTF $g$ (for some explicit $d= O_{\epsilon, k}(1)$) such that the resulting PTF $g$ has nearly the same partition sizes and noise stability. Further, $\mathsf{Collision}(g)$ has small probability.

\begin{theorem}\label{thm:1}
  Let $f: \mathbb{R}^n \rightarrow [k]$ satisfy $\E[f] = \boldsymbol{\mu}$ where $\boldsymbol{\mu} = (\mu_1, \ldots, \mu_k)$. Then, for every
  $\epsilon>0$, there exists a multivariate PTF $g: \mathbb{R}^n \rightarrow
  [k]$ of degree $d = d(k, \epsilon)$ such that
  \begin{itemize}
    \item $\Vert \mathbf{E}[g] - \boldsymbol{\mu} \Vert_1 \le \epsilon$,
    \item $\langle g, P_t g \rangle \ge \langle f, P_t f \rangle - \epsilon$,
    \item $\Pr [x \in \mathsf{Collision} (g)] \le \epsilon$, and
  \item $g$ is $(d, \epsilon)$-balanced.
  \end{itemize}
\end{theorem}

We will prove Theorem~\ref{thm:1} in two parts. The first step is to show that
we can replace $f$ by a function with explicit Hermite decay:
\begin{lemma}\label{lem:1} 
  For every $f: \mathbb{R}^n \rightarrow [k]$ and every $\epsilon>0$, there exists a function
 $h: \mathbb{R}^n \rightarrow [k]$ satisfying the following: 
 \begin{itemize}
   \item $\E[h] = \E[f]$,
 \item $\langle h, P_t h \rangle \ge \langle f, P_t f \rangle - \epsilon$, 
 \item $\mathsf{W}^{> d} [h] \le \epsilon$ for some explicit $d = d(k, \epsilon)$. 
 \end{itemize}
\end{lemma} 

The second step in the proof of Theorem~\ref{thm:1} goes from explicit
Hermite decay to an actual PTF:

\begin{lemma}\label{lem:2}
Let $h : \mathbb{R}^n \rightarrow [k]$ such that $\mathsf{W}^{>d}[h] \le \epsilon$. Then, there is a PTF $g: \mathbb{R}^n \rightarrow [k]$ of degree $d$ such that 
$ \mathbf{E}_{x } [g(x) \not = h(x) ] \le k^2 \cdot  \epsilon$. Further, $\Pr_{x } [x \in \mathsf{Collision} (h)]  \le k^2 \cdot \epsilon$. 
\end{lemma}
It is easy to see that Theorem~\ref{thm:1} follows in a straightforward way by combining Lemma~\ref{lem:1} and Lemma~\ref{lem:2}. 
Note that the condition of $g$ being $(d,\epsilon)$-balanced is obtained by simply applying Proposition~\ref{prop:balanced}.

We begin with the proof of Lemma~\ref{lem:2}, which is easier.

\begin{proofof}{Lemma~\ref{lem:2}}
Let us view $h: \mathbb{R}^n \rightarrow \mathbb{R}^k$ and define $h^{(0)} : \mathbb{R}^n \rightarrow \mathbb{R}^k $ as $h^{(0)} (x) = h(x) -\frac{1}{k} \cdot \mathbf{1}$. Here the function $\mathbf{1}$ is defined as $\mathbf{1} : x \mapsto (1, 1 , \ldots, 1)$ for all $x \in \mathbb{R}^n$. Note that range of  the function $h^{(0)}$ at all points is a $k$-dimensional vector with the entry $(k-1)/k$ in one of the coordinates and $-1/k$ in every other coordinate. Call such a point \emph{a $k$-lattice point}.  

Let $h^{(0)}_{\le d} : \mathbb{R}^n \rightarrow \mathbb{R}^k$ (as defined earlier) correspond to the function obtained by truncating the Hermite expansion of $h^{(0)}$ at degree $d$. Consider the degree-$d$ polynomials $p_1, \ldots, p_k: \mathbb{R}^n \rightarrow \mathbb{R}$ where $p_i$ is the $i^{th}$ coordinate of $h^{(0)}_{\le d}$. We now define $g = \mathsf{PTF}(p_1, \ldots, p_k)$. Consider any point $x \in \mathbb{R}^n$. 
\begin{itemize}
\item If $x \in \mathsf{Collision}(g)$, then observe that $\Vert h^{(0)} (x) - h^{(0)}_{\le d}(x)  \Vert_2^2 \ge k^{-2}$. To see this, note that if $x \in \mathsf{Collision}(g)$, there are two possibilities. (a) $p_i(x) \leq 0$ for all $1 \le i \le k$: It then easily follows that  $\ell_2$ distance of $h^{(0)}_{\le d}(x)$ from any $k$-lattice point is at least $1/k$. (b) $p_i(x) , p_j(x)> 0$ for distinct $i,j$. Again the $\ell_2$ distance of $h^{(0)}_{\le d}(x)$ from any $k$-lattice point is at least $1/k$. Thus, in both these cases, $\Vert h^{(0)} (x) - h^{(0)}_{\le d}(x)\Vert_2^2 \ge \frac{1}{k^2}$. 
\item If $x \not \in \mathsf{Collision}(g)$ and $g(x) \not = h(x)$, then again $\Vert h^{(0)} (x) - h^{(0)}_{\le d}(x)  \Vert_2^2 \ge k^{-2}$. To see this, note that if $g(x) \not = h(x)$, then there exists $1 \le i, j \le k$ and $i \not = j$, such that $p_i(x) >0$, $p_j(x) \leq 0$. On the other hand, $h(x) = \mathbf{e}_j$ and thus the $j^{th}$ coordinate of $h^{(0)}(x) = (k-1)/k$. This also implies that 
$\Vert h^{(0)} (x) - h^{(0)}_{\le d}(x)\Vert_2^2 \ge \frac{(k-1)^2}{k^2}$. 
\end{itemize}
Combining these two, we get 
\[
\mathbf{E}_{x } \big[ \Vert h^{(0)} (x) - h^{(0)}_{\le d}(x)\Vert_2^2 \big] \geq \frac{1}{k^2} \cdot  \big( \Pr_{x } [x \in \mathsf{Collision}(g)] +  \Pr_{x } [x \not \in \mathsf{Collision}(g) \ \wedge \ g(x) \not = h(x) ]\big). 
\]
As $h^{(0)}$ and $h$ differ only in the Hermite coefficient for $S = 0$, hence the left hand side is the same as $ \mathsf{W}^{\ge d}[h]$. Combining this with the above inequality achieves all the stated guarantees. 
\end{proofof}

The proof of Lemma~\ref{lem:1} is longer, and so we begin with an outline.
The first observation is that
bounding $\E[|\nabla h|]$ implies a bound on $\mathsf{W}^{> d}[h]$.
This follows a standard spectral argument, and is stated as Corollary~\ref{corr:gradient}.
The second observation is that the function $h$ obtained by thresholding
$P_t f$ at a suitable value (chosen, for example, so that $\E[h] = \E[f]$)
satisfies $\langle h, P_t h\rangle \ge \langle f, P_t f\rangle$. This is stated as Lemma~\ref{lem:rounding-stability}.

Based on the previous paragraph, it seems like we would like to bound $\E[|\nabla h|]$ where $h$ is obtained
by thresholding $P_t f$; let $a \in \R^k$ be the desired threshold value, so that thresholding $P_t f$ at $a$
produces a partition with the right measures. It turns out, unfortunately, that for $h$ defined in this way,
$\E[|\nabla h|]$ could be
arbitrarily large. A key insight of~\cite{KNOW14} is that (using the co-area formula and gradient bounds on $P_t f$)
the partition produced by thresholding at a random value near $a$ has bounded expected surface area. In particular,
although thresholding exactly at $a$ might be a bad idea,
there exist many good nearby values at which to threshold.
Based on this observation, we construct $h$ in two steps. In the first step, we define $h$ by thresholding $P_t f$,
but only on the set of $x \in \R^n$
for which $P_t f(x)$ is not too close to $a$. By choosing ``not too close'' in a suitable random way, the observation
of~\cite{KNOW14} implies that this step only contributes a bounded amount to $\E[|\nabla h|]$. Since the first step
is almost the same as just thresholding $P_t f$ at $a$, it is consistent with our desire that
$\langle h, P_t h\rangle \ge \langle f, P_t f\rangle - \epsilon$.

In the second step, we partition the remaining part of $\R^n$ by chopping it with halfspaces of the correct size.
Since halfspaces have a bounded surface area, this also contributes a bounded amount to $\E[|\nabla h|]$.
Crucially, this step does not
destroy the value of $\langle h, P_t h\rangle$; fundamentally, this is because $P_t f$ is almost constant on the set
we are partitioning.

\subsection{Surface area and spectrum}

In our outline of Lemma~\ref{lem:2}'s proof, we claimed that a control of $\E[|\nabla h|]$ implies a control of
$\mathsf{W}^{> d}[h]$. Here we prove that claim, using 
a theorem of Ledoux~\cite{Ledoux:94} that gives a lower bound on the gradient in terms of the noise sensitivity. 

\begin{theorem}\label{thm:Ledoux}[Ledoux]
Let $f: \mathbb{R}^n \rightarrow [0,1]$ be of bounded variation. Then, for every $t>0$, 
\[
\mathop{\mathbf{E}}_{x \sim \gamma_n} [f \cdot (f - P_t f)]  \le \frac{\arccos (e^{-t})}{\sqrt{2\pi}} \cdot \mathop{\mathbf{E}}_{x \sim \gamma_n}[|\nabla f|]. 
\]
\end{theorem}
Using this theorem, it is possible to establish an upper bound on $\mathsf{W}^{\ge k}[f]$ in terms of $|\nabla f|$ as done below. 
\begin{corollary}\label{corr:gradient}
Let $f: \mathbb{R}^n \rightarrow [0,1]$ be of bounded variation. Then, 
\[
\mathsf{W}^{\ge d} [f]  \leq O\bigg( \frac{1}{\sqrt{d}}\bigg) \cdot \mathbf{E}[|\nabla f|].
\]
\end{corollary}
\begin{proof}
\begin{eqnarray*}
\mathbf{E}[ f (f- P_t f)] = \sum_{j=0}^{\infty} (1- e^{-t \cdot j}) \mathsf{W}^{j}[f] \ge (1-e^{-d \cdot t}) \cdot  \mathsf{W}^{ \ge d}[f].
\end{eqnarray*}
Choose $d=1/t$ and apply Theorem~\ref{thm:Ledoux}, we get 
\[
\mathsf{W}^{ \ge d}[f] \leq O(1) \cdot \mathbf{E}[ f (f- P_{\frac{1}{d}} f)] \le  \frac{\arccos (e^{-\frac1d})}{\sqrt{2\pi}} \cdot \mathop{\mathbf{E}}_{x \sim \gamma_n}[|\nabla f|] \le O\bigg( \frac{1}{\sqrt{d}}\bigg) \cdot \mathbf{E}[|\nabla f|].
\]
\end{proof}

\subsection{Various Lemmas for thresholding}

Recall that our proof of Lemma~\ref{lem:1} is based on thresholding $P_t f$. Before proving it, we
introduce various properties of the thresholding procedure.
The first lemma states that the ``best'' way to round a $\Delta_k$-valued function to
a $\{e_1, \dots, e_k\}$-valued function while preserving its expectation is simply to threshold it.
Here, ``best'' means that we are trying to maximize the correlation between the original function
and the rounded one.

\begin{lemma}\label{lem:rounding-correlation}
  Take $f: \R^n \to \Delta_k$, $g: \R^n \to \{e_1, \dots, e_k\}$, and $z \in \R^k$.
  Suppose that whenever $g(x) = e_i$, we have
  \[
    i \in \mathrm{argmax}\{f_j(x) - z_j: j = 1, \dots, k\}.
  \]
  Then $g$ maximizes $\E[\langle f, h \rangle]$ among all
  $h: \R^n \to \{e_1, \dots, e_n\}$ satisfying $\E [h] = \E[g]$.
\end{lemma}

\begin{proof}
  Let $h: \R^n \to \{e_1, \dots, e_n\}$ satisfy $\E[h] = \E[g]$. By the defining assumption of $g$,
  $\inr{g}{f-z} \ge \inr{h}{f-z}$ pointwise. Hence,
  \[
    \E[\inr{h}{f}] = \E[\inr{h}{f-z}] + \inr{\E[h]}{z} \ge \E[\inr{g}{f-z}] + \inr{\E[g]}{z} = \E[\inr{g}{f}].
  \]
\end{proof}

Let us assign a notation to the sort of rounding involved in Lemma~\ref{lem:rounding-correlation}.

\begin{definition}
  For $f: \R^n \to \Delta_k$ and $z \in \R^k$, let $T_z(f)$ denote the set of functions $g: \R^n \to \{e_1, \dots, e_k\}$
  such that
  \[
    \langle g(x), f(x) - z \rangle \ge f_i(x) - z_i
  \]
  for every $i = 1, \dots, k$ and every $x \in \R^n$.
\end{definition}

The next step is to show that a function obtained by thresholding $P_t f$ is always at least as noise-stable (with parameter $t$)
as the original function $f$. 

\begin{lemma}\label{lem:rounding-stability}
  Let $f: \R^n \to \{e_1, \dots, e_k\}$ and take $t > 0$. If $g \in T_z(f)$ for some $z \in \R^k$
  satisfies $\E[g] = \E[f]$
  then $\Stab_t(g) \ge \Stab_t(f)$.
\end{lemma}

\begin{proof}
  Since $\E[g] = \E[f]$, Lemma~\ref{lem:rounding-correlation} implies that
  $\E[\langle g, P_t f\rangle] \ge \E[\langle f, P_t f\rangle] = \Stab_t(f)$. On the other hand,
  the Cauchy-Schwarz inequality and the semi-group property of $P_t$ imply that
  \[
    \E[\langle g, P_t f \rangle]
    = \E[\langle P_{t/2} g, P_{t/2} f \rangle]
    \le \sqrt{\Stab_t(g) \Stab_t(f)}.
  \]
  The claim follows.
\end{proof}

In order to make Lemma~\ref{lem:rounding-stability} useful, we need to show that we can always find a rounding
with the same expectation as the original function.

\begin{lemma}\label{lem:rounding-existence}
  For any $f: \R^n \to \Delta_k$, there exists some $z\in \R^k$ and $g \in T_z(f)$ such that $\E[g] = \E[f]$.
\end{lemma}

Since the proof of Lemma~\ref{lem:rounding-existence} is mainly technical, we postpone it to the appendix.

\subsection{Proof of Lemma~\ref{lem:1}}

We introduce three final ingredients before beginning the proof of Lemma~\ref{lem:1}: the first is a gradient bound on $P_t f$
that is due (in a much more precise form) to Bakry and Ledoux~\cite{BakryLedoux:96}.
\begin{theorem}\label{thm:bakry-ledoux}
  If $f: \R^n \to [0, 1]$ then for any $t > 0$, $\|\nabla P_t f\|_\infty \le C/\sqrt{t}$.
\end{theorem}

The second ingredient is the co-area formula in Gaussian space. To state this, let $\gamma_n^{+}$ denote the Gaussian surface area in $n$-dimensions. The co-area formula (see~\cite{federer1969} for a reference) is stated next. For any $f: \mathbb{R}^n \rightarrow \mathbb{R}$ and for any continuous, compactly supported $\mu : \mathbb{R} \rightarrow \mathbb{R}$ 
\[
\int_{\mathbb{R}} \mu(t) \cdot \gamma_n^{+} (\{x: f(x) \ge t\}) dt = \int_{\mathbb{R}^n} \mu(f(x)) |\nabla(f(x))| d \gamma(x).
\]
The co-area formula relates the Gaussian surface of a Boolean function (specified by $\{x: f(x) \ge t\}$) to the gradient of $f$.

The final ingredient is a weighted version of Hall's marriage theorem.
It will be used to show that we can divide a certain amount of ``missing'' volume into pieces of the right size,
while preserving certain constraints.
\begin{theorem}\label{thm:hall}
  Let $G = (U, V, w, E)$ be a finite, vertex-weighted, bipartite graph. That is, $U$ and $V$ are finite sets, $w$ is a weight function $w: U \sqcup V \to (0, \infty)$, and $E \subset U \times V$ is the
  set of edges. Suppose that for every $U' \subset U$, $\sum_{u \in U} w(u) \le \sum_{v \sim U'} w(v)$, where $v \sim U'$ if there exists $u \in U'$ such that $(u, v) \in E$. Then
  there exists $p: V \to \Delta_{|U|}$ such that $p_u(v) > 0$ implies $u \sim v$, and such that
  \[
    \sum_{v \in V} p_u(v) w(v) = w(u)
  \]
  for every $u \in U$.
\end{theorem}
In the case that $w \equiv 1$, Theorem~\ref{thm:hall} follows from the usual formulation of Hall's marriage theorem
(which guarantees in addition that $p: V \to \{e_1, \dots, e_{|U|}\}$). When
$w$ is integer-valued (or, by scaling, rational-valued), it follows by applying Hall's marriage
theorem to the graph in which $u$ is replicated $w(u)$ times. The general case follows by
a simple approximation argument.

\begin{proofof}{Lemma~\ref{lem:1}}
We begin by applying Lemma~\ref{lem:rounding-existence} to $P_t f$: let $z$ and $g \in T_z(P_t f)$
be such that $\E[g] = \E[P_t f] = \E[f]$. According to Lemma~\ref{lem:rounding-stability}, $\Stab_t(g) \ge \Stab_t(f)$.
However, $\E[|\nabla g|]$ cannot be controlled in general; therefore, we will move to an approximation of $g$.

Let $\mu$ be a probability measure on $[0, \epsilon]$ with continuous density bounded by $2/\epsilon$.
By the co-area formula,
\begin{align*}
  &\int_\R \mu(t) \gamma^+(\{x: P_t f_i(x) - z_i > P_t f_j(x) - z_j + t\})\, dt \\
  &= \int_{\R^n} \mu(P_t f_i(x) - z_i - P_t f_j(x) + z_j) |\nabla P_t f_i(x) - \nabla P_t f_j(x)|\, d\gamma(x) \\
  &\le \frac{2}{\epsilon} \int_{\R^n} |\nabla P_t f_i(x) - \nabla P_t f_j(x)|\, d\gamma(x) \\
  &\le \frac{C}{\epsilon \sqrt t},
\end{align*}
where the last inequality follows from Theorem~\ref{thm:bakry-ledoux}.
In particular, there exist some $y_{ij}^+$ and $y_{ij}^-$ in $[0, \epsilon]$ such that if
$A_{ij}^{\pm} = \{x \in \R^n: P_t f_i(x) - z_i > P_t f_j(x) - z_j \pm y_{ij^{\pm}}\}$ then
\begin{equation*}\label{eq:A_ij-surface}
\gamma^+(A_{ij}^\pm) \le C(\epsilon)/\sqrt t.
\end{equation*}
We repeat this construction of $y_{ij}$ and $A_{ij}$ for every ordered pair $(i, j)$.
Define $A_i = \bigcap_{j \ne i} A_{ij}^+$, and note that $A_i \subset \{x: g(x) = e_i\}$.
On the other hand, $A_{ij}^- \supset \{x: g(x) = e_i\}$ for every $i, j$.
Next, define
\begin{align*}
  C_i &= \bigcap_{j \ne i} A_{ij}^- \\
  C_I &= \bigcup_{i \in I} C_i \\
  B_I &= C_I \setminus \bigcup_{J \supsetneq I} C_J \setminus \bigcup_{i \in I} A_i.
\end{align*}
The meaning of these sets is the following: $A_i$ is the set where $f_i - z_i$ significantly larger than
any other
$f_j - z_j$. On $C_i$, $f_i - z_i$ is almost $\max_j f_j - z_j$;
on $C_I$, $f_i - z_i$ is almost maximal for \emph{every} $i \in I$; and on $B_I$, the set
of $i$ for which $f_i - z_i$ is almost maximal is exactly $I$. Importantly, the collection
of all $A_i$ and $B_I$ form a partition of $\R^n$. Our basic strategy will be to set $h$ to be $e_i$
on $A_i$, and then to define $h$ on the remaining part of the space in order to satisfy two properties:
$\E[h] = \E[f]$ and $h(x) = e_i$ only if $x \in B_I$ for some $I \ni i$.

Since the Gaussian surface area obeys the inequalities $\gamma_n^+(A \cap B) \le \gamma_n^+(A) + \gamma_n^+(B)$
and $\gamma_n^+(A \cup B) \le \gamma_n^+(A) + \gamma_n^+(B)$, and since $B_I$ and $A_i$ are defined
using a finite (depending on $k$) number of intersections and unions, it follows that
\begin{equation}\label{eq:A-B-surface}
  \begin{aligned}
    \gamma^+(A_i) &\le C(\epsilon, k)/\sqrt t \\
    \gamma^+(B_I) &\le C(\epsilon, k)/\sqrt t.
  \end{aligned}
\end{equation}

Now, for any $I \subset [k]$,
\[
  \bigcup_{J: J \cap I \ne \emptyset} B_J \cup \bigcup_{i \in I} A_i
\]
contains the set of $x$ for which $g(x) \in \{e_i: i \in I\}$. It follows that
\[
  \sum_{J: J \cap I \ne \emptyset} \gamma_n(B_J)
  \ge \sum_{i \in I} \E [g_i] - \gamma_n(A_i).
\]
Consider the bipartite graph where $U = [k]$ and $V = \{I: I \subset [k]\}$, and $(i, I) \in E$ if $i \in I$.
We assign the weights $w(i) = \E[g_i] - \gamma_n(A_i)$ to $i \in U$ and $w(I) = \gamma_n(B_I)$ for $I \in V$
(note that $w(i) \ge 0$ because the construction of $A_i$ ensures that $A_i \subset \{x: g(x) = e_i)\}$).
The displayed equation above ensures that this weighted graph satisfies the hypothesis of Theorem~\ref{thm:hall},
and so there exists $p: V \to \Delta_k$ with $p_i(I) > 0$ only if $i \in I$, and with
\[
  \sum_{I \ni i} p_i(I) \gamma_n(B_I) = \E[g_i] - \gamma_n(A_i).
\]

Finally, we will use $p$ to define $h$. First, for every $I$ and $i \in I$, let $B_{I,i}$ be a set of the form
$\{x \in \R^n: a \le x_1 \le b\}$ such that $\gamma_n(B_{I,i} \cap B_I) = p_i(I) \gamma_n(B_I)$; moreover, we choose $B_{I,i}$
such that $\gamma_n(B_{I,i} \cap B_{I,j}) = 0$ when $i \ne j$. Then we
set $h(x)$ to equal $e_i$ on the set $A_i \cup \bigcup_{I \ni i} (B_{I,i} \cap B_I)$. By the defining property of $p$,
$h$ satisfies $\E[h_i] = \E[g_i] = \E[f_i]$. Moreover, $h(x) = e_i$ only ona subset of $A_i \cup \bigcup_{I \ni i} B_I$,
and on this set $P_t f_i(x) - z_i \ge \max_j P_t f_j(x) - z_j - \epsilon = \langle g(x), P_t f(x) - z\rangle - \epsilon$.
It follows that
\[
  \E[\langle h, P_t f\rangle] \ge \E[\langle g, P_t f\rangle] - \epsilon \ge \Stab_t(f) - \epsilon.
\]
By the Cauchy-Schwarz inequality,
\[
  \sqrt{\Stab_t(h)} \ge \sqrt{\Stab_t(f)} - \frac{\epsilon}{\sqrt{\Stab_t(f)}},
\]
which implies that $\Stab_t(h) \ge \Stab_t(f) - 2\epsilon$.

Finally, we address the surface aread of $h$.  Recall that
\[
  \{x: h(x) = e_i\} = A_i \cup \bigcup_{I,i} B_{I,i} \cap B_I.
\]
The number of terms on the right is some constant depending on $k$. By~\eqref{eq:A-B-surface},
$\gamma_n^+(A_i)$ and $\gamma_n^+(B_I)$ are bounded by $C(\epsilon, t, k)$. Since $B_{I, i}$
is an intersection of two halfspaces, its Gaussian surface
area is at most a constant. It follows that $\gamma_n^+(\{x: h(x) = e_i\}) \le C(\epsilon, k)/\sqrt t$ for every $i$.
Hence, $\E[|\nabla h|] \le C(\epsilon, k)/\sqrt t$.

After applying Corollary~\ref{corr:gradient}, this completes the proof except that $d = d(k, \epsilon, t)$
instead of the claimed $d(k, \epsilon)$, where $d(k, \epsilon, t)$ blows up as $t \to 0$.
To eliminate this dependence on $t$, it suffices to note that the claim is trivial if $t \ll (\epsilon/k)^2$.
Indeed, one can easily construct $h$ with $\E[h] = \E[f]$, $\Stab_t(h) \ge \E[|h|^2] - O(k \sqrt t)$
and $\E[|\nabla h|] \le O(k)$. For example, we could take the pieces $\{x: h(x) = e_i\}$ to be parallel
slabs of the form $\{x: a \le x_1 \le b\}$, where $a$ and $b$ are chosen so that the slabs have the
correct volumes. If $t \ll (\epsilon/k)^2$ then such an example will satisfy the claim of the lemma.
\end{proofof}

\section{Reduction from PTFs to PTFs on a constant number of variables}\label{sec:junta}

In this section, we are going to prove the following theorem. 
\begin{theorem}\label{thm:reduction}
  Let $f: \mathbb{R}^n \rightarrow [k]$ be a degree-$d$, $(d,\epsilon)$-balanced  PTF with $\E_{x}[f(x)] = \boldsymbol{\mu}$ where $\boldsymbol{\mu} = (\mu_1, \ldots, \mu_k)$.  Further, let us assume that 
$\Pr[x \in \mathsf{Collision} (f)] \le \epsilon/(40 k^2)$.  
  Then, for every $\epsilon>0$, there exists a degree-$d$ 
PTF $\fh: \mathbb{R}^{n_0} \rightarrow [k]$ such that, \begin{itemize} \item  $\Vert \mathbf{E}_{x} [\fh(x)] - \boldsymbol{\mu} \Vert_1 \le \epsilon$, \item $ \langle \fh, P_t \fh \rangle \ge \langle f, P_t f \rangle - \epsilon$. 
\end{itemize} Further, $n_0 = n_0 (d,k,\epsilon)$ is an explicitly defined function.  

\end{theorem}
The above theorem states that given a degree-$d$ multivariate PTF over $n$ variables, there is another multivariate PTF which induces approximately the same partition sizes and has approximately the same noise stability (at any fixed noise rate $t$) but the new PTF is only over some (explicitly defined) $O_{d,t}(1)$ variables.
Our main workhorse for this section is going to be two structural theorems for low-degree polynomials proven in \cite{DS14}. In order to state these theorems, we will need a few definitions from that paper. In particular, we will need to define the relation between polynomials and tensors and then define the notion of an $\epsilon$-\emph{eigenregular} polynomial. {Before we do that, let us observe that Theorem~\ref{thm:main-stab} follows very easily by combining 
Theorem~\ref{thm:reduction} and Theorem~\ref{thm:1}. 

\subsubsection*{Proof of Theorem~\ref{thm:main-stab}} 
Let us us assume that given measure $\boldsymbol{\mu} = (\mu_1, \ldots, \mu_k)$ and noise rate $t>0$, the most noise stable partition is $f: \mathbb{R}^n \rightarrow [k]$. Then, applying Theorem~\ref{thm:1} (with error parameter $\epsilon/(40k^2)$), we obtain a PTF $g_1$ of degree $d = O_{k,\epsilon}(1)$ such that it is $(d,\epsilon/2)$ balanced and $\Pr[x \in \mathsf{Collision}(g_1)]\le \epsilon/(40k^2)$. Further, note that $\Vert \mathbf{E}[g_1] - \mathbf{E}[f] \Vert_1 \le \epsilon/2$ and $\langle g_1, P_t g_1 \rangle \ge \langle f, P_t f \rangle - \epsilon/2$. 

We next apply Theorem~\ref{thm:reduction} on this function $g_1$ to obtain $\fh$ which is a degree-$d$ PTF on $n_0 = O_{d,\epsilon,k}(1)$ variables such that $\Vert \mathbf{E}[g_1] - \mathbf{E}[\fh] \Vert_1 \le \epsilon/2$ and 
$\langle \fh, P_t \fh \rangle \ge \langle g_1, P_t g_1 \rangle - \epsilon/2$. 

Combining these two facts, we obtain $\Vert \mathbf{E}[g_1] - \mathbf{E}[\fh] \Vert_1 \le \epsilon$ and 
$\langle \fh, P_t \fh \rangle \ge \langle f, P_t f \rangle - \epsilon$. Setting $g = \fh$ concludes the proof.

} 
\subsection*{Connection between polynomials and tensors}
We give a brief description of the connection between symmetric tensors and polynomials under the Gaussian distribution. The interested reader may consult the book~\cite{Janson:97} for a detailed background. 
Let $\mathcal{H}$ denote the Hilbert space $\mathbb{R}^n$ and let $\mathcal{H}^{\otimes q}$ be used to denote the $q$-ary tensor product of $\mathcal{H}$. An element $f  \in \mathcal{H}^{\otimes q}$ is said to be symmetric if its invariant under any permutation $\sigma: [q] \rightarrow [q]$. Let $\mathcal{H}^{\odot q}$ denote the symmetric subspace of $\mathcal{H}^{\otimes q}$. 
A tensor $f \in \mathcal{H}^{\otimes q}$ is multilinear if $f(i_1, \ldots, i_q)=0$ for all diagonal elements $(i_1, \ldots, i_q)$ i.e. whenever there exists $1 \le j < \ell \le q$ such that $i_j = i_\ell$.  We now describe a map between the space $\mathcal{H}^{\odot q}$ and polynomials. Recall that $H_q$ denotes
a Hermite polynomial.
\begin{definition}
The iterated Ito integral $I_q$ maps $\mathcal{H}^{\odot q}$ as follows: Let $h \in \mathcal{H}$ be a unit vector and note that  $h^{\otimes q} \in \mathcal{H}^{\odot q}$. Then, 
$I_q(h^{\otimes q}) = H_q( \langle h , x \rangle)$ where $x \in \mathbb{R}^n$. The map $I_q$ is extended linearly to $\mathcal{H}^{\odot q}$. 
\end{definition}
For the convenience of the reader, here we describe a basis of the space $\mathcal{H}^{\odot q}$. Consider an unordered multiset $S = \{s_1, \ldots, s_q\} \subseteq [n]$ of size $q$. Define $\Phi_S \in \mathcal{H}^{\odot q}$ is defined as
$$
\Phi_S = \sum_{\sigma \in \mathsf{Sym}_q} e_{s_{\sigma(1)}} \otimes \ldots \otimes e_{s_{\sigma(q)}}.
$$
The following construction of basis for $\mathcal{H}^{\odot q}$ is obvious and is stated without proof.
\begin{proposition}\label{prop:basis}
Let $\mathcal{S}$ be the set of unordered multisets of $[n]$ of size $q$. Then, $\{\Phi_S \}_{S \in \mathcal{S}}$ forms an orthogonal basis of $\mathcal{H}^{\odot q}$. An important feature of this basis is that any element in the basis is dependent on at most $q$ of the standard basis elements of $\mathbb{R}^n$. 
\end{proposition}
A fundamental property of the map $I_q$ is that it is an isometry between the space of symmetric tensors and polynomials endowed with the standard normal measure (see~\cite{Janson:97} for a proof). 
\begin{proposition}\label{prop:4}
$\mathop{\mathbf{E}}_{x \sim \gamma_n} [I_q(f) \cdot I_p(g)]  = \delta_{p=q} \cdot \langle f,  g \rangle.$
\end{proposition} 
We will  refer to the range of $I_q$ as the Wiener chaos $\mathcal{W}^{q}$. Based on Proposition~\ref{prop:4}, $I_q$ is a bijective map from $H^{\odot q}$ to $\mathcal{W}^q$. 
On the other hand, it is easy to show that any polynomial $p: \mathbb{R}^n \rightarrow \mathbb{R}$ of degree at most $d$ can be expressed as 
\begin{equation}\label{eq:tensor-decomp}
p = \sum_{q=0}^d I_q(f_q) \ \ \textrm{where}  \ \ f_q \in \mathcal{H}^{\odot q}. 
\end{equation}
{While the above theory  does not really require us to focus on the class of multilinear polynomials i.e. every monomial has degree at most $1$ in any variable, 
many of the results we will use from \cite{DS14} are stated in that paper just for multilinear polynomials. So, at some places, we restrict our attention to just multilinear polynomials. multilinear polynomials i.e. every monomial has degree at most $1$ in any variable.
 It is easy to show that if $f \in \mathcal{H}^{\odot q}$ is a multilinear tensor, then $I_q(f)$  is a multilinear polynomial. Conversely, if $p: \mathbb{R}^n \rightarrow \mathbb{R}$ is a multilinear polynomial, 
 then the tensors $\{f_q\}_{0 \le q\le d}$ appearing in the decomposition of $p$ in (\ref{eq:tensor-decomp}) are all multilinear. }
~\\
\textbf{Ito multiplication formula: } We now state the formula for product of two polynomials $I_p(f)$ and $I_q(g)$ in terms of $f$ and $g$. For a reference, see Nourdin's  survey~\cite{Nourdin2013}. To state the formula, for $r \le p \wedge q$,  let us define the contraction product of two tensors $f$ and $g$ (denoted by $f \ {\otimes}_r \ g \in \mathcal{H}^{p+q-2r}$),
\[
f \ {\otimes}_r \ g (t_1, \ldots, t_{p+q-2r})= \sum_{1 \le z_1, \ldots, z_r \le n} f(t_1, \ldots, t_{p-r}, z_{1}, \ldots, z_r) \cdot g(t_{p-r+1}, \ldots, t_{p+q-r}, z_{1}, \ldots, z_r).
\]
The symmetrized contraction product, denoted by $f \ \widetilde{{\otimes}_r} \ g \in \mathcal{H}^{p+q-2r}$ is defined as the symmetrization of the tensor $f \ {\otimes}_r \ g$.  \begin{proposition}[Ito's multiplication formula]\label{prop:Ito}
\[
I_p(f) \cdot I_q(g) = \sum_{r=0}^{p \wedge q} r! \cdot \binom{p}{r} \cdot \binom{q}{r}  \cdot \frac{\sqrt{(p+q-2r)!}}{\sqrt{p! \cdot q!}}\cdot I_{p+q-2r} (f \ \widetilde{{\otimes}_r} \ g). 
\]
\end{proposition}
We now list a basic fact about contraction products of tensors which shall be helpful later. To state these facts, observe that for $f \in \mathcal{H}^{\otimes i}$ and $g \in \mathcal{H}^{\otimes j}$,  the contraction product 
$f \otimes_r g$ can be viewed as a matrix multiplication between matrices $\mathcal{M}_f$ and $\mathcal{M}_g$ where $\mathcal{M}_f$ (resp. $\mathcal{M}_g$) is the matrix obtained by representing $f \in \mathbb{R}^{[n]^{i-r} \times [n]^{r}}$ (resp.  representing $g \in \mathbb{R}^{[n]^{r} \times [n]^{j-r}}$). 
The following trivial fact follows immediately. 
\begin{fact}~\label{fact:contraction}
For $f \in \mathcal{H}^{\otimes n_1}$ and $g \in \mathcal{H}^{\otimes n_2}$ and $0 \le r \le n_1 \wedge n_2$, $\Vert f \otimes_r g \Vert_F \le \Vert f \Vert_F \cdot \Vert g \Vert_F$. Consequently, $\Vert f \widetilde{\otimes}_r g \Vert_F \le \Vert f \Vert_F \cdot \Vert g \Vert_F$. 
\end{fact}

%
%

Before we proceed further, we will make a minor simplifying assumption, namely that the polynomials involved in defining $f$ in Theorem~\ref{thm:reduction} can be assumed to be multilinear. As we have said earlier, this is because the results in \cite{DS14} are stated for multilinear polynomials (though it should be easily possible to carry it over to non-multilinear polynomials). In order to show this, we will use the following simple lemma. 
\begin{lemma}\label{lem:multilinear} 
Let $f: \mathbb{R}^n \rightarrow [k]$ be a degree-$d$, $(d,\epsilon)$-balanced PTF. Then, for any $\epsilon>0$, there exists a degree-$d$, $(d,2\epsilon)$-balanced multilinear PTF $f_{\mathsf{multi}}: \mathbb{R}^\ell \rightarrow \mathbb{R}$ such that 
\begin{itemize}
\item $\Vert \mathbf{E}_{x } [f(x) ] - \mathbf{E}_{x } [f_{\mathsf{multi}}(x)  ] \Vert_1\le \epsilon$, 
\item $| \mathbf{E}_{x } [ \langle f, P_t  f \rangle ] - \mathbf{E}_{x } [\langle f_{\mathsf{multi}}, P_t f_{\mathsf{multi}}\rangle ] | \le \epsilon$.
\item $\Pr_{x } [x \in \mathsf{Collision}(f_{\mathsf{multi}})] \le \Pr_{x } [x \in \mathsf{Collision}(f)]  + \epsilon$.
\end{itemize}
Here $\ell = n \cdot (k^2/d \epsilon)^{3d} \cdot d^2$. 
\end{lemma}
~\\ We defer the proof of this lemma to Appendix~\ref{app:multilinear}. However, note that by applying Lemma~\ref{lem:multilinear}, we can pretend that the PTF $f$ in Theorem~\ref{thm:reduction} is multilinear. 
Next, we introduce the notion of eigenregularity of tensors and the associated polynomials. 
\subsubsection*{Eigenvalues of tensors and polynomials}

We will now define the notion of eigenvalues of tensors and the corresponding polynomials. 
\begin{definition}
Let $f \in \mathcal{H}^{\odot d}$ and $d \ge 2$. Consider a partition of $[d]$ into $S, \overline{S}$. Then, 
$$
\lambda_{S, \overline{S}}(f) = \max_{g \in \mathcal{H}^{S}, h \in \mathcal{H}^{\overline{S}}} \frac{\langle f, g \otimes h \rangle}{\Vert g \Vert_F \cdot \Vert h \Vert_F},
$$
where $\Vert g \Vert_F$ denotes the Frobenius norm of $g$. We define $\lambda_{\max}(f)$ as
$$
\lambda_{\max}(f) = \max_{S: 0< |S|< d} \lambda_{S, \overline{S}} (f). 
$$
\end{definition}

\begin{definition}\label{def:eigen-tensor}
Let $p: \mathbb{R}^n \rightarrow \mathbb{R}$ be a multilinear polynomial of degree $d>0$  and let $(A_0, \ldots, A_d) \in \mathcal{H}^{\odot 0} \times  \ldots \times \mathcal{H}^{\odot d}$ be the tensor associated with $p$. Then, 
$$
\lambda_{\max}(p)  = \max_{1 < i \le d} \lambda_{\max}(A_i). 
$$
Further, we say that $p$ is $\delta$-eigenregular if $\frac{\lambda_{\max}(p)}{\sqrt{\Var(p)}} \le \delta$. 
\end{definition}
Note that we are not considering $\lambda_{\max}(A_1)$ in the definition of eigenregularity. 
From the fact that definition of eigenregularity is invariant under unitary transformation of variables, we have the following important fact.  
\begin{fact}
If $p(x)$ is $\delta$-eigenregular, then for any $\rho>0$,  $p(\rho \cdot x + \sqrt{1-\rho^2} y)$ is also $\delta$-eigenregular. 
\end{fact}
The next fact states that contraction product with eigenregular tensor is significantly contractive. 
\begin{fact}~\label{fact:contraction1}
Let $f \in \mathcal{H}^{\odot d}$  such that  $\lambda_{\max} (f) \le \kappa$. Let $g \in \mathcal{H}^{\otimes d'}$ and $0< r \le \min \{ d-1,d' \}$.  Then, 
$\Vert f \otimes_r g \Vert_F \le \kappa$. 
\end{fact}
\begin{proof}
View $f$ (resp. $g$) as the matrix $\mathcal{M}_f$ (resp. $\mathcal{M}_g$) such that $\mathcal{M}_f \in \mathbb{R}^{[n]^{d-r} \times [n]^{r}}$ (resp. $\mathcal{M}_g \in \mathbb{R}^{[n]^{r} \times [n]^{d'-r}}$). Then, $f \otimes_r g = \mathcal{M}_f \cdot \mathcal{M}_g$. Note that $\lambda_{\max}(f) \le \kappa$ directly translates to $\Vert \mathcal{M}_f \Vert_{op} \le \kappa$. For $\mathcal{S} \in [n]^{d'-r}$, let $\mathcal{M}_{g, \mathcal{S}}$ denote the corresponding column of $\mathcal{M}_{g}$. Then the corresponding column of $\mathcal{M}_f \cdot \mathcal{M}_g$ is given by $\mathcal{M}_f \cdot \mathcal{M}_{g, \mathcal{S}}$. Thus, 
\[
\Vert \mathcal{M}_f \cdot \mathcal{M}_g \Vert_F^2  = \sum_{\mathcal{S} \in [n]^{d'-r}} \Vert \mathcal{M}_f \cdot  \mathcal{M}_{g, \mathcal{S}} \Vert_F^2 \le \sum_{\mathcal{S} \in [n]^{d'-r}} \kappa^2 \cdot \Vert \mathcal{M}_{g, \mathcal{S}} \Vert_F^2 = \kappa^2 \cdot \Vert \mathcal{M}_g \Vert_F^2.
\]
This finishes the proof. 
\end{proof}
We next have the following easy proposition which says the sum of copies of the same polynomial over disjoint variables is eigenregular. 
\begin{fact}\label{fact:eigen-2}
Let $p: \mathbb{R}^{n} \rightarrow \mathbb{R}$ and let $q : \mathbb{R}^{n \cdot k} \rightarrow \mathbb{R}$ defined as 
$$
q(X_1, \ldots, X_k) = \frac{1}{\sqrt{k}} \cdot \big( p(X_1) + \ldots + p(X_k)), 
$$
where each $X_i$ represents a disjoint block of $n$ variables. Then $q$ is $\frac{1}{\sqrt{k}}$-eigenregular. 
\end{fact}
\begin{proof}
Let $d$ be the degree of $q$ and let  $q(X_1, \ldots, X_k) =  \sum_{j=0}^{d}I_j(f_j)$. Let $\mathcal{H} = \mathbb{R}^{n \cdot k}$ and let $\mathcal{H} = \mathcal{H}_1 \otimes \ldots \otimes \mathcal{H}_k$ where $\mathcal{H}_i = \mathbb{R}^n$ corresponds to the coordinates in $X_i$. Note that $f_j$ has a block diagonal structure on $\mathcal{H}_1 \otimes \ldots \otimes \mathcal{H}_k$ where in fact, the same block is repeated. For $1 \le i \leq k$, let $f_j^{(i)}$ denote the $i^{th}$ block of $f_j$ and likewise for any other tensor.  For any non-trivial partition $S, \overline{S}$ of $[j]$, we have
\begin{eqnarray*}
\lambda_{S, \overline{S}}(f_j) &=&  \max_{g \in \mathcal{H}^{\odot S}, h \in \mathcal{H}^{\odot \overline{S}}} \frac{\langle f_j , g \otimes h \rangle}{\Vert g \Vert_F \cdot \Vert h \Vert_F} \\
&=& \max_{g \in \mathcal{H}^{\odot S}, h \in \mathcal{H}^{\odot \overline{S}}} \frac{\sum_{i=1}^k \langle f_j^{(i)}, g^{(i)} \otimes h^{(i)} \rangle }{\Vert g \Vert_F \cdot \Vert h \Vert_F} \\
&\ge& \max_{g \in \mathcal{H}^{\odot S}, h \in \mathcal{H}^{\odot \overline{S}}} \frac{\sum_{i=1}^k \langle f_j^{(i)}, g^{(i)} \otimes h^{(i)} \rangle }{ \sqrt{\sum_{i=1}^k \Vert g^{(i)} \Vert_F^2} \cdot  \sqrt{\sum_{i=1}^k \Vert h^{(i)} \Vert_F^2} }
\end{eqnarray*}
Let us denote $f_j^{(0)}$ as the common value of $f_j^{(i)}$ for $1 \le i \le k$ and note that $\Vert f_j^{(0)} \Vert_F = 1/\sqrt{k}$. The right hand side is then the same as  
\begin{eqnarray*}
 &&\max_{g \in \mathcal{H}^{\odot S}, h \in \mathcal{H}^{\odot \overline{S}}} \frac{\sum_{i=1}^k \langle f_j^{(0)}, g^{(i)} \otimes h^{(i)} \rangle }{ \sqrt{\sum_{i=1}^k \Vert g^{(i)} \Vert_F^2} \cdot  \sqrt{\sum_{i=1}^k \Vert h^{(i)} \Vert_F^2} } \\
 &\le& \max_{g \in \mathcal{H}^{\odot S}, h \in \mathcal{H}^{\odot \overline{S}}} \frac{\sum_{i=1}^k \frac{1}{\sqrt{k}} \Vert g^{(i)} \otimes h^{(i)} \Vert_F }{ \sqrt{\sum_{i=1}^k \Vert g^{(i)} \Vert_F^2} \cdot  \sqrt{\sum_{i=1}^k \Vert h^{(i)} \Vert_F^2} }\\
 &=& \max_{g \in \mathcal{H}^{\odot S}, h \in \mathcal{H}^{\odot \overline{S}}} \frac{\sum_{i=1}^k \frac{1}{\sqrt{k}} \Vert g^{(i)} \Vert_F \cdot \Vert  h^{(i)} \Vert_F }{ \sqrt{\sum_{i=1}^k \Vert g^{(i)} \Vert_F^2} \cdot  \sqrt{\sum_{i=1}^k \Vert h^{(i)} \Vert_F^2} } \le \frac{1}{\sqrt{k}}. 
\end{eqnarray*}
Both the first and last inequalities are simple applications of the Cauchy-Schwartz inequality and the equality uses the fact that $\Vert A \otimes B \Vert_F = \Vert A \Vert_F \cdot \Vert B \Vert_F$. This implies that for all $j>1$, $\lambda_{\max}(f_j) \le \frac{1}{\sqrt{k}}$. As $\mathsf{Var}(q)=1$, we get that $\lambda_{\max}(q) \le 1/\sqrt{k}$. 
\end{proof}

~\\Having defined the notion of eigenregularity, we now recall the main results from \cite{DS14}.~\\
\textbf{Central limit theorem for Gaussian polynomials:} 
\begin{theorem}\label{thm:CLT}
Let $d >1$ and let $p_1, \ldots, p_t: \mathbb{R}^n \rightarrow \mathbb{R}$ be $t$ degree-$d$ polynomials such that for $1\le i \le t$, $\mathbf{E}[p_i]=0$ and $\mathsf{Var}(p_i) \le 1$ and each $p_i$ is $\epsilon$-eigenregular. Let $F = (p_1(x), \ldots, p_t(x))$ where $x \sim \gamma_n$ and let $C$ denote the covariance matrix of $F$ i.e. $C[i,j] = \mathbf{E}_{x \sim \gamma_n} [p_i(x) \cdot p_j(x)]$.  Let $Z= (Z_1, \ldots, Z_t)$ be a $t$-dimensional Gaussian with mean zero and covariance $C$. 
Then, for any $\alpha:\mathbb{R}^t \rightarrow \mathbb{R}$ such that $\alpha \in \mathcal{C}^2$, 
$$
\big| \mathbf{E}_{x \sim \gamma_n} \big[ \alpha(p_1(x), \ldots, p_t(x)) \big] - \mathbf{E}_{Z} \big[ \alpha(Z_1, \ldots, Z_t) \big] \big|  \le 2^{O( d \log d)} \cdot t^2 \cdot \sqrt{\epsilon} \cdot \Vert \alpha '' \Vert_{\infty}. 
$$

\end{theorem}

The next result we will need from \cite{DS14} is the regularity lemma for low-degree polynomials. However, as the statement of the regularity lemma is quite cumbersome, we will first state the following definition. Again, the definition is quite cumbersome but the authors hope that stating this definition first will make the subsequent proof more modular. 
\begin{definition}
Let $\beta:  [1,\infty) \to (0,1)$ that satisfies $\beta(x) \leq 1/x$. Let us also be given $k$ lists of $d+1$ multlinear polynomials;  the
$s$-th list is $\tilde{p}_{s,0},\dots,\tilde{p}_{s,d}$ where $\tilde{p}_{s,q} \in {\cal W}^q$ and
\grade{
$\Var[\tilde{p}_{s,q}]=1$ for $1 \leq q \leq d.$} Then a $(\beta, M,N)$ decomposition of this list is given by a collection of polynomials 
$\Outer(p_{s,q})$
and a collection of polynomials that we denote
$\{\Inner(p_{s,q})_\ell\}_{\ell = 1,\dots, \num({s,q})}$; here $\num({s,q})$ is the number of arguments of the
polynomial $\Outer(p_{s,q}).$  (``$\Outer$'' stands for ``outer''
and ``$\Inner$'' stands for ``inner''.)

For $s=1,\dots,k$, $\blue{0} \leq q \leq d$ and $x \in \R^n$,
$\tilde{p}_{s,q}$ admits the following decomposition: 
\begin{equation} \label{eq:tildep}
\tilde{p}_{s,q}(x)=
\Outer(p_{s,q})\left(
\Inner(p_{s,q})_1(x),\dots, \Inner(p_{s,q})_{\num({s,q})}(x)\right)
\end{equation}
(Intuitively, each $\tilde{p}_{s,q}$ is a polynomial that has been
decomposed into constituent sub-polynomials:\\ 
$\Inner(p_{s,q})_1,\dots, \Inner(p_{s,q})_{\num(s,q)}$;
The following conditions make this precise.)
The following conditions hold:

\begin{enumerate}

\item For each $s \in [k], \blue{0} \leq q \leq d$ the polynomial
$\tilde{p}_{s,q}(x)$ belongs to the $q$-th Wiener chaos ${\cal W}^q$. Additionally, each polynomial $\Inner(p_{s,q})_\ell$ \blue{with $q \geq 1$}
lies in ${\cal W}^j$ for some $1 \leq j \leq q$ and
has $\Var[\Inner(p_{s,q})_\ell] = 1.$

\item
Each polynomial $\Outer(p_{s,q})$ is a multilinear polynomial in its
$\num({s,q})$ arguments.  Moreover, if $\mathrm{Coeff}(p_{s,q})$ denotes the sum of the absolute values of
the coefficients of $\Outer(p_{s,q})$, then $\sum_{s,q} \mathrm{Coeff}(p_{s,q}) \le M$ and
$\sum_{s,q} \num({s,q}) \le N$. The degree of  the polynomial $\Outer(p_{s,q})$ is at most $d$. 

\item Further, let $\mathrm{Num} = \sum_{s=1}^k \sum_{q=0}^{d}
\num({s,q})$  and $\mathrm{Coeff} = \sum_{s=1}^k \sum_{q=0}^{d}
\mathrm{Coeff}(p_{s,q})$.
\ignore{be the total number of
polynomials $\Inner(p_{s,q})_\ell$ which are arguments to all the
$\Outer(p_{s,q})$
polynomials.  }Then,
each polynomial $\Inner(p_{s,q})_\ell$ is $\beta(\mathrm{Num}+\mathrm{Coeff})$-eigenregular.

\end{enumerate}

\end{definition}
Intuitively, the above definition states the following: Given a decreasing function $\beta(\cdot)$ and a list of $k$ multilinear polynomials of degree $d$, a $(\beta, M,N)$ decomposition of expresses the polynomials as a ``outer polynomial" composed with ``inner polynomials". The quality of this decomposition is captured by the following parameters: 
\begin{itemize}
\item The sum of the arities (denoted by $\mathrm{Num}$) and the sum of absolute values of coefficients (denoted by $\mathrm{Coeff}$) of the outer polynomials. 
\item The eigenregularity of the inner polynomials captured by the function $\beta(\cdot)$. 
\end{itemize}
A $(\beta, M, N)$ decomposition ensures that  $\mathrm{Num} \le N$, $\mathrm{Coeff} \le M$ and that each of the inner polynomials is $\beta(\mathrm{Num} + \mathrm{Coeff})$-eigenregular. 
The main decomposition theorem of \cite{DS14} is that for any decreasing function $\beta(\cdot)$ and any given $k$ lists of $d+1$ multilinear polynomials $\{p_{s,q}\}_{1 \le s \le k , 0 \le q \le d}$, there is an approximate $(\beta, M, N)$ decomposition.
~\\
\textbf{Regularity lemma for collections of low-degree polynomials:}
\begin{theorem}~\label{thm:decomp} 
Fix $d \geq 2$ and fix any non-increasing computable
function $\beta:  [1,\infty) \to (0,1)$ that satisfies $\beta(x) \leq 1/x$.
There is a procedure {\bf \blue{MultiRegularize-Many-Wieners}}$_{d,\beta}$
with the following properties.
The procedure takes as input the following:

\begin{itemize}

\item It is given $k$ lists of $d+1$ multilinear Gaussian polynomials; the
$s$-th list is $p_{s,0},\dots,p_{s,d}$ where $p_{s,q} \in {\cal W}^q$ and
\grade{
$\Var[p_{s,q}]=1$ for $1 \leq q \leq d.$}

\item It also takes as input
a parameter $\tau>0$.

\end{itemize}
For explicitly defined $M_{\beta}(k,d,\tau)$ and $N_{\beta}(k,d,\tau)$, let us define $M = M_\beta(k,d,\tau)$ and $N = N_\beta(k,d, \tau)$. The output of the procedure are $k$ lists of $d+1$ multilinear polynomials; the $s$-th list is $\tilde{p}_{s,0},\dots,\tilde{p}_{s,d}$ where $\tilde{p}_{s,q} \in {\cal W}^q$ and $(\beta, M, N)$ decomposition for the $k$ lists, $\tilde{p}_{s,0},\dots,\tilde{p}_{s,d}$ (for $1 \le s \le k$). Here $p_{s,0} = \tilde{p}_{s,0}$ for $1 \le s \le k$. Further, for $1 \le s \le k$ and $1 \le q \le d$, we have 
\[
\Var[p_{s,q}-\tilde{p}_{s,q}] \leq \tau.
\]

\end{theorem}
The above theorem essentially states that given any non-increasing function $\beta(\cdot)$ and error parameter $\tau$, there are (explicitly defined) quantities $M = M_\beta(k,d,\tau)$ and $N = N_\beta(k,d, \tau)$ such that the procedure outputs a $(\beta, M,N)$ decomposition of the polynomials $\{p_{s,q}\}$. 

\subsubsection{Proof of Theorem~\ref{thm:reduction}}
Let us assume that $f= \mathsf{PTF}(p_1, \ldots, p_k)$ 
where $f$ is the PTF appearing in the statement of Theorem~\ref{thm:reduction}
After applying Lemma~\ref{lem:multilinear}, we can pretend that the PTF $f =  \mathsf{PTF}(p_1, \ldots, p_k)$ in Theorem~\ref{thm:reduction} is multilinear. 
After scaling the polynomials $p_1, \ldots, p_k$,  we can assume that $\mathsf{Var}(p_s)=1$ for all $1 \le s \le k$. With this, we note that there are constants $c_{s,q}$ for $1 \le s \le k$ and $1 \le q \le d$ such that 
\[
p_{s} = p_{s,0} + \sum_{q=1}^d c_{s,q} p_{s,q}, 
\]
where $p_{s,q} \in \mathcal{W}^q$, $\mathsf{Var}(p_{s,q}) =1$ for all $s \in [k]$, $1 \le q \le d$ and $\sum_{q=1}^d c_{s,q}^2=1$. 

\subsubsection*{Anti-concentration for low-degree polynomials}
The following theorem is due to Carbery and Wright. 
\begin{theorem}\cite{CW:01}\label{thm:CW}
Let $p: \mathbb{R}^n \rightarrow \mathbb{R}$ be a degree-$d$ polynomial.  Then, for $\epsilon>0$, 
$$
\sup_{\theta \in \mathbb{R}} \Pr_{x} [|p(x) -\theta| \le \epsilon \cdot \sqrt{\mathsf{Var}[p]}] = O(d \cdot \epsilon^{1/d}). 
$$
\end{theorem}
An immediate consequence is that $\Pr_{x} [|p(x)| > d^{-O(d)}] \ge \frac 12$ for a degree-$d$
polynomial $p$ with $\mathsf{Var}[p] = 1$. With this observation and Theorem~\ref{thm:hyper},
it is easy to deduce the following bound. See Lemma~5 in \cite{DS14}. 
\begin{theorem}\label{thm:combine-hyper}
Let $a, b: \mathbb{R}^n \rightarrow \mathbb{R}$ be degree $d$ polynomials. Further, $\mathbf{E}_x [a(x) - b(x)]=0$ and $\mathsf{Var}[a-b] \le (\tau/d)^{3d} \cdot \mathsf{Var}[a]$. Then, $\Pr_{x}[\mathsf{sign}(a(x)) \not = \mathsf{sign}(b(x))] = O(\tau)$. 
\end{theorem}
 
We will now prove Theorem~\ref{thm:reduction}. To do this, there are several steps involved. For the moment, let $\beta : \mathbb{N} \rightarrow [0,1)$ be a sufficiently fast decreasing function. We will choose the precise function $\beta(\cdot)$ later. As a first step, we run the procedure {\bf \blue{MultiRegularize-Many-Wieners}}$_{d,\beta}$ on the list $\{p_{s,q}\}_{1 \le s \le k , 0 \le q \le d}$ with $\tau = (\epsilon/ (kd))^{3d}$. Let us assume that the output of the procedure are polynomials $\tilde{p}_{s,q}$ (for $1 \le s \le k, 1 \le q \le d$) where 
\[
\tilde{p}_{s,q} = \Outer(p_{s,q})\left(
\Inner(p_{s,q})_1(x),\dots, \Inner(p_{s,q})_{\num({s,q})}(x)\right).
\]
Let us define the polynomials $\tilde{p}_s$ (for $1 \leq s \le k$)  as $\tilde{p}_s = \tilde{p}_{s,0} +  \sum_{q=1}^d c_{s,q} \tilde{p}_{s,q}$. Note that $\tilde{p}_{s,0} = p_{s,0}$. Further, 
\[
\mathsf{Var}(p_s - \tilde{p}_s) = \sum_{q=1}^d c_{s,q}^2 \mathsf{Var}(p_{s,q} - \tilde{p}_{s,q})  \le \tau. 
\]
Define the PTF $\tilde{f} = \mathsf{PTF}(\tilde{p}_1, \ldots, \tilde{p}_k)$. Since $\mathsf{Var}(\tilde{p}_s - p_s) \le \tau$ and $\mathbf{E}[\tilde{p}_s-  p_s] =0$, applying Theorem~\ref{thm:combine-hyper} and a union bound, 
\begin{eqnarray}
\Vert \mathbf{E}_{x \sim \gamma_n} [\tilde{f}(x) ] - \mathbf{E}_{x \sim \gamma_\ell} [f_{}(x)  ] \Vert&\le&\epsilon \nonumber,  \\
\big| \mathbf{E}_{x \sim \gamma_n} [ \langle \tilde{f}, P_t  \tilde{f} \rangle ] - \mathbf{E}_{x \sim \gamma_{\ell}} [\langle f_{}, P_t f_{}\rangle ] \big| &\le& \epsilon,  \nonumber \\ 
\Pr_{x \sim \gamma_{\ell}} [x \in \mathsf{Collision}(\tilde{f})] \le \Pr_{x \sim \gamma_{n}} [x \in \mathsf{Collision}(f_{})] &+& \epsilon.
\end{eqnarray}
Thus, from now on, we can work with the function $\tilde{f} = \mathsf{PTF}(\tilde{p}_1, \ldots, \tilde{p}_k)$. Recall that 
$$
\tilde{p}_{s,q} = \Outer(p_{s,q})\left(
\Inner(p_{s,q})_1(x),\dots, \Inner(p_{s,q})_{\num({s,q})}(x)\right).
$$
For the subsequent discussion, let us define $\mathcal{P} = \{\Inner(p_{s,q})_1(x),\dots, \Inner(p_{s,q})_{\num({s,q})}(x)\}_{1 \le s \le k, 1 \le q \le d}$. 
Next, we consider another family of polynomials $\mathcal{R} = \{\Inner(r_{s,q})_1(x),\dots, \Inner(r_{s,q})_{\num({s,q})}(x)\}_{1 \le s \le k, 1 \le q \le d}$. We will require two properties of the family $\mathcal{R}$. For $1 \le s \le k$, $1 \le q \le d$ and $1 \le \ell \le \mathrm{num}(s,q)$,  $\Inner(r_{s,q})_\ell : \mathbb{R}^{n_0} \rightarrow \mathbb{R}$ (where we will ensure that $n_0 = O_{d,k,\beta,\epsilon}(1)$.). 
\begin{cond}\label{cond:junta}
\begin{itemize}
\item For every $1 \le s \le k$, $1 \le q \le d$ and  $ 1 \le \ell \leq \num(s,q)$, $\Inner(r_{s,q})_\ell \in \mathcal{W}^j$ if and only if $\Inner(p_{s,q})_\ell \in \mathcal{W}^j$. In other words, 
$\Inner(r_{s,q})_\ell$ and $\Inner(p_{s,q})_\ell$ belong to the same level of Wiener chaos. 
\item The covariance of the family $\mathcal{R}$ is same as that of the family $ \{\Inner(p_{s,q})_1(x),\dots, \Inner(p_{s,q})_{\num({s,q})}(x)\}_{1 \le s \le k, 1 \le q \le d}$. In other words, 
for $1 \le s_1, s_2 \le k$, $1 \le q_1, q_2 \le d$ and $1 \le \ell_i \leq \num(s_i, q_i)$ $($for $i \in \{1,2\})$, 
\begin{equation}
\mathbf{E}[\Inner(p_{s_1,q_1})_{\ell_1}(x) \cdot \Inner(p_{s_2,q_2})_{\ell_2}(x) ] = \mathbf{E}[\Inner(r_{s_1,q_1})_{\ell_1}(x) \cdot \Inner(r_{s_2,q_2})_{\ell_2}(x) ] .  \nonumber
\end{equation}
\item Every polynomial in the family $\mathcal{R}$ is $\beta(\mathrm{Num} + \mathrm{Coeff})$-eigenregular. 
\end{itemize}
\end{cond}
An important lemma (stated next and  proven in Appendix~\ref{app:corrjunta}) states that it is possible to realize a family $\mathcal{R}$ over much fewer variables. 
\begin{lemma}~\label{lem:junta-construct}
It is possible to construct a family $\mathcal{R}$ meeting the requirements of Condition~\ref{cond:junta} with $n_0= \mathsf{poly}(d, \mathrm{Num}, \beta(\mathrm{Num} + \mathrm{Coeff}))$ variables. 
\end{lemma}
Having constructed the family $\mathcal{R}$, for $1 \le s \le k$ and $1 \le q \le d$, we define
\[
\tilde{r}_{s,q} =\Outer(r_{s,q}) \left(\Inner(r_{s,q})_1(x),\dots, \Inner(r_{s,q})_{\num({s,q})}(x)\right),
\]
where the polynomial $\Outer(r_{s,q}): \mathbb{R}^{\num(s,q)} \rightarrow \mathbb{R}$ is the same as the polynomial $\Outer(p_{s,q})$. For $1 \le s \le k$, we define
\[
\tilde{r}_s =  r_{s,0} + \sum_{q=1}^d c_{s,q} \tilde{r}_{s,q},
\]
where $r_{s,0}  = p_{s,0}$. Finally, we define $\fh = \mathsf{PTF}(\tilde{r}_1, \ldots, \tilde{r}_k)$. Note that $\fh: \mathbb{R}^{n_0} \rightarrow \{0,1\}$.
We will now show that this construction has all the properties stated in Theorem~\ref{thm:reduction}. 
First of all, note that as long as the function $\beta(\cdot)$ is explicitly defined, $n_0$ is an explicit function of $d$ and $\epsilon$. The proof of the remaining properties will be done in several steps. Towards this, we will define a new family of polynomials which will essentially be a ``noise-attenuated" versions of the polynomials. These polynomials will be over $2n_0$ variables. 
For $1 \le s \le k$, $1 \le q \le d$ and $1 \le \ell \le \num(s,q)$, we define 
\[
\Inner(u_{s,q})_\ell(x,y) = \Inner(p_{s,q})_\ell(e^{-t}(x) + \sqrt{1 - e^{-2t}}(y)), \ \Inner(v_{s,q})_\ell(x,y) = \Inner(r_{s,q})_\ell(e^{-t}(x) + \sqrt{1 - e^{-2t}}(y)). 
\]
\[
\tilde{u}_{s,q} = \Outer(u_{s,q}) \left(\Inner(u_{s,q})_1(x,y),\dots, \Inner(u_{s,q})_{\num({s,q})}(x,y)\right), \
\tilde{v}_{s,q} =\Outer(v_{s,q}) \left(\Inner(v_{s,q})_1(x,y),\dots, \Inner(v_{s,q})_{\num({s,q})}(x,y)\right) 
\]
Here $ \Outer(u_{s,q}) = \Outer(p_{s,q})$ and $\Outer(v_{s,q}) = \Outer(r_{s,q})$. Finally, for $1 \le s \le k$, we define 
\[
\tilde{u}_s = u_{s,0} + \sum_{q=1}^d c_{s,q} \tilde{u}_{s,k} , \ \ \tilde{v}_s = v_{s,0} + \sum_{q=1}^d c_{s,q} \tilde{v}_{s,k},
\]
where $u_{s,0} = v_{s,0} = p_{s,0}= r_{s,0}$. For the convenience of the reader, note that the  domain of the families of polynomials defined using the letter `u' is $\mathbb{R}^{2n}$ and that of polynomials defined using the letter `v' is $\mathbb{R}^{2n_0}$. Let us define $\tilde{f}_u : \mathbb{R}^{2n} \rightarrow [k]$ and $\fh_{v} : \mathbb{R}^{2n_0} \rightarrow [k]$ as $$\tilde{f}_u = \mathsf{PTF}(\tilde{u}_1, \ldots, \tilde{u}_k), \ \  \fhv= \mathsf{PTF}(\tilde{v}_1, \ldots, \tilde{v}_k).$$
Defining $\tilde{f}_u$ and $\fhv$ as above, we have the following useful relations: 
$$
\mathop{\mathbf{E}}_{x \sim \gamma_n} [\langle \tilde{f}(x) , P_{t} \tilde{f}(x) \rangle] = \mathop{\mathbf{E}}_{x \sim \gamma_n, y \sim \gamma_n}[\langle \tilde{f}(x), \tilde{f}_u(x,y) \rangle], \ 
$$ 
$$
\mathop{\mathbf{E}}_{x \sim \gamma_n} [\langle \fh(x) , P_{t} \fh(x)\rangle] = \mathop{\mathbf{E}}_{x \sim \gamma_n, y \sim \gamma_n}[\langle \fh(x), \fhv(x,y) \rangle].
$$
The next claim establishes relations between the pairwise correlations of polynomials in the family
$\{\Inner(u_{s,q})_\ell\}$ and $\{\Inner(v_{s,q})_\ell\}$. 
\begin{claim}\label{clm:const-corr}
\begin{enumerate}
\item For all $1 \le s_1, s_2 \le k$, $1 \le q_1, q_2 \le d$ and $1 \le \ell_i \le \mathsf{num}(s_i, q_i)$ (for $i \in \{1,2\}$), we have \[
\mathbf{E}[\Inner(u_{s_1,q_1})_{\ell_1}(x,y) \cdot \Inner(u_{s_2,q_2})_{\ell_2}(x,y) ] = \mathbf{E}[\Inner(v_{s_1,q_1})_{\ell_1}(x,y) \cdot \Inner(v_{s_2,q_2})_{\ell_2}(x,y) ] .\]
\item For all $1 \le s_1, s_2 \le k$, $1 \le q_1, q_2 \le d$ and $1 \le \ell_i \le \mathsf{num}(s_i, q_i)$ (for $i \in \{1,2\}$), we have \[
\mathbf{E}[\Inner(p_{s_1,q_1})_{\ell_1}(x) \cdot \Inner(u_{s_2,q_2})_{\ell_2}(x,y) ] = \mathbf{E}[\Inner(r_{s_1,q_1})_{\ell_1}(x) \cdot \Inner(v_{s_2,q_2})_{\ell_2}(x,y) ] .\]
\end{enumerate}
\end{claim}
\begin{proof}
\textbf{Proof of Item 1.} The proof of the first item is quite easy. Namely, note that $\Inner(u_{s_1,q_1})_{\ell_1}$ and 
$\Inner(u_{s_2,q_2})_{\ell_2}$ are obtained by applying a unitary transformation (on the space of variables) to the polynomials $\Inner(p_{s_1,q_1})_{\ell_1}$ and 
$\Inner(p_{s_2,q_2})_{\ell_2}$. Thus, we have
\[
\mathbf{E}[\Inner(u_{s_1,q_1})_{\ell_1}(x,y) \cdot \Inner(u_{s_2,q_2})_{\ell_2}(x,y) ] = \mathbf{E}[\Inner(p_{s_1,q_1})_{\ell_1}(x) \cdot \Inner(p_{s_2,q_2})_{\ell_2}(x) ] .
\]
Likewise, 
\[
\mathbf{E}[\Inner(r_{s_1,q_1})_{\ell_1}(x) \cdot \Inner(r_{s_2,q_2})_{\ell_2}(x) ] = \mathbf{E}[\Inner(v_{s_1,q_1})_{\ell_1}(x,y) \cdot \Inner(v_{s_2,q_2})_{\ell_2}(x,y) ] .
\]
However, by construction of the family $\{\Inner(r_{s,q})_{\ell}\}$, we have that 
\[
\mathbf{E}[\Inner(r_{s_1,q_1})_{\ell_1}(x) \cdot \Inner(r_{s_2,q_2})_{\ell_2}(x) ] = \mathbf{E}[\Inner(p_{s_1,q_1})_{\ell_1}(x) \cdot \Inner(p_{s_2,q_2})_{\ell_2}(x) ] .
\]
This proves the first item. ~\\
\textbf{Proof of Item 2:} The proof of this is somewhat more involved. First of all, note that that there exists $j_1, j_2 \in \mathbb{N}$, such that 
$$
\Inner(p_{s_1,q_1})_{\ell_1}(x), \ \Inner(r_{s_1,q_1})_{\ell_1}(x) \in \mathcal{W}^{j_1} \\ \ \ \ ; \ \  \Inner(u_{s_2,q_2})_{\ell_2}(x,y), \ \Inner(v_{s_2,q_2})_{\ell_2}(x,y) \in \mathcal{W}^{j_2}. 
$$
Note that if $j_1 \not = j_2$, then Item 2 trivially holds as polynomials in different levels of the Wiener chaos are orthogonal. Thus, from now onwards, assume that $j_1 = j_2  =j$. 
Next, we observe that 
\[
\mathbf{E}[\Inner(p_{s_1,q_1})_{\ell_1}(x) \cdot \Inner(u_{s_2,q_2})_{\ell_2}(x,y) ]  = \langle \Inner(p_{s_1,q_1})_{\ell_1}(x),\  P_t \  \Inner(p_{s_2,q_2})_{\ell_2}(x) \rangle
\]
However, $P_t \  \Inner(p_{s_2,q_2})_{\ell_2}(x) = e^{-jt} \cdot \Inner(p_{s_2,q_2})_{\ell_2}(x)$. Thus, 
\[
\mathbf{E}[\Inner(p_{s_1,q_1})_{\ell_1}(x) \cdot \Inner(u_{s_2,q_2})_{\ell_2}(x,y) ] = e^{-jt} \mathbf{E}[\Inner(p_{s_1,q_1})_{\ell_1}(x) \cdot \Inner(p_{s_2,q_2})_{\ell_2}(x) ].
\]
Likewise, 
\[
\mathbf{E}[\Inner(r_{s_1,q_1})_{\ell_1}(x) \cdot \Inner(v_{s_2,q_2})_{\ell_2}(x,y) ] 
 = e^{-jt} \cdot \mathbf{E}[\Inner(r_{s_1,q_1})_{\ell_1}(x) \cdot \Inner(r_{s_2,q_2})_{\ell_2}(x,y) ]
\]
However, 
\[
\mathbf{E}[\Inner(r_{s_1,q_1})_{\ell_1}(x) \cdot \Inner(r_{s_2,q_2})_{\ell_2}(x,y) ] = 
\mathbf{E}[\Inner(p_{s_1,q_1})_{\ell_1}(x) \cdot \Inner(p_{s_2,q_2})_{\ell_2}(x,y) ]. 
\]
This proves the second item as well. 
\end{proof}

We will now list all the conditions required on the function $\beta(\cdot)$ that we will be required by various parts of our proof. It is easy to see that by choosing the function $\beta(\cdot)$ to be sufficiently fast decreasing, one can satisfy all the conditions. As we have noticed earlier, once $\beta(\cdot)$ is chosen explicitly, the construction of the PTF $\fh$ is completely explicit. The conditions that we enforce on $\beta(\cdot)$ are given below. While the conditions enforced are somewhat tedious to state, they are stated in the form so that they are readily usable later in the paper. 
\begin{cond}\label{cond:enforce-beta}
The function $\beta: [1, \infty) \rightarrow [0,1]$ should satisfy the following conditions: 
\begin{enumerate}
\item For $\xi = \frac{\epsilon}{40k^2}$ and 
for $B^{(1)}, \delta^{(1)}, c^{(1)}$ defined as \[
B^{(1)} = \Omega\bigg( \ln  \frac{\mathsf{Num} \cdot d}{\xi} \bigg)^{d/2} ;  \ \delta^{(1)} = \bigg( \frac{\xi}{d \cdot B^{(1)} \cdot \sqrt{\mathsf{Num}} \cdot \mathsf{Coeff}^{1/d}} \bigg)^d; \   c_{(1)} = \frac{\mathsf{Num}}{\delta^{(1)} \cdot \sqrt{\xi}},
\]
we  require
\[
\beta(\mathsf{Coeff} + \mathsf{Num}) \cdot \mathsf{Coeff}^2 \cdot 2^{\mathsf{Num} \cdot (\mathsf{Num}\cdot 2d+1)} \cdot 2^{O(d^3)} \leq \frac{1}{4} , \ \ 
 2^{O(d \log d)} \cdot \mathsf{Num}^2 \cdot 4c_{(1)}^{2} \cdot \beta(\mathsf{Coeff} + \mathsf{Num}) \le \xi.
\]
\item Define  $\xi = \frac{\epsilon}{40k^2}$, $L = 4 \cdot 9^{d+1} \cdot (d+1)^2 \cdot \log^{d/2} (k \cdot d/\epsilon) $ and $\frac{L}{2} \cdot L^{-2^{2d}} \equiv \vartheta.$ Further,  define $B^{(2)}$, $\delta^{(2)}$ and $c_{(2)}$ as
\[
B^{(2)} = \Omega\bigg( \ln  \frac{\mathsf{2\cdot Num} \cdot 2d}{\xi} \bigg)^{d} ;  \ \delta^{(2)} = \bigg( \frac{\xi \cdot \vartheta^{1/2d}}{2d \cdot B^{(2)} \cdot \sqrt{2 \cdot \mathsf{Num}} \cdot \mathsf{Coeff}^{2/d}} \bigg)^{2d}; \   c_{(2)} = \frac{2 \cdot \mathsf{Num}}{\delta^{(2)} \cdot \sqrt{\xi}}.
\]
We require
\[
\beta(\mathsf{Coeff}^2 + 2 \cdot \mathsf{Num}) \cdot \mathsf{Coeff}^4 \cdot 2^{2\mathsf{Num} \cdot (\mathsf{Num}\cdot 4d+1)} \cdot 2^{O(d^3)} \leq \frac{\vartheta}{4} , \ \ 
 2^{O(d \log d)} \cdot \mathsf{Num}^2 \cdot 4c_{(2)}^{2} \cdot \beta(\mathsf{Coeff}^2 + 2 \cdot \mathsf{Num}) \le \xi.
\]
\end{enumerate}
\end{cond}

We now state a lemma whose consequences shall be used repeatedly in the rest of the proof. The lemma is stated below. 

\begin{lemma}\label{lem:monomial}
Let $p_1, \ldots, p_t \in \mathbb{R}^n \rightarrow \mathbb{R}$ such that for $1 \le i \le t$, $p_i = I_{q_i}(h_i)$ where $p_i \in \mathcal{W}^{q_i}$. Further, $\mathsf{Var}(p_i) =1$ and $\lambda_{\max}(h_i) \le \kappa$. Let $C \in \mathbb{R}^{t \times t}$ where $C(i,j) = \langle h_i, h_j \rangle$. 
There exists a function $F : \mathbb{R}^{t \times t} \times \mathbb{Z}^{t} \rightarrow \mathbb{R}$ such that 
\[
\bigg| \mathbf{E}\big[\prod_{i=1}^t p_i \big] - F(C, q_1, \ldots, q_t) \bigg| \le  2^{t (q_1 + \ldots + q_t+1)} \cdot \kappa. 
\]
In other words, up to the error term of $2^{t (q_1 + \ldots + q_t)} \cdot \kappa$, the expectation of the product $\prod_{i=1}^t p_i$ is just dependent on the degrees of the polynomials and the covariance matrix of the polynomials. 
\end{lemma}

The proof of Lemma~\ref{lem:monomial} is quite long and involved and is deferred to the appendix. However, this lemma can be used to derive the following important consequences stated below.

\begin{claim}~\label{lem:variance-lowerbound}
Let $\Psi:\mathbb{R}^{\mathsf{Num}_\Psi} \rightarrow \mathbb{R}$ be a degree-$d$ polynomial such that sum of the absolute values of its coefficients is given by $\mathsf{Coeff}_\Psi$. Let $\{A_i\}_{1 \le i \le \mathsf{Num}_\Psi}$ and $\{B_i \}_{1 \le i \le \mathsf{Num}_\Psi}$  be centered families polynomials of variance $1$ such that for all $1 \le i \le m$, $\exists 0< j \le d$ such that $A_i, B_i \in \mathcal{W}^j$. Further, for all $1 \le i \le j \le \mathsf{Num}_\Psi$, $\mathbf{E}[A_i A_j] = \mathbf{E}[B_i B_j]$. If each $A_i, B_i$ is $\zeta$-eigenregular where 
$
\zeta \cdot \mathsf{Coeff}_\Psi \cdot 2^{\mathsf{Num}_\Psi \cdot (\mathsf{Num}_\Psi \cdot d+1)} \leq \kappa, 
$
then $|\mathbf{E}[\Psi(A_1, \ldots, A_{\mathsf{Num}})] - \mathbf{E}[\Psi(B_1, \ldots, B_{\mathsf{Num}})]| \le \kappa$. 
\end{claim}
\begin{proof}
Applying Lemma~\ref{lem:monomial} and triangle inequality, we have
$$
|\mathbf{E}[\Psi(A_1, \ldots, A_{\mathsf{Num}})] - \mathbf{E}[\Psi(B_1, \ldots, B_{\mathsf{Num}})]|  \le \zeta \cdot\mathsf{Coeff}_\Psi \cdot 2^{\mathsf{Num}_\Psi \cdot (\mathsf{Num}_\Psi \cdot d+1)}. 
$$
Plugging the condition on $\zeta$, we finish our proof. 
\end{proof}
Thus, the above lemma states that if the polynomial families $\{A_i\}$ and $\{B_i\}$ are \emph{sufficiently eigenregular} with matching means and covariance, the expectation of $\Psi(A_1, \ldots, A_{\mathsf{Num}})$ 
is close to $\Psi(B_1, \ldots, B_{\mathsf{Num}})$. The next lemma puts a bound on the absolute value of the expectation of $\Psi(A_1, \ldots, A_{\mathsf{Num}})$. The proof is quite elementary except it relies on Proposition~\ref{prop:upper-bound} which is essentially a combination of hypercontractivity. 
\begin{claim}~\label{clm:variance-lowerbound2}
Let $\Psi:\mathbb{R}^{\mathsf{Num}_\Psi} \rightarrow \mathbb{R}$ be a degree-$d$ polynomial such that sum of the absolute values of its coefficients is given by $\mathsf{Coeff}_\Psi$. Let $\{A_i\}_{1 \le i \le \mathsf{Num}_\Psi}$ be a collection of centered  polynomials of variance $1$ such that for all $1 \le i \le m$, $\exists 0< j \le d$ such that $A_i \in \mathcal{W}^j$. 
Then, $\mathbf{E}[|\Psi(A_1, \ldots, A_m)|] \le \mathsf{Coeff}_\Psi \cdot 2^{O(d^3)}$. 
\end{claim}
\begin{proof}
Let $1 \le i_1,\ldots, i_d \leq m$. Let $Z_{i_1,\ldots, i_d} = \prod_{j=1}^d A_{i_j}$. Then, applying Jensen's inequality and Proposition~\ref{prop:upper-bound} iteratively (and exploiting the fact that the degree of the product $A_{i,j}$ is at most $d^2$), we have 
\[
\mathbf{E}[|Z_{i_1,\ldots, i_d}|] \le \sqrt{\mathbf{E}[Z^2_{i_1,\ldots, i_d}]} \le \sqrt{(d^2 (d^2 +1))^d \cdot 9^{d^3}}. 
\]
Let $\Psi(z_1, \ldots, z_m) = \sum_{\mathcal{S}} c_{\mathcal{S}} \prod_{j \in \mathcal{S}} A_{j}$ (where $\mathcal{S}$ ranges over all multistep of size at most $d$ on the set $[m]$).  Here $\mathsf{Coeff}_\Psi = \sum_{\mathcal{S}} |c_{\mathcal{S}}|$. Then, 
\[
\mathbf{E}[|\Psi(A_1, \ldots, A_m)| ] \le \sum_{\mathcal{S}} |c_{\mathcal{S}}| \cdot \sqrt{(d^2 (d^2 +1))^d \cdot 9^{d^3}} \le \mathsf{Coeff}_\Psi \cdot 2^{O(d^3)}.
\]
\end{proof}
The next lemma is an extension of Claim~\ref{lem:variance-lowerbound}. In particular, let us assume that $\{A_i\}$ and $\{B_i\}$ are \emph{sufficiently eigenregular} polynomial families then for any degree-$d$ polynomial $\Psi$, $\Psi(A_1, \ldots, A_{\mathsf{Num}})$ and $\Psi(B_1, \ldots, B_{\mathsf{Num}})$ are non-negative with approximately the same probability.

\begin{lemma}~\label{lem:sign-match}
Let $\Psi:\mathbb{R}^{\mathsf{Num}_\Psi} \rightarrow \mathbb{R}$ be a degree-$d$ polynomial such that sum of the absolute values of its coefficients is given by $\mathsf{Coeff}_\Psi$. Let $\{A_i\}_{1 \le i \le \mathsf{Num}_\Psi}$ and $\{B_i \}_{1 \le i \le \mathsf{Num}_\Psi}$  be centered families polynomials of variance $1$ such that for all $1 \le i \le m$, $\exists 0< j \le d$ such that $A_i, B_i \in \mathcal{W}^j$. Further, for all $1 \le i \le j \le \mathsf{Num}_\Psi$, $\mathbf{E}[A_i A_j] = \mathbf{E}[B_i B_j]$, $\mathsf{Var}(\Psi(A_1, \ldots, A_{\mathsf{Num}_\Psi})) \ge \eta$ and $|\mathbf{E}[\Psi(A_1, \ldots, A_{\mathsf{Num}_\Psi})] \le \eta$. 
Let us define $B, \delta, c$ as follows \[
B = \Omega\bigg( \ln  \frac{\mathsf{Num}_\Psi \cdot d}{\epsilon} \bigg)^{d/2} ;  \ \delta = \bigg( \frac{\epsilon \cdot \eta^{1/2d}}{d \cdot B \cdot \sqrt{\mathsf{Num}_\Psi} \cdot \mathsf{Coeff}_\Psi^{1/d}} \bigg)^d; \   c = \frac{\mathsf{Num_\Psi}}{\delta \cdot \sqrt{\epsilon}}.
\]
If each $A_i, B_i$ is $\zeta$-eigenregular  satisfies 
\begin{eqnarray*}
\zeta \cdot \mathsf{Coeff}^2_\Psi \cdot 2^{\mathsf{Num}_\Psi \cdot (\mathsf{Num}_\Psi \cdot 2d+1)} \cdot 2^{O(d^3)} \leq \frac{\eta}{4} , \ \ 
 2^{O(d \log d)} \cdot \mathsf{Num}_{\Psi}^2 \cdot (4c^2) \cdot \zeta \le \epsilon.
\end{eqnarray*}
then $|\mathbf{E}[\mathsf{sign}(\Psi(A_1, \ldots, A_{\mathsf{Num}}))] - \mathbf{E}[\mathsf{sign}(\Psi(B_1, \ldots, B_{\mathsf{Num}}))]| \le O(\epsilon)$. 
\end{lemma}
\begin{proof}
Let us define the function $\Upsilon(z_1, \ldots, z_{\mathsf{Num}_\Psi}) = \Psi^2(z_1, \ldots, z_{\mathsf{Num}_\Psi})$. Let us define $\mathsf{Num}_\Upsilon$ to be the number of arguments of $\Upsilon$ and $\mathsf{Coeff}_\Upsilon$ to be the sum of the absolute value of its coefficients. Observe that $\Upsilon$ is a degree-$2d$ polynomial such that $\mathsf{Coeff}_{\Upsilon} \le \mathsf{Coeff}^2_{\Psi}$ and $\mathsf{Num}_\Upsilon  = \mathsf{Num}_\Psi$. Now, applying Claim~\ref{lem:variance-lowerbound} to the function $\Upsilon$, 
\begin{equation}\label{eq:variance-diff1}
\big|\mathbf{E} [ \Psi^2(A_1, \ldots, A_{\mathsf{Num}})] -  \mathbf{E} [ \Psi^2(B_1, \ldots, B_{\mathsf{Num}})]  \big| \le \eta/4. 
\end{equation}
Similarly, we also get that
\[
\big|\mathbf{E} [ \Psi(A_1, \ldots, A_{\mathsf{Num}})] -  \mathbf{E} [ \Psi(B_1, \ldots, B_{\mathsf{Num}})]  \big| \le \zeta \cdot \mathsf{Coeff}_\Psi \cdot 2^{\mathsf{Num}_\Psi \cdot (\mathsf{Num}_\Psi \cdot d+1)} \le \frac{\eta}{4 \cdot \mathsf{Coeff}_\Psi \cdot 2^{\mathsf{Num}_\Psi^2 \cdot d} \cdot 2^{O(d^3)}}.
\]
Applying Claim~\ref{clm:variance-lowerbound2}, we get
\begin{equation}\label{eq:variance-diff2}
\big|\big(\mathbf{E} [ \Psi(A_1, \ldots, A_{\mathsf{Num}})]\big)^2- \big(\mathbf{E} [ \Psi(B_1, \ldots, B_{\mathsf{Num}})]\big)^2 \big| \le \eta/4. 
\end{equation}
Combining (\ref{eq:variance-diff1}) and (\ref{eq:variance-diff2}), get that 
$
|\mathsf{Var}(\Psi(A_1, \ldots, A_{\mathsf{Num}})) - \mathsf{Var}(\Psi(A_1, \ldots, A_{\mathsf{Num}}))| \le \eta/2
$ and thus, $$\mathsf{Var}(\Psi(B_1, \ldots, B_{\mathsf{Num}})) \ge \eta/2.$$  Next, observe that for $\Psi$, the sum of squares of its coefficients (denoted by $S$) is at most $\mathsf{Coeff}_\Psi^2$. Recall that we define $B, \delta, c$ as follows \[
B = \Omega\bigg( \ln  \frac{\mathsf{Num}_\Psi \cdot d}{\epsilon} \bigg)^{d/2} ;  \ \delta = \bigg( \frac{\epsilon \cdot \eta^{1/2d}}{d \cdot B \cdot \sqrt{\mathsf{Num}_\Psi} \cdot \mathsf{Coeff}_\Psi^{1/d}} \bigg)^d; \   c = \frac{\mathsf{Num_\Psi}}{\delta \cdot \sqrt{\epsilon}}.
\]
Then, applying Theorem~\ref{thm:mollification}, we get that there is a function $\tilde{g}_c : \mathbb{R}^{\mathsf{Num}} \rightarrow [0,1]$ such that $\Vert \tilde{g}^{(1)}_c \Vert_{\infty} \le 2c$, $\Vert \tilde{g}^{(1)}_c \Vert_{\infty} \le 4c^2$, 
\[
\big|\mathbf{E} [ \mathsf{sign}(\Psi(A_1, \ldots, A_{\mathsf{Num}}))]  - \mathbf{E} [ \tilde{g}_c(\Psi(A_1, \ldots, A_{\mathsf{Num}}))] \big| \le O(\epsilon).
\]
\[
\big|\mathbf{E} [ \mathsf{sign}(\Psi(B_1, \ldots, B_{\mathsf{Num}}))]  - \mathbf{E} [ \tilde{g}_c(\Psi(B_1, \ldots, B_{\mathsf{Num}}))] \big| \le O(\epsilon).
\]
Finally, applying Theorem~\ref{thm:CLT}, we have 
\[
\big| \mathbf{E} [ \tilde{g}_c(\Psi(B_1, \ldots, B_{\mathsf{Num}}))]- \mathbf{E} [ \tilde{g}_c(\Psi(A_1, \ldots, A_{\mathsf{Num}}))] \big| \le 2^{O(d \log d)} \cdot \mathsf{Num}_\Psi^2 \cdot (4c^2) \cdot \zeta \le \epsilon.
\]
The last inequality uses the second condition on $\zeta$.  Combining these three inequalities, we obtain the proof. 
\end{proof}
We will now use the above lemma to prove the following corollary which essentially states that the cdf of the polynomials $\tilde{r}_s$ (resp. $\tilde{v}_s$) is close to $\tilde{p}_s$ (resp. $\tilde{u}_s$). In fact, the same holds for their joint cdf as well. All the consequences in the above corollary are proven subject to the conditions listed in Condition~\ref{cond:enforce-beta}. For the rest of the proof, set 
$\xi = \frac{\epsilon}{40k^2}$. 
\begin{corollary}~\label{corr:bounds}
For all $1 \le s \le k$, $1 \le s' \le k$, 
\begin{itemize}
\item[(a)] $\big| \Pr_{x \sim \gamma_{n}}[ \tilde{p}_s \ge 0 ] - \Pr_{x \sim \gamma_{n_0}}[ \tilde{r}_s \ge 0]\big| \le \xi. $
\item[(b)] $\big| \Pr_{x,y \sim \gamma_{n}}[ \tilde{u}_s \ge 0 ] -\Pr_{x,y \sim \gamma_{n_0}}[ \tilde{v}_s \ge 0]\big| \le \xi. $
\item[(c)] $\big| \Pr_{x \sim \gamma_{n}}[ \tilde{p}_s \cdot \tilde{p}_{s'} \ge 0 ] - \Pr_{x \sim \gamma_{n_0}}[ \tilde{r}_s \cdot \tilde{r}_{s'} \ge 0]\big| \le \xi. $
\item[(d)] $\big| \Pr_{x,y \sim \gamma_{n}}[ \tilde{u}_s \cdot \tilde{u}_{s'} \ge 0 ] - \Pr_{x,y \sim \gamma_{n_0}}[ \tilde{v}_s \cdot \tilde{v}_{s'} \ge 0]\big| \le \xi. $
\item[(e)] $\big| \Pr_{x,y \sim \gamma_{n}}[ \tilde{p}_s \cdot \tilde{u}_{s'} \ge 0 ] - \Pr_{x,y \sim \gamma_{n_0}}[ \tilde{r}_s \cdot \tilde{v}_{s'} \ge 0]\big| \le \xi. $
\end{itemize}
\end{corollary}
\begin{proof}
We first note that it suffices to prove (a), (c) and (e). This is because the polynomials $\{\tilde{u}_s\}$ (resp. $\{\tilde{v}_s\}$) are obtained by an isometric transformation of variables 
on the polynomials $\{\tilde{p}_s\}$ (resp. $\{ \tilde{r}_s\}$). For $1 \le s \le k$, note that there is a $\Psi_s: \mathbb{R}^{\mathsf{Num}} \rightarrow \mathbb{R}$ 
\[
\tilde{p}_s = \Psi_s( \{\Inner(p_{s,q})_1(x),\dots, \Inner(p_{s,q})_{\num({s,q})}(x)\}_{ 1 \le q \le d})
\]
\[
\tilde{r}_s = \Psi_s( \{\Inner(r_{s,q})_1(x),\dots, \Inner(r_{s,q})_{\num({s,q})}(x)\}_{ 1 \le q \le d})
\]
~\\
\textbf{Proof of Item (a):} Note that  by definition $\mathsf{Var}(\tilde{p}_s)=1$. 
 Let us define $B^{(1)}, \delta^{(1)}, c^{(1)}$  \[
B^{(1)} = \Omega\bigg( \ln  \frac{\mathsf{Num} \cdot d}{\xi} \bigg)^{d/2} ;  \ \delta^{(1)} = \bigg( \frac{\xi}{d \cdot B^{(1)} \cdot \sqrt{\mathsf{Num}} \cdot \mathsf{Coeff}^{1/d}} \bigg)^d; \   c_{(1)} = \frac{\mathsf{Num}}{\delta^{(1)} \cdot \sqrt{\xi}},
\]
and thus the first item of Condition~\ref{cond:enforce-beta} implies
\[
\beta(\mathsf{Coeff} + \mathsf{Num}) \cdot \mathsf{Coeff}^2 \cdot 2^{\mathsf{Num} \cdot (\mathsf{Num}\cdot 2d+1)} \cdot 2^{O(d^3)} \leq \frac{1}{4} , \ \ 
 2^{O(d \log d)} \cdot \mathsf{Num}^2 \cdot 4c_{(1)}^{2} \cdot \beta(\mathsf{Coeff} + \mathsf{Num}) \le \xi.
\]
We now apply Lemma~\ref{lem:sign-match} to the function $\Psi_s$ with $\eta=1$.  Note that by construction,
\[
\mathbf{E}[\Inner(p_{s,q_1})_{\ell_1} \cdot \Inner(p_{s,q_2})_{\ell_2}] = \mathbf{E}[\Inner(r_{s,q_1})_{\ell_1} \cdot \Inner(r_{s,q_2})_{\ell_2}]. 
\]
Thus, we obtain $\big| \Pr_{x \sim \gamma_{n}}[ \tilde{p}_s \ge 0 ] - \Pr_{x \sim \gamma_{n_0}}[ \tilde{r}_s \ge 0]\big| \le \xi, $ proving (a). 
~\\
\textbf{Proof of Item (c):} For $1 \le s, s' \le k$, note that there is a function $\Psi_{s,s'}: \mathbb{R}^{2 \cdot \mathsf{Num}} \rightarrow \mathbb{R}$ such that 
\[
\tilde{p}_s \cdot \tilde{p}_{s'} = \Psi_{s,s'}( \{\Inner(p_{s,q})_1(x),\dots, \Inner(p_{s,q})_{\num({s,q})}(x)\}_{ 1 \le q \le d}, \{\Inner(p_{s',q})_1(x),\dots, \Inner(p_{s',q})_{\num({s',q})}(x)\}_{ 1 \le q \le d})
\]
\[
\tilde{r}_s \cdot \tilde{r}_{s'} = \Psi_{s,s'}( \{\Inner(r_{s,q})_1(x),\dots, \Inner(r_{s,q})_{\num({s,q})}(x)\}_{ 1 \le q \le d}, \{\Inner(r_{s',q})_1(x),\dots, \Inner(r_{s',q})_{\num({s',q})}(x)\}_{ 1 \le q \le d})
\]
Note that $\Psi_{s,s'}$ is a degree-$2d$ polynomial and $\mathsf{Coeff}_{\Psi_{s,s'}} \le \mathsf{Coeff}^2$ and $\mathsf{Num}_{\Psi_{s,s'}} \le 2 \mathsf{Num}$.  Let us define $L = 4 \cdot 9^{d+1} \cdot (d+1)^2 \cdot \log^{d/2} (k \cdot d/\epsilon) $. By using the $(d,2\epsilon)$-balancedness of $\tilde{p}_s$ and $\tilde{p}_{s'}$ and applying Lemma~\ref{lemma:variance-product-lower}, 
\[
\mathsf{Var}(\tilde{p}_s \cdot \tilde{p}_{s'}) \ge \frac{L}{2} \cdot L^{-2^{2d}} \equiv \vartheta. 
\]
 Let us define
\[
B^{(2)} = \Omega\bigg( \ln  \frac{\mathsf{2 \cdot Num} \cdot 2d}{\xi} \bigg)^{d} ;  \ \delta^{(2)} = \bigg( \frac{\xi \cdot \vartheta^{1/2d}}{2d \cdot B^{(2)} \cdot \sqrt{2\cdot \mathsf{Num}} \cdot \mathsf{Coeff}^{2/d}} \bigg)^{2d}; \   c_{(2)} = \frac{\mathsf{Num}}{\delta^{(2)} \cdot \sqrt{\xi}}.
\]
Applying Item 2 of Condition~\ref{cond:enforce-beta}, we see that  
\[
\beta(\mathsf{Coeff}^2 + 2 \cdot \mathsf{Num}) \cdot \mathsf{Coeff}^4 \cdot 2^{2\mathsf{Num} \cdot (\mathsf{Num}\cdot 4d+1)} \cdot 2^{O(d^3)} \leq \frac{\vartheta}{4} , \ \ 
 2^{O(d \log d)} \cdot \mathsf{Num}^2 \cdot 4c_{(2)}^{2} \cdot \beta(\mathsf{Coeff}^2 + 2 \cdot \mathsf{Num}) \le \xi.
\]
Note that applying Claim~\ref{clm:const-corr}, we have that
for any $t,t' \in \{s,s'\}$ and any $1 \le q, q' \le d$ and $\ell \le \mathsf{num}(t,q)$, $\ell' \le \mathsf{num}(t',q')$, we have
\[
\mathbf{E}[\Inner(p_{t,q})_{\ell} \cdot \Inner(p_{t',q'})_{\ell'} ] = \mathbf{E}[\Inner(r_{t,q})_{\ell} \cdot \Inner(r_{t',q'})_{\ell'} ]. 
\]
Thus, we can 
now apply Lemma~\ref{lem:sign-match} to the function $\Psi_{s,s'}$ with $\eta=\vartheta$ and set
This ensures that $\big| \Pr_{x \sim \gamma_{n}}[ \tilde{p}_s \cdot \tilde{p}_{s'}\ge 0 ] - \Pr_{x \sim \gamma_{n_0}}[ \tilde{r}_s \cdot \tilde{r}_{s'} \ge 0]\big| \le \xi, $ proving (c).  
~\\
\textbf{Proof of Item (e):} The proof of item (e) is the same as the proof of item (c) with $\tilde{u}_{s'}$ taking over the role of $\tilde{p}_{s'}$. We leave the proof to the reader. 
\end{proof}
Before we move on, we summarize the following consequence of the above corollary which will be useful for us in a subsequent work (and appears to be an interesting technical result in its own right). 
\begin{theorem}
Let $p_1, \ldots, p_k: \mathbb{R}^n \rightarrow \mathbb{R}$ be degree-$d$ polynomials and for $\delta>0$, the following two conditions: (i) For all $1 \le s \le \ell$, $\mathsf{Var}(p_s)=1$ and (ii) For all $1 \le s \le \ell$, $|\mathbf{E}[p_s]| \le \log^{d/2} (k \cdot d /\delta)$. For $1 \le s \le k$ and $t>0$, define  $u_s: \mathbb{R}^{2n} \rightarrow \mathbb{R}$
as follows: $u_s(x,y) = p_s(e^{-t} x + \sqrt{1-e^{-2t}} y)$. Then, there is an explicitly computable 
$n_0 = n_0(k, d, \epsilon)$ and polynomials $\tilde{r}_1, \ldots, \tilde{r}_{k}: \mathbb{R}^{n_0} \rightarrow \mathbb{R}$   with the following properties: For $1 \le s \leq \ell$, define $\tilde{v}_s: \mathbb{R}^{2n_0} \rightarrow \mathbb{R}$ as $\tilde{v}_s(x,y) = \tilde{r}_s(e^{-t} x + \sqrt{1-e^{-2t}} y)$. Then, for $ 1 \le s, s' \le k$, 
\begin{enumerate}
\item[(a)] $\big|\Pr_{x \sim \gamma_n} [{p}_s \ge 0] - \Pr_{x \sim \gamma_n} [\tilde{r}_s \ge 0]\big| \le \epsilon$. 
\item[(b)] $\big|\Pr_{x,y \sim \gamma_{n}} [{u}_s \ge 0] - \Pr_{x,y \sim \gamma_{n_0}} [\tilde{v}_s \ge 0]\big| \le \epsilon$. 
\item[(c)] $\big|\Pr_{x \sim \gamma_n} [{p}_s \cdot {p}_{s'} \ge 0] - 
\Pr_{x \sim \gamma_{n_0}} [\tilde{r}_s \cdot \tilde{r}_{s'} \ge 0]\big| \le \epsilon$. 
\item[(d)] $\big|\Pr_{x,y \sim \gamma_n} [{u}_s\cdot {u}_{s'} \ge 0] - 
\Pr_{x,y \sim \gamma_{n_0}} [\tilde{v}_s \cdot \tilde{v}_{s'} \ge 0]\big| \le \epsilon$. 
\item[(e)] $\big|\Pr_{x,y \sim \gamma_n} [{p}_s\cdot {u}_{s'} \ge 0] - 
\Pr_{x,y \sim \gamma_{n_0}} [\tilde{r}_s \cdot \tilde{v}_{s'} \ge 0]\big| \le \epsilon$. 
\end{enumerate}
\end{theorem}
\begin{proof}
Note that the bound on the arity of $\tilde{r}_1, \ldots, \tilde{r}_k$ directly follows from our  construction. Further, note that  $\mathsf{Var}(p_s-\tilde{p}_s) \le \tau$, hence, we obtain that \begin{equation}~\label{eq:gamma-2} \Pr_{x \sim \gamma_n} [\mathsf{sign}(p_s(x)) \not = \mathsf{sign}(\tilde{p}_s(x))] =O(\epsilon).\end{equation} Proofs of item (a), (b), (c), (d) and (e) follow by combining (\ref{eq:gamma-2}) (respectively) with item (a), (b), (c), (d) and (e) of Corollary~\ref{corr:bounds}.
\end{proof}

We will now use the above corollary  to claim that $\Pr_{x \sim \gamma_{n_0}} [x \in \mathsf{Collision}(\fh)]$ is small. Before proceeding, let us define one additional piece of notation: Namely, for a multivariate PTF $f = \mathsf{PTF}(p_1, \ldots, p_k)$, let us define $\mathsf{Unique}(f,i) = \overline{\mathsf{Collision}(f)} \cap \{x : p_i(x)  \ge 0\}$. In other words, it is the set of all points such that $i$ is the unique index such that $p_i(x) \ge 0$.
\begin{claim}~\label{clm:collision-upper}
$\Pr_{x \sim \gamma_{n_0}} [x \in \mathsf{Collision}(\fh)]\le  (k^2+1) \cdot \Pr[x \in \mathsf{Collision}(f)] + k \cdot \xi + \frac{3k^2 \xi}{2}. $ \end{claim}
\begin{proof}
First of all, note that any of $p_1(x), \ldots, p_k(x) =0$ only in  a measure zero set of the domain. So, from now on, let us assume that at all points $x$, $p_s(x) \not =0$ (for $1 \le s \le k$). Next, observe that for real-valued quantities $A, B \not =0$, we have 
\begin{equation}\label{eq:identity}
\mathbf{I} \ [A \ge 0 \wedge B \ge 0] = \frac12 \big( \mathbf{I} \ [ A \cdot B \ge 0] + \mathbf{I} \ [A \ge 0] + \mathbf{I} \ [B \ge 0] - 1\big). 
\end{equation}
Thus, 
\begin{eqnarray}
\big| \Pr [p_{s}(x) \ge 0 \wedge p_{s'}(x) \ge 0]  - \Pr [r_{s}(x) \ge 0 \wedge r_{s'}(x) \ge 0] \big| &\le& \frac{1}{2} \big( \big| \Pr[p_{s}(x) \ge 0 ]  - \Pr [r_{s}(x) \ge 0 ] \big|\big) \nonumber\\
&+& \frac{1}{2} \big( \big| \Pr_{x \sim \gamma_n} [p_{s'}(x) \ge 0 ]  - \Pr_{x \sim \gamma_n} [r_{s'}(x) \ge 0 ] \big|\big) \nonumber \\ &+& 
\frac{1}{2}  \big( \big| \Pr_{x \sim \gamma_n} [p_{s'}(x) \cdot p_s(x) \ge 0 ]  - \Pr_{x \sim \gamma_n} [r_{s'}(x)  \cdot r_s(x)\ge 0 ] \big|\big)~\label{eq:gamma1}. 
\end{eqnarray}
Applying the first and third item of  Corollary~\ref{corr:bounds}, we have that the right hand side is bounded by $3\xi/2$.  Now observe that 
\[
\Pr_{x \sim \gamma_n} [p_s(x) \ge 0] - \sum_{s' \not =s} \Pr_{x \sim \gamma_n} [p_s(x) \ge 0 \wedge p_{s'}(x) \ge 0]  \leq \Pr_{x \sim \gamma_n} [x \in \mathsf{Unique}(f,s)] \le \Pr_{x \sim \gamma_n} [p_s(x) \ge 0]
\]
\[
\Pr_{x \sim \gamma_n} [r_s(x) \ge 0] - \sum_{s' \not =s} \Pr_{x \sim \gamma_n} [r_s(x) \ge 0 \wedge r_{s'}(x) \ge 0]  \leq \Pr_{x \sim \gamma_n} [x \in \mathsf{Unique}(\fh,s)] \le \Pr_{x \sim \gamma_n} [r_s(x) \ge 0]
\]
Applying (\ref{eq:gamma1}), we obtain
\begin{eqnarray*}
\big|\Pr_{x \sim \gamma_n} [x \in \mathsf{Unique}(f,s)]- \Pr_{x \sim \gamma_n} [x \in \mathsf{Unique}(\fh,s)] \big| &\leq& \big| \Pr_{x \sim \gamma_n} [p_s(x) \ge 0] -\Pr_{x \sim \gamma_n} [r_s(x) \ge 0]\big| \\ &+& \big|\sum_{s' \not =s} \Pr_{x \sim \gamma_n} [p_s(x) \ge 0 \wedge p_{s'}(x) \ge 0]\big| + \frac{3k \xi}{2}.
\end{eqnarray*}
Adding over all $s$ and applying the first item of Corollary~\ref{corr:bounds},  we obtain 
\[
\sum_{s=1}^k \big|\Pr_{x \sim \gamma_n} [x \in \mathsf{Unique}(f,s)]- \Pr_{x \sim \gamma_n} [x \in \mathsf{Unique}(\fh,s)] \big| \le k\cdot \xi + \frac{3k^2 \xi}{2} +\sum_{s' \not =s} \Pr_{x \sim \gamma_n} [p_s(x) \ge 0 \wedge p_{s'}(x) \ge 0]
\]

Noting that for all $(s,s')$, $\Pr [p_s(x) \ge 0 \wedge p_{s'}(x) \ge 0] \le \Pr [x \in \mathsf{Collision}(f)]$, we obtain
\[
\sum_{s=1}^k \big|\Pr_{x \sim \gamma_n} [x \in \mathsf{Unique}(f,s)]- \Pr_{x \sim \gamma_n} [x \in \mathsf{Unique}(\fh,s)] \big| \le k^2 \cdot \Pr[x \in \mathsf{Collision}(f)] + k \cdot \xi + \frac{3k^2 \xi}{2}. 
\]
This however immediately implies that 
$$
\Pr_x [x \in \mathsf{Collision}(\fh)] \le  (k^2+1) \cdot \Pr[x \in \mathsf{Collision}(f)] + k \cdot \xi + \frac{3k^2 \xi}{2}. 
$$
\end{proof}

\begin{lemma}~\label{lem:noise-stab}
\[
\big| \mathbf{E}[\langle f, P_t f \rangle]  - \mathbf{E}[\langle \fh, P_t \fh \rangle]\big| \le 2\Pr[x \in \mathsf{Collision}(f)] +2 \Pr[x \in \mathsf{Collision}(\fh)] + \frac{3k\xi}{2}. 
\]
\end{lemma}
\begin{proof}
We begin by noting the following two inequalities 
\[
\big| \mathbf{E}[\langle f, P_t f \rangle]  -\sum_{s=1}^k \Pr[p_s(x) \ge 0 \wedge u_s(x,y) \ge 0] \big| \le 2 \Pr[x \in \mathsf{Collision}(f)].
\]
\[
\big| \mathbf{E}[\langle \fh, P_t \fh \rangle]  -\sum_{s=1}^k \Pr[r_s(x) \ge 0 \wedge v_s(x,y) \ge 0] \big| \le 2 \Pr[x \in \mathsf{Collision}(\fh)].
\]
Thus, 
\begin{eqnarray}
\big| \mathbf{E}[\langle f, P_t f \rangle] - \mathbf{E}[\langle \fh, P_t \fh \rangle] \big| &\le& 2\Pr[x \in \mathsf{Collision}(f)] +2 \Pr[x \in \mathsf{Collision}(\fh)] \nonumber \\ 
&+& \sum_{s=1}^k \big|\Pr[p_s(x) \ge 0 \wedge u_s(x,y) \ge 0]- \Pr[r_s(x) \ge 0 \wedge v_s(x,y) \ge 0] \big| \label{eq:noise-junta-1}
\end{eqnarray}
We now seek to bound $\Pr[p_s(x) \ge 0 \wedge u_s(x) \ge 0]- \Pr[r_s(x) \ge 0 \wedge v_s(x) \ge 0]$. As before, using that the polynomials $p_s, u_s, r_s$ and $v_s$ are zero only
on a measure zero set and using the identity from (\ref{eq:identity}), we obtain that 
\begin{eqnarray*}
&& \big|\Pr[p_s(x) \ge 0 \wedge u_s(x,y) \ge 0]- \Pr[r_s(x) \ge 0 \wedge v_s(x,y) \ge 0] \big| \\ &\le& \frac{1}{2} \cdot  \big|\Pr[p_s(x) \ge 0]- \Pr[r_s(x) \ge 0] \big| +  \frac{1}{2} \cdot  \big|\Pr[u_s(x,y) \ge 0]- \Pr[v_s(x,y) \ge 0] \big|  \\ &+&  \frac{1}{2} \cdot  \big|\Pr[p_s(x) \cdot u_s(x,y) \ge 0]- \Pr[r_s(x) \cdot v_s(x,y) \ge 0] \big|.
\end{eqnarray*} 
Applying Corollary~\ref{corr:bounds} to bound all the terms, we obtain that 
$$
 \big|\Pr[p_s(x) \ge 0 \wedge u_s(x,y) \ge 0]- \Pr[r_s(x) \ge 0 \wedge v_s(x,y) \ge 0] \big| \le \frac{3 \xi}{2}. 
$$
Plugging this bound back to (\ref{eq:noise-junta-1}), we obtain that 
$$
\big| \mathbf{E}[\langle f, P_t f \rangle] - \mathbf{E}[\langle \fh, P_t \fh \rangle] \big| \le   2\Pr[x \in \mathsf{Collision}(f)] +2 \Pr[x \in \mathsf{Collision}(\fh)] + \frac{3k\xi}{2}. 
$$

\end{proof}

We are now in a position to finish the proof of Theorem~\ref{thm:reduction} and show that the construction $\fh$ indeed satisfies the required properties. Note that by construction, $\fh$ is a degree-$d$ PTF over $n_0$ variables where as we have said before, once $\beta (\cdot)$ is explicitly selected, $n_0 = n_0(k, \epsilon, d)$ is an explicit function. Further, it is easy to see that there exists $\beta(\cdot)$ (we just need to be sufficiently fast decreasing) which satisfies the requirement of Condition~\ref{cond:enforce-beta}. 

We first bound $\Vert \mathbf{E}[\fh] - \boldsymbol{\mu}\Vert_1 = \Vert \mathbf{E}[\fh] - \mathbf{E}[f] \Vert_1$. To do this, note that 
\begin{eqnarray*}
\Vert \mathbf{E}[\fh] - \mathbf{E}[f] \Vert_1 &\le& \sum_{s=1}^k |\Pr[\tilde{p}_s(x) \ge 0] - \tilde{r}_s(x) \ge 0] | + \Pr[x \in \mathsf{Collision}(f)]  + \Pr[x \in \mathsf{Collision}(\fh)] \\ &\le& k \cdot \xi + \Pr[ x \in \mathsf{Collision}(f)] + (k^2 +1) \cdot \Pr[ x \in \mathsf{Collision}(f)] + k \cdot \xi + \frac{3 k^2 \cdot \xi}{2}. 
\end{eqnarray*}
Here the last inequality follows by applying the first item of Corollary~\ref{corr:bounds} and Claim~\ref{clm:collision-upper}. Plugging in the value of $\xi = \frac{\epsilon}{20k^2}$ and the upper bound on $\Pr[x \in \mathsf{Collision}(f)]$, we obtain that $\Vert \mathbf{E}[\fh] - \mathbf{E}[f] \Vert_1 \le \epsilon$. 

For the second part, note that applying Lemma~\ref{lem:noise-stab}, we obtain
\begin{eqnarray*}
\langle \fh, P_t \fh \rangle \ge \langle f, P_t f \rangle - 2\Pr[x \in \mathsf{Collision}(f)] -2 \Pr[x \in \mathsf{Collision}(\fh)] - \frac{3k\xi}{2}. 
\end{eqnarray*}
Again plugging in the value $\xi = \frac{\epsilon}{20k^2}$ and the upper bound on $\Pr[x \in \mathsf{Collision}(f)]$, we obtain that 
$$
\langle \fh, P_t \fh \rangle \ge \langle f, P_t f \rangle - \epsilon.
$$

\bibliographystyle{plain}
\bibliography{allrefs,all,mossel}

\appendix

\section{Existence of a rounding threshold}

\begin{proofof}{Lemma~\ref{lem:rounding-existence}}
  We will prove the following more general statement: given $z \in \R^k$ and $i \in [k]$, let $A_i(z) \subset \Delta_k$ be the set
  $\{x \in \Delta_k: x_i - z_i = \max_j x_j - z_j\}$. For every probability measure $\mu$ on $\Delta_k$ and every $q \in \Delta_k$,
  there exists $z \in \R^k$ such that $\mu(\bigcup_{i \in I} A_i(z)) \ge \sum_{i \in I} q_i$ for every $I \subset [k]$.
  To obtain the original statement from this one, let $\mu$ be the distribution of $f$ and let $q = \E[f]$.
  By the preceding claim about $\mu$, we may take $z \in \R^k$ such that
  $\mu(\bigcup_{i \in I} A_i(z)) \ge \sum_{i \in I} q_i$ for every $I \subset [k]$. For
  $I \subset [k]$, let $B_I$ be the set of $x \in \R^n$ such that $\{i: f(x_i) - z_i = \max_j f(x_j) - z_j\} = I$. Then $\{B_I: I \subset [k]\}$ is a
  partition of $\R^n$. Moreover, if we define $w(I) = \gamma_n(B_I)$ then
  \[
    \sum_{J: J \cap I \ne \emptyset} w(J) = \gamma_n\Big(\bigcup_{J: J \cap I \ne \emptyset} B_J\Big) = \mu\Big(\bigcup_{i \in I} A_i(z)\Big) \ge \sum_{i \in I} q_i.
  \]
  Now we construct a bipartite graph with $U = [k]$ and $V = \{I: I \subset [k]\}$. The pair $(i, I)$
  is an edge if $i \in I$. We define $w$ on $V$ as above, and we define $w$ on $U$ by $w(i) = p_i$.
  According to the displayed equation above, the condition of Theorem~\ref{thm:hall}
  holds; hence, there exists $p: V \to \Delta_k$ such that $p_i(I) > 0$ only when $i \in I$, and
  such that $\sum_{I \ni i} p_i(I) \gamma_n(B_I) = q_i$. Now we will define $g$: for each 
  $I \subset [k]$, partition $B_I$ arbitrarily into sets $\{B_{I,i}: i \in I\}$, where
  $\gamma_n(B_{I,i}) = p_i(I) \gamma_n(B_I)$. This may be done because $\gamma_n$ has no atoms.
  Finally, define $g$ to be $e_i$ on every $B_{I,i}$.
  Then the condition $\sum_{I \ni i} p_i(I) \gamma_n(B_I) = q_i$ ensures that $\gamma_n(\{g = e_i\}) = q_i$.
  Moreover, note that $g(x) = e_i$ implies that $x \in B_I$ for some $I$, which implies
  that $f_i(x) - z_i = \max_j (f_j(x) - z_j)$. This completes the construction of $g$.

  It remains to prove the claim about $\mu$.
  We will assume initially that $\mu$ is absolutely continuous with respect to the
  uniform measure on $\Delta_k$. As a consequence, the function $\psi: \Delta_k \to \R^k$ defined by
  $\psi_i(z) = \mu(A_i(-kz))$
  is continuous. Moreover, the image of $\psi$ is in $\Delta_k$, and we need prove that it is all of $\Delta_k$.
  
  For $I \subsetneq \{1, \dots, k\}$, let $F_I = \{x \in \Delta_k: x_i = 0 \text{ for all } i \in I\}$. Note that
  every face of $\Delta_k$ is of the form $F_I$ for some $I \subsetneq \{1, \dots, k\}$.
  Next, we claim that $\psi$ maps $F_I$ into $F_I$ for every $I$. Indeed, if $z \in F_I$ then there is at least one $j \not \in I$
  such that $z_j \ge 1/k$. For this $j$ and any $i \in I$, if $x$ is in the interior of $\Delta_k$ then
  \[
    x_i + k z_i = x_i < x_j + k z_j.
  \]
  It follows that $A_i(-kz)$ does not intersect the interior
  of $\Delta_k$; hence, $\psi_i(z) = 0$. Since this holds for all $i \in I$, $\psi(z) \in F_I$.
  By Sperner's lemma, any map from $\Delta_k$ into itself that leaves all faces invariant must be onto, and so
  $\psi$ is onto, as claimed.

  To complete the proof, we must eliminate the assumption that $\mu$ is absolutely continuous with
  respect to the uniform measure. For an arbitrary probability measure $\mu$, let $\mu_n$ be a sequence of
  absolutely continuous probability measures that converge to $\mu$ in distribution. Define $\psi_n: \Delta_k \to \Delta_k$
  by $\psi_n(z) = (\mu_n(A_1(z)), \dots, \mu_n(A_k(z)))$.
  By the previous argument, for every $n$ there is some $z_n \in \Delta_k$ such that $\psi(z_n) = q$. Since
  $\Delta_k$ is compact, we may pass to a subsequence and thereby assume that $z_n$ converges to some limit $z_\infty$.
  Note that
  \[
    \bigcup_{i \in I} A_i(z_\infty) = \bigcap_{n=1}^\infty \Big(\bigcup_{i \in I} \bigcup_{m \ge n} A_i(z_m) \cup A_i(z_\infty)\Big).
  \]
  Moreover, $\bigcup_{m \ge n} A_i(z_n) \cap A_i(z_\infty)$ is closed for every $n$, since $z_\infty$ is the only
  limit point of the sequence $z_n$. It follows that
  \[
    \mu\Big(\bigcup_{i \in I} A_i(z_\infty)\Big)
    = \lim_{n \to \infty} \mu\Big(\bigcup_{i \in I} \bigcup_{m \ge n} A_i(z_m) \cup A_i(z_\infty)\Big)
    \ge \lim_{n \to \infty} \lim_{\ell \to \infty} \mu_\ell\Big(\bigcup_{i \in I} \bigcup_{m \ge n} A_i(z_m) \cup A_i(z_\infty)\Big)
    \ge \sum_i q_i,
  \]
  since $\sum_{i \in I} \mu_\ell(A_i(z_\ell)) = \sum_{i \in I} q_i$ for every $I$ and $\ell$
  (using the fact that $\mu$ has a density, and so it assigns no mass to $A_i(z_\ell) \cap A_j(z_\ell)$).
\end{proofof}

\section{Bounds on variance of product of low-degree polynomials}

Let $p, q: \mathbb{R}^n \rightarrow \mathbb{R}$ be degree $d_1$ and $d_2$ polynomial respectively defined as
\[
p = \sum_{r=0}^{d_1} I_r(f_r) \ \textrm{ and } \ q = \sum_{r=0}^{d_2} I_r(g_r).
\]
We will use Ito's multiplication formula to establish upper and lower bounds on the variance of the product polynomial $p \cdot q$. 
\begin{proposition}\label{prop:upper-bound} 
Let $p, q$ be as defined above such that $\mathbf{E}[q]=0$. Then, 
\[
\mathbf{E}[(p \cdot q)^2]  \le 9^d \cdot \mathbf{E}[p^2] \cdot \mathbf{E}[q^2]. 
\]
\end{proposition} 
\begin{proof}
By applying Cauchy-Schwarz inequality, we have 
$ \mathbf{E}[(p \cdot q)^2]  \le \sqrt{\mathbf{E}[p^4] \cdot \mathbf{E}[q^4]}$. It follows from Theorem~\ref{thm:hyper} that $\mathbf{E}[p^4] \le 9^d (\mathbf{E}[p^2])^2$ and $\mathbf{E}[q^4] \le 9^d (\mathbf{E}[q^2])^2$  (see \cite{ODonnell:book} for a short inductive proof of this statement). Combining these observations yields the proof.  

\end{proof}

The next proposition proves a lower bound on the variance of $p \cdot q$ in terms of the Frobenius norm of its highest degree component. 
\begin{proposition}\label{prop:variance-bound}
\[
\mathsf{Var}(p \cdot q) \ge {\Vert f_{d_1} \Vert_F^2 \cdot \Vert g_{d_2} \Vert_F^2}.
\] 
\end{proposition}
\begin{proof} 
\begin{eqnarray*}
p \cdot q &=& \sum_{r_1=0}^{d_1} \sum_{r_2=0}^{d_2} I_{r_1}(f_{r_1}) \cdot I_{r_2}(g_{r_2})  \\
&=& \sum_{r_1=0}^{d_1} \sum_{r_2=0}^{d_2} \sum_{r=0}^{r_1 \wedge r_2} r! \cdot \binom{r_1}{r} \cdot \binom{r_2}{r} \cdot \frac{\sqrt{(r_1+r_2 - 2r)!}}{\sqrt{r_1!}\sqrt{r_2!}} \cdot I_{r_1+r_2 - 2r} (f_{r_1} \widetilde{\otimes_r} g_{r_2}). 
\end{eqnarray*}
The second equality uses Ito's multiplication formula (Proposition~\ref{prop:Ito}). Since the term with $r=0$, $r_1= d_1$ and $r_2=d_2$ does not get canceled with any other term, we have
\[
\mathsf{Var}(p \cdot q) \ge \mathsf{Var} \bigg( \frac{\sqrt{(d_1+d_2 )!}}{\sqrt{d_1!}\sqrt{d_2!}}  \cdot I_{d_1+d_2} (f_{d_1} \widetilde{\otimes_0} g_{d_2}) \bigg). 
\]
Next, we use the following fact proven by Neuberger~\cite{Neuberger:74}. 
\begin{fact}~\label{fac:Neu}[Neuberger]
Let $A \in \mathcal{H}^{\odot d_1}$ and $B \in \mathcal{H}^{\odot d_2}$. Then, 
$$
\Vert A \widetilde{\otimes_0} B \Vert_F^2 \ge \frac{1}{\binom{d_1+d_2}{d_1}} \cdot \Vert A \Vert_F^2 \cdot \Vert B \Vert_F^2.
$$
\end{fact} 
\[
\mathsf{Var} \bigg( \frac{\sqrt{(d_1+d_2 )!}}{\sqrt{d_1!}\sqrt{d_2!}}  \cdot I_{d_1+d_2} (f_{d_1} \widetilde{\otimes_0} g_{d_2}) \bigg) \geq \frac{{(d_1+d_2 )!}}{{d_1!} \cdot {d_2!}} \Vert f_{d_1} \widetilde{\otimes_0} g_{d_2} \Vert_F^2 
\]
Using Fact~\ref{fac:Neu}, we get 
\[
\mathsf{Var} \bigg( \frac{\sqrt{(d_1+d_2 )!}}{\sqrt{d_1!}\sqrt{d_2!}}  \cdot I_{d_1+d_2} (f_{d_1} \widetilde{\otimes_0} g_{d_2}) \bigg) \geq {\Vert f_{d_1} \Vert_F^2 \cdot \Vert g_{d_2} \Vert_F^2}.
\]

\end{proof}
\begin{lemma}~\label{lemma:variance-product-lower}
Let $p, q: \mathbb{R}^n \rightarrow \mathbb{R}$ be degree $d$ polynomials such that $\mathsf{Var}(p) = \mathsf{Var}(q)=1$. Let $T = \max \{\mathbf{E}[p^2], \mathbf{E}[q^2]\}$ and 
$L=4 \cdot T \cdot 9^{d+1} \cdot (d+1)^2$. 
Then
\[
\mathsf{Var}(p \cdot q) \ge \frac{L}{2} \cdot (L)^{2^{-2d}}. 
\] 
\end{lemma}
\begin{proof}
Let us express $p$ and $q$ in terms of iterated Ito integrals. 
\[
p = \sum_{r=0}^{d} I_r(f_r) \ \textrm{ and } \ q = \sum_{r=0}^{d} I_r(g_r).
\]
Let $\alpha = \mathbf{E}[p^2]$ and $\beta= \mathbf{E}[q^2]$. 
Let $\Gamma : \mathbb{N} \rightarrow [0,1]$ be a function (which we shall fix later). Consider the following iterative process:
\begin{itemize}
\item Start with $i=d$ and $j=d$. If $\Vert f_i \Vert_F \cdot \Vert g_j \Vert_F \ge \Gamma(i+j)$, then stop the process. 
\item If $\Vert f_i \Vert_F \le \sqrt{\Gamma(i+j)}$, then $i \leftarrow i-1$, else $j \leftarrow  j-1$. 
\item If any of $i$ or $j$ reaches zero, terminate the process. 
\end{itemize}
Recall that  $T = \max \{\alpha, \beta\} $ and $L = 4 \cdot T \cdot 9^{d+1} \cdot (d+1)^2$. We let $\Gamma (x) = L \cdot L^{-2^{x}}$. First, we claim that for $\Gamma(\cdot)$ as chosen here, the iterative process above terminates with $(i,j)$ such that neither of them is zero. Towards a contradiction, assume that the above process terminates with $i=0$ and $j>0$. Further, for $d \le \ell \le 1$, when the value of $i$ drops from $\ell$ to $\ell-1$, let the value of $i+j$ be $\kappa_\ell$. Note that $\kappa_1 < \ldots < \kappa_d$ and $\kappa_1 \ge 2$.  Thus, if the process terminates with $i=0$, 
then 
\[
\mathsf{Var}(p)= \sum_{\ell=1}^d \Vert f_\ell \Vert_F^2  \le \sum_{\ell=1}^d  \Gamma(\kappa_\ell) \le 2 \cdot \Gamma(\kappa_1)  <1
\]
This results in a contradiction which means that both $i$ and $j$ must be non-zero at the end of the process.  Here the penultimate inequality uses that for our choice of $\Gamma(\cdot)$, $\Gamma(x+1) \le \Gamma(x)/2$ and the last inequality uses that $\Gamma(\kappa_1) <1/2$. 

Let us assume that $(i_0, j_0)$ is the pair returned by the above iterative process and define $\tilde{p} = \sum_{r=0}^{i_0} I_r(f_r)$ and $\tilde{q} = \sum_{r=0}^{j_0} I_r(g_r)$. Applying Proposition~\ref{prop:variance-bound}, we have
$
\mathsf{Var}(\tilde{p} \cdot \tilde{q}) \ge {\Gamma^2(i_0+j_0)}. 
$
Observe that $p \cdot q = \tilde{p} \cdot \tilde{q} + (p - \tilde{p}) \cdot q + \tilde{p} \cdot (q-\tilde{q}).$ Thus, it follows that $\mathsf{Var}(p \cdot q - \tilde{p} \cdot \tilde{q}) = \mathsf{Var}((p - \tilde{p}) \cdot q + \tilde{p} \cdot (q-\tilde{q}))$. Applying Jensen's inequality, we get
\[
\sqrt{\mathsf{Var}(p \cdot q)} \ge \sqrt{\mathsf{Var}(\tilde{p} \cdot \tilde{q})} - \sqrt{\mathsf{Var}(\tilde{p} \cdot {(q-\tilde{q})})} - \sqrt{\mathsf{Var}(q \cdot {(p-\tilde{p})})}.\]
Hence, it suffices to bound $\mathsf{Var}(\tilde{p} \cdot {(q-\tilde{q})})$ and $\mathsf{Var}(q \cdot {(p-\tilde{p})})$. To do this, notice that by definition of the process, 
\[
\mathsf{Var}(q-\tilde{q}), \ \mathsf{Var}(p-\tilde{p}) < \sum_{\ell = i_0 +j_0+1}^{d} \Gamma(\ell) <  2 \cdot \Gamma(i_0+j_0+1). 
\]
Applying Proposition~\ref{prop:upper-bound}, 
\[
\mathsf{Var}(q \cdot {(p-\tilde{p})}), \  \mathsf{Var}(q \cdot {(p-\tilde{p})})\le 2 \cdot \Gamma(i_0+j_0+1) \cdot d \cdot (d+1) \cdot 8^d \cdot T. 
\]
Thus, 
\[
\sqrt{\mathsf{Var}(p \cdot q)} \ge {\Gamma(i_0+j_0)} - 2 \sqrt{2} \cdot (d+1) \cdot (2\sqrt{2})^d \cdot \sqrt{T} \cdot \sqrt{\Gamma(i_0+j_0+1)} .
\]
However, note that for any $x \ge 0$, $\sqrt{\Gamma(x+1) } = \Gamma(x) / \sqrt{L}$. Plugging the value of $L$, we obtain that 
\[
\sqrt{\mathsf{Var}(p \cdot q)}  \ge {\Gamma(i_0+j_0)}/2 \ge \frac{L}{2} \cdot L^{2^{-2d}}.
\]
\end{proof}

\section{Transformation of PTFs to multilinear PTFs}~\label{app:multilinear}
In this section, we prove Lemma~\ref{lem:multilinear}. To recall, this lemma says that given any multivariate PTF, one can get a multilinear multivariate PTF which has the same partition sizes and noise stability and satisfies the balanced condition. 
\begin{lemma} 
Let $f: \mathbb{R}^n \rightarrow [k]$ be a degree-$d$, $(d,\epsilon)$-balanced PTF. Then, for any $\epsilon>0$, there exists a degree-$d$, $(d,2\epsilon)$-balanced multilinear PTF $f_{\mathsf{multi}}: \mathbb{R}^\ell \rightarrow \mathbb{R}$ such that 
\begin{itemize}
\item $\Vert \mathbf{E}_{x \sim \gamma_n} [f(x) ] - \mathbf{E}_{x \sim \gamma_\ell} [f_{\mathsf{multi}}(x)  ] \Vert_1 \le \epsilon$, 
\item $| \mathbf{E}_{x \sim \gamma_n} [ \langle f, P_t  f \rangle ] - \mathbf{E}_{x \sim \gamma_{\ell}} [\langle f_{\mathsf{multi}}, P_t f_{\mathsf{multi}}\rangle ] | \le \epsilon$.
\item $\Pr_{x \sim \gamma_{\ell}} [x \in \mathsf{Collision}(f_{\mathsf{multi}})] \le \Pr_{x \sim \gamma_{n}} [x \in \mathsf{Collision}(f)]  + \epsilon$.
\end{itemize}
Here $\ell = n \cdot (k^2/d \epsilon)^{3d} \cdot d^2$. 
\end{lemma}
\begin{proof}
Let us assume that $f = \mathsf{PTF}(p^{(1)}, \ldots, p^{(k)})$. Assume that $p^{(i)} = \sum_{q=0}^d I_q (f_q^{(i)})$. Choose some $T \in \mathbb{N}$ which will be fixed later.
Let us consider a collection of $\ell = n \cdot T$ variables denoted by  $\{x_{i,j}\}_{1 \le i \le n, 1 \le j \le T}$. Let $\Phi$ be a mapping (defined below) which maps $\{x_i\}_{i \in \mathbb{N}}$ to linear forms over 
$$
\Phi (x_i) = \frac{x_{i,1} + \ldots + x_{i,T}}{\sqrt{T}}. 
$$
 Let the polynomial $r^{(i)} : \mathbb{R}^{\ell} \rightarrow \mathbb{R}$ be defined as  the image of polynomial $p^{(i)}$ under the mapping $\Phi$. 
In other words, if  $f_q^{(i)} = \sum_{1 \le j_1, \ldots, j_i \le n} \alpha^{(q)}_{j_1, \ldots, j_i} e_{j_1} \otimes \ldots \otimes e_{j_i} $, then we let
$r^{(i)} = \sum_{q=0}^d I_q(g^{(i)}_q)$ where $g^{(i)}_q$ is defined as
$$
g_{q}^{(i)}= \sum_{1 \le j_1, \ldots, j_q \le n} \alpha^{(i)}_{j_1, \ldots, j_q} \frac{e_{j_1,1} + \ldots + e_{j_1,T}}{\sqrt{T}} \otimes \ldots \otimes \frac{e_{j_q,1} + \ldots + e_{j_q,T}}{\sqrt{T}} 
$$
Let us define $f' : \mathbb{R}^{\ell} \rightarrow [k]$ as $f'= \mathsf{PTF}(r^{(1)}, \ldots, r^{(k)})$. As the distribution of $(\Phi(x_1), \ldots, \Phi(x_n))$ is the same as $(x_1, \ldots, x_n)$,  we have that for $1\le i \le k$,
\begin{eqnarray}
\Pr_{x \sim \gamma_n}[f(x)=i] &=&\Pr_{x \sim \gamma_\ell}[f'(x)=i], \nonumber \\ 
\mathbf{E}_{x \sim \gamma_n} [\langle f, T_{\rho} f \rangle]&=&\mathbf{E}_{x \sim \gamma_\ell} [\langle f', T_{\rho} f' \rangle], \nonumber \\
\textrm{and} \ \Pr_{x \sim \gamma_n} [x \in \mathsf{Collision}(f)] &=& \Pr_{x \sim \gamma_\ell} [x \in \mathsf{Collision}(f')]. \label{eq:equality} 
\end{eqnarray}
Let us now consider the multilinear version of $g_{q}^{(i)}$ which is obtained by deleting any diagonal term and denote it by $h_{q}^{(i)}$. Define the polynomials $w^{(1)}, \ldots, w^{(k)}$  as 
$
w^{(i)} = \sum_{q=0}^d I_q(h_{q}^{(i)}). 
$
Let us define $f_{\mathsf{multi}} : \mathbb{R}^{\ell} \rightarrow [k]$ as $f_{\mathsf{multi}} = \mathsf{PTF} (w^{(1)}, \ldots, w^{(k)})$. 
We will now prove that for a suitable choice of $T$, $f_{\mathsf{multi}}$ has the desired properties.  First, by definition, it follows that $f_{\mathsf{multi}}$ is a degree-$d$ multilinear PTF. 
Next,  for $1 \le j_1, \ldots, j_q \le n$, let us define $\mathcal{S}_{j_1, \ldots, j_q}$ as 
$$
\mathcal{S}_{j_1, \ldots, j_q} = \{(s_1, \ldots, s_q) \in [T]^q : | \cup_{b=1}^q (j_b, s_b) | <  q\}. 
$$
With this definition, observe that
\[
r^{(i)} - w^{(i)}  =  \sum_{q=0}^d \sum_{1 \le j_1, \ldots, j_q \le n} \sum_{(s_1, \ldots, s_q) \in \mathcal{S}_{j_1, \ldots, j_q} } \alpha^{(i)}_{j_1, \ldots, j_q}   \cdot \frac{1}{T^{s/2}} \cdot I_q \big(\otimes_{u=1}^q e_{j_u, s_u} \big)  .
\]
%
%
As the vectors $\{\otimes_{u=1}^T e_{j_u, s_u} \}_{j_u \in [n], s_u \in [T]}$ are orthonormal, using Proposition~\ref{prop:4}, we get that
$$
\mathsf{Var}[r^{(i)} -  w^{(i)} ] = \sum_{q=0}^d \sum_{1 \le j_1, \ldots, j_q \le n}  (\alpha^{(i)}_{j_1, \ldots, j_q})^2 \cdot \frac{|\mathcal{S}_{j_1, \ldots, j_q}|}{T^s} 
$$
On the other hand, 
$$
\mathsf{Var}[r^{(i)}] = \sum_{q=0}^d \sum_{1 \le j_1, \ldots, j_q \le n}  (\alpha^{(i)}_{j_1, \ldots, j_q})^2
$$
It is easy to see that we can uniformly bound $|\mathcal{S}_{j_1, \ldots, j_q}| $ by $\frac{T^s \cdot q^2 }{ T} \le \frac{T^s \cdot d^2}{T}$.  Thus, 
$$
\mathsf{Var}[r^{(i)} - w^{(i)} ] \leq \mathsf{Var}[r^{(i)} ] \cdot \frac{d^2}{T}. 
$$
By applying Theorem~\ref{thm:combine-hyper}, we get that if $T \ge (\tau/d)^{3d} \cdot d^2$, then 
$$
\Pr_{x \sim \gamma_\ell} [\mathsf{sign} (r^{(i)})  \not = \mathsf{sign} (w^{(i)})] \le \tau. 
$$
From this, using a union bound,  it follows that 
$$
\Pr_{x \sim \gamma_\ell} [f_{\mathsf{multi}}(x) \not = f'(x)]   \le k \cdot \tau. 
$$
Likewise, 
$$
\Pr_{x \sim \gamma_\ell} [x \in \mathsf{Collision} (f_{\mathsf{multi}})] \leq \Pr_{x \sim \gamma_\ell} [x \in \mathsf{Collision} (f')] + k \cdot \tau. 
$$
Combining the above equations with (\ref{eq:equality}) and setting $\tau = \epsilon/k$, gives us the claim. Also, by an application of Cauchy-Schwarz inequality, we have that 
\[
\sqrt{\mathsf{Var}(w^{(i)})} \ge \sqrt{\mathsf{Var}(r^{(i)})}  \cdot \bigg( 1- \frac{d}{\sqrt{T}}\bigg) \ge \sqrt{\mathsf{Var}(r^{(i)})}  \cdot \big( 1- \epsilon\big)
\]
Hence, by rescaling the polynomials $w^{(i)}$, it is easy to see that we can ensure that $f_{\mathsf{multi}}$ is a $(d,2\epsilon)$-balanced PTF. 

\end{proof}

\section{Mollification for the sign function}
In this section, we will prove Theorem~\ref{thm:mollification}. For the convenience of the reader, we recall the setting of this theorem. Namely, there are degree-$d$ polynomials $A_1, \ldots, A_m: \mathbb{R}^{n} \rightarrow \mathbb{R}$ of variance $1$ (and mean $0$). Given  parameters $S$,  $\epsilon$ and $\eta$ $\in \mathbb{R}^+$ (whose use will be made clear soon), we define additional dependent parameters $B$, $\delta$ and $c$ as follows: 
\begin{equation}\label{eq:param2} 
B = \Omega\bigg( \ln  \frac{md}{\epsilon} \bigg)^{d/2} ;  \ \delta = \bigg( \frac{\epsilon \cdot \eta^{1/2d}}{d \cdot B \cdot \sqrt{m} \cdot S^{1/2d}} \bigg)^d; \   c = \frac{m}{\delta \cdot \sqrt{\epsilon}}
\end{equation}
\begin{theorem}\label{thm:mollification}
Let $\phi: \mathbb{R}^{m} \rightarrow \mathbb{R}$ be a degree-$d$ polynomial such that the sum of squares of its (non-constant) coefficients is bounded by $S$. 
There is a function $\tilde{g}_c: \mathbb{R}^m \rightarrow [0,1]$ such that for $A_1, \ldots, A_m$ as described above, if $\mathsf{Var}(\phi(A_1, \ldots, A_m)) \ge \eta$, then 
\[
\mathop{\mathbf{E}}_{x \sim \gamma_n} [|\mathsf{sign}(\phi(A_1, \ldots, A_m)) - \tilde{g}_c(A_1, \ldots, A_m)|] \le O(\epsilon). 
\]
Further,  $\Vert \tilde{g}_c^{(1)} \Vert_{\infty} \le c$, $\Vert \tilde{g}_c^{(2)} \Vert_{\infty} \le 4c^2$.
\end{theorem}
~ As this theorem is essentially proven in \cite{DS14}, we will only give a sketch of the proof here. To do this, we recall the following lemma from \cite{DS14} (appears as Claim 49 in the paper). 
\begin{claim}\label{clm:Lipschitz}
Let $x, y \in \mathbb{R}^m$  such that $\Vert x \Vert_\infty \le B$. If $\Vert x- y \Vert_2 \le \delta \le B$, then $|\phi(x) - \phi(y)| \le d \cdot (2B)^{d} \cdot \delta \cdot \sqrt{S} \cdot m^{d/2}$. 
\end{claim} 
The next claim we need is the following. Assume that $g: \mathbb{R}^m \rightarrow \{0,1\}$ defined as $g (x)  = \mathsf{sign} (\phi(x))$. Let $R$ denote the  set $g^{-1}(1)$ and $\partial R$ denote its boundary. Applying quite standard mollification based techniques, the authors in \cite{DS14} proved that that for every $c>0$, there is a function $\tilde{g}_c: \mathbb{R}^m \rightarrow [0,1]$ such that $\Vert \tilde{g}_c^{(1)} \Vert_{\infty} \le c$, $\Vert \tilde{g}_c^{(2)} \Vert_{\infty} \le 4c^2$ such that
\begin{equation}\label{eq:diffmoli} 
\big| g_c(x) - \tilde{g}_c(x) \big| \leq \min \bigg\{ 1, \frac{m^2}{c^2 \cdot \mathsf{dist}^2(x, \partial R)}\bigg\}. 
\end{equation}
In the above $\mathsf{dist} (x, \partial R)$  denotes the Euclidean distance of $x$ from the boundary of $R$.  We now get back to our proof of Theorem~\ref{thm:mollification}. To prove this, we first make the following observation (which is immediate from Claim~\ref{clm:Lipschitz}). 
\begin{obs}
Let $x \in \mathbb{R}^m $ and $\Vert x \Vert_{\infty} \le B$. If $\delta \leq B$ and $|\phi(x)| \ge d \cdot (2B)^{d} \cdot \delta \cdot \sqrt{S} \cdot m^{d/2}$, then $\mathsf{dist}(x, \partial R)\ge \delta$.
\end{obs}
Applying (\ref{eq:diffmoli}), if $\mathsf{dist}(x, \partial R)\ge \delta$, then $| g_c(x) - \tilde{g}_c(x) | \le m^2 /(c^2 \cdot \delta^2)$. Thus, $\mathbf{E}_{x \sim \gamma_n} [|\mathsf{sign}(\phi(A_1, \ldots, A_m)) - \tilde{g}_c(A_1, \ldots, A_m)|]$ is bounded by 
\begin{eqnarray*}
\Pr_{x \sim \gamma_n} \big[\sup_{1 \le j \leq m} |A_j(x)| > B\big] + \Pr_{x \sim \gamma_n} \big[ |\phi(A_1(x), \ldots, A_m(x))| \le d \cdot (2B)^{d} \cdot \delta \cdot \sqrt{S} \cdot m^{d/2}\big] + \frac{m^2}{c^2 \delta^2}. 
\end{eqnarray*}
We now bound the above terms. To bound the first term, note that each of the $A_j(x)$ is a degree-$d$ polynomial with mean $0$ and variance $1$. Thus, applying Theorem~\ref{thm:hyper}, we have 
$$
\Pr_{x \sim \gamma_n} \big[\sup_{1 \le j \leq m} |A_j(x)| > B\big]  \le m \cdot d \cdot \exp \big( -B^{2/d}\big).
$$
To bound the second term, note that $\mathsf{Var}(\phi(A_1, \ldots, A_m)) \ge \eta$. Then, 
\[
\Pr_{x \sim \gamma_n} \big[ |\phi(A_1(x), \ldots, A_m(x))| \le d \cdot (2B)^{d/2} \cdot \delta \cdot \sqrt{S} \cdot m^{d/2}\big] \le \frac{d \cdot (2B) \cdot \delta^{1/d} \cdot S^{1/2d} \cdot \sqrt{m}}{\eta^{1/2d}}.
\]
Now, observe that trivially, the requirement of $\delta \leq B$ is trivially satisfied. With the setting of parameters in (\ref{eq:param2}), we see that all the terms are $O(\epsilon)$. This finishes the proof of Theorem~\ref{thm:mollification}.

\section{Construction of polynomial families with matching correlation and eigenregularity}~\label{app:corrjunta}
In this section, we will prove Lemma~\ref{lem:junta-construct}. We will restate it here for the convenience of the reader. 
\begin{lemma*}
It is possible to construct a family $\mathcal{R}$ meeting the requirements of Condition~\ref{cond:junta} with $n_0= \mathsf{poly}(d, \mathrm{Num}, \beta(\mathrm{Num} + \mathrm{Coeff}))$ variables. 
\end{lemma*}
\begin{proof}
To prove this lemma, for $1 \le i \le d$, let $\mathcal{P}_i = \{\Inner(p_{s,q})_\ell\}_{\ell = 1,\dots, \num({s,q})} \cap  \mathcal{W}^i$. Let $m_i$ denote the size of $\mathcal{P}_i$. We will construct $\mathcal{R}$ by constructing the corresponding set $\mathcal{R}_i$ for each $1 \le i \le d$. Note that  individually constructing $\mathcal{R}_i $ for each $1 \le i \le d$ suffices for our construction. This is because if $r_1 \in \mathcal{W}^i$, $r_2 \in \mathcal{W}^j$ and $i \not =j$, then $\mathbf{E}[r_1(x) \cdot r_2(x)]=0$. 

For convenience of notation, let us enumerate the elements of $\mathcal{P}_i$ as $\{p_1, \ldots, p_{m_i}\}$. We will now construct $\mathcal{R}_i = \{r_1, \ldots, r_{m_i}\}$. To do this, 
let $p_j = I_i(h_j)$ where $h_j \in \mathcal{H}^{\odot i}$ (here $\mathcal{H} = \mathbb{R}^n$). By standard linear algebra argument, we get that given orthonormal basis $\mathcal{B}$
of $\mathcal{H}^{\odot i}$ (with an ordering among its elements), there are vectors $v_1, \ldots, v_{m_i}  \in  \mathcal{H}^{\odot i}$ such that \begin{itemize}
\item[(a)] $\{v_1, \ldots, v_{m_i}\} $  lies in the linear span of  the first $m_i$ elements of $\mathcal{B}$, 
\item[(b)] $\langle v_j, v_{j'} \rangle = \langle h_j , h_{j'} \rangle$. 
\end{itemize}
We now instantiate the basis $\mathcal{B}$. Consider the standard basis $\{e_1, \ldots, e_n\}$ for $\mathcal{H}$ 
and using Proposition~\ref{prop:basis}, we obtain that there is a basis for $\mathcal{H}^{\odot i}$ where each basis element is just a function of at most $i$ of the elements of $\{e_1, \ldots, e_n\}$.  Let $T = i \cdot m_i$ and let $\mathcal{H}_1 = \mathbb{R}^T$. Thus, we get that there are elements $v_1, \ldots, v_{m_i} \in \mathcal{H}^{ \odot i}_1$ such that 
for $1 \le j_1, j_2 \le m_i$, $\langle v_{j_1}, v_{j_2} \rangle = \langle h_{j_1}, h_{j_2} \rangle$. For $1 \le j \le m_i$, we define $q_j = I_i(v_j)$. Note that for all $1 \le j \le m_i$, $q_j :\mathbb{R}^T \rightarrow \mathbb{R}$. 
Let $\delta= \beta(\mathrm{Num} + \mathrm{Coeff})$, $\kappa=\lceil 1/\delta^2 \rceil$ and $n_0 = T \cdot \kappa$. Finally, for $1 \le j \le m_i$, we define $r_j: \mathbb{R}^{n_0} \rightarrow \mathbb{R}$ as follows: Divide $x \in \mathbb{R}^{n_0}$ into $\kappa$ blocks of size $T$ each (call the blocks $X_1, \ldots, X_{\kappa}$) and define
\[
r_j = \frac{1}{\sqrt{\kappa}} \cdot \big( q_{j}(X_1) + \ldots  + q_{j}(X_\kappa)\big). 
\]
We now verify the properties of the polynomials $\{r_j \}_{1 \le j \le m_i}$. First, note that $n_0 = \mathsf{poly}(d, \mathrm{Num}, \beta(\mathrm{Num} + \mathrm{Coeff}))$. Secondly, 
\[
\mathbf{E}[r_{j_1} \cdot r_{j_2}] = \sum_{j=1}^\kappa \frac{1}{\kappa} \cdot  \mathbf{E}[q_{j_1}(X_j) \cdot q_{j_2}(X_j)] = \sum_{j=1}^\kappa \frac{1}{\kappa} \cdot \langle v_{j_1}, v_{j_2} \rangle=\sum_{j=1}^\kappa \frac{1}{\kappa} \cdot \langle h_{j_1}, h_{j_2} \rangle = \mathbf{E} [p_{j_1} \cdot p_{j_2}].
\]
The first equality relies on the observation that $q_{j_1}, q_{j_2} \in \mathcal{W}^{i}$ for $i>1$ and thus $\mathbf{E}[q_{j_1}(X_j) \cdot q_{j_2}(X_\ell)]=0$ if $j \not =\ell$. The second and fourth equality uses Proposition~\ref{prop:4}. Finally, using Fact~\ref{fact:eigen-2}, we get that the polynomials $\{r_{j}\}_{1 \le j \le m_i}$ are $\delta = \beta(\mathrm{Num} + \mathrm{Coeff})$-eigenregular.

\end{proof}
\section{Expectations of products of eigenregular polynomials}
In this section, we prove Lemma~\ref{lem:monomial}. It is restated below for the convenience of the reader.  
\begin{lemma*}\textbf{\ref{lem:monomial}}
Let $p_1, \ldots, p_t \in \mathbb{R}^n \rightarrow \mathbb{R}$ such that for $1 \le i \le t$, $p_i = I_{q_i}(h_i)$ where $p_i \in \mathcal{W}^{q_i}$. Further, $\mathsf{Var}(p_i) =1$ and $\lambda_{\max}(h_i) \le \kappa$. Let $C \in \mathbb{R}^{t \times t}$ where $C(i,j) = \langle h_i, h_j \rangle$. 
There exists a function $F : \mathbb{R}^{t \times t} \times \mathbb{Z}^{t} \rightarrow \mathbb{R}$ such that 
\[
\bigg| \mathbf{E}\big[\prod_{i=1}^t p_i \big] - F(C, q_1, \ldots, q_t) \bigg| \le  2^{t (q_1 + \ldots + q_t+1)} \cdot \kappa. 
\]
In other words, up to the error term of $2^{t (q_1 + \ldots + q_t)} \cdot \kappa$, the expectation of the product $\prod_{i=1}^t p_i$ is just dependent on the degrees of the polynomials and the covariance matrix of the polynomials. 
\end{lemma*}

To prove this lemma, we will express the product $p_1 \cdot \ldots \cdot p_t$ in terms of the tensors $\{h_1, \ldots, h_t\}$ using successive applications of Ito's multiplication formula.  Towards this, we first have the following definition. 
\begin{definition} 
Given $p_1, \ldots, p_t$ as above, a tuple $\mathbf{r} = (r_1, \ldots, r_{t-1})$ is said to be a contraction sequence if it satisfies the following definition: 
\begin{equation}\label{eq:sequence}
\textrm{For all } 1\le i \le t-1, \  0 \le r_i \le \min \bigg\{ q_{i+1}, q_i+ \sum_{j=1}^{i-1} (q_j - 2 r_j) \bigg\}.
\end{equation}
\end{definition}
~\\ Let $\mathcal{S}_{\mathsf{valid}}$ be the set of all contraction sequences. 
Also for a contraction sequence $\mathbf{r}=(r_1, \ldots, r_t)$,
define the quantity 
\[
\gamma_{\mathbf{r}} = \prod_{i=1}^{t-1} r_i! \cdot \binom{q_{i+1}}{r_i} \cdot \binom{q_i+ \sum_{j=1}^{i-1} (q_j - 2 r_j) }{r_i} \cdot \frac{\sqrt{(q_{i+1}+ \sum_{j=1}^{i} (q_j - 2 r_j) )!}}{\sqrt{q_{i+1}!} \cdot \sqrt{(q_i+ \sum_{j=1}^{i-1} (q_j - 2 r_j))!}}. 
\]
Also, let us define $g^{\mathbf{r}}_1 = h_1$ and for $1 < i  \le t$, define $g^{\mathbf{r}}_i $ iteratively as $g^{\mathbf{r}}_i  = g^{\mathbf{r}}_{i-1} \widetilde{\otimes}_{r_{i-1}} h_i$. Define  $m^{\mathbf{r}}_i = q_{i+1} + \sum_{j=1}^{i} (q_j - 2 r_j)$ and  observe that $g^{\mathbf{r}}_{i+1}  \in \mathcal{H}^{\odot m^{\mathbf{r}}_i}$. 
Applying Ito's multiplication formula (Proposition~\ref{prop:Ito}) iteratively, we obtain that
\begin{equation}\label{eq:Ito1}
\prod_{i=1}^t p_i = \sum_{\mathbf{r} =(r_1, \ldots, r_{t-1}) \in \mathcal{S}_{\mathsf{valid}}} \gamma_{\mathbf{r}} I_{m^{\mathbf{r}}_{t-1}}(g^{\mathbf{r}}_{t}).
\end{equation}
The next claim establishes a simple upper bound on $\gamma_{\mathbf{r}}$. 
\begin{claim}~\label{clm:bound-easy1}
$\sum_{\mathbf{r} =(r_1, \ldots, r_{t-1}) \in \mathcal{S}_{\mathsf{valid}}} |\gamma_{\mathbf{r}}| \le 2^{t (q_1 + \ldots + q_t +1 )} . $
\end{claim}
\begin{proof} By applying the fact that $\binom{n}{x} \le 2^n$, it is easy to show that 
\[
r_i! \cdot \binom{q_{i+1}}{r_i} \cdot \binom{q_i+ \sum_{j=1}^{i-1} (q_j - 2 r_j) }{r_i} \cdot \frac{\sqrt{(q_{i+1}+ \sum_{j=1}^{i} (q_j - 2 r_j) )!}}{\sqrt{q_{i+1}!} \cdot \sqrt{(q_i+ \sum_{j=1}^{i-1} (q_j - 2 r_j))!}} \le 2^{\sum_{j=1}^i (q_j - 2 r_j) + q_{i+1} + r_i }. 
\] 
This implies that 
\[
 |\gamma_{\mathbf{r}}|  \le 2^{\sum_{i=1}^{t-1} (\sum_{j=1}^i (q_j - 2 r_j) + q_{i+1} + r_i ) } \le 2^{t(q_1 +\ldots + q_t) - \sum_{i=1}^{t-1} r_{i}}. 
\]
As a result, we obtain 
\[
\sum_{\mathbf{r} =(r_1, \ldots, r_{t-1}) \in \mathcal{S}_{\mathsf{valid}}} |\gamma_{\mathbf{r}}| \le \sum_{\mathbf{r} =(r_1, \ldots, r_{t-1}) \in \mathcal{S}_{\mathsf{valid}}}2^{t(q_1 +\ldots + q_t) - \sum_{i=1}^{t-1} r_{i}} \le 2^{t(q_1 + \ldots + q_t +1)}. 
\]
\end{proof}
Our strategy for proving Lemma~\ref{lem:monomial} is as follows. We will first partition the sum in (\ref{eq:Ito1}) into two sets and show that the terms belonging to the first set are all bounded by $\kappa$. We will then further partition the terms in the  second set into two sets: Again, the terms in the first set will be bounded by $\kappa$ whereas terms in the second set will just be a function of $(C, q_1, \ldots, q_t)$ which will conclude our proof. Towards this, let us partition $\mathcal{S}_{\mathsf{valid}}$ into two partitions, $\mathcal{S}_z$ and $\mathcal{S}_{\mathsf{valid}} \setminus \mathcal{S}_z$ where the first set is defined as 
 $\mathcal{S}_z = \{\mathbf{r} \in \mathcal{S}_{\mathsf{valid}}:$ For all $i <t$, $(r_i=0) \vee (r_i = q_{i+1}) \}$.
Next, prove the following proposition. 
\begin{proposition}~\label{prop:approx-error2}
Let $h_1, \ldots, h_t$ be as above such that for all $1 \le i \le t$,  $ \Vert h_1 \Vert_F = \ldots = \Vert h_t \Vert_F=1$. Further, each $h_i$ is $\kappa$-eigenregular. If $\mathbf{r} \not \in \mathcal{S}_z$, then $\Vert g^{\mathbf{r}}_{t} \Vert_F \le \kappa$. 
\end{proposition}
\begin{proof}
If $\mathbf{r} \not \in \mathcal{S}_z$, then there is a coordinate $1 \le j \le t-1$ such that $r_j \not = 0$ and $r_j \not = q_{j+1}$. Appealing to Fact~\ref{fact:contraction}, it is easy to see that  $\Vert g^{\mathbf{r}}_{j} \Vert_F \le 1$. Next, we appeal to Fact~\ref{fact:contraction1}, we get $\Vert g^{\mathbf{r}}_{j} {\otimes}_{r_{j}} h_{j+1} \Vert_F \le \kappa$. As symmetrization decreases norm, we obtain that $\Vert g^{\mathbf{r}}_{j} \widetilde{\otimes}_{r_{j}} h_{j+1} \Vert_F \le \kappa$. Finally, again using that $\Vert h_{j+2} \Vert_F, \ldots, \Vert h_t \Vert_F \le 1$ and applying Fact~\ref{fact:contraction} iteratively, we obtain $\Vert g^{\mathbf{r}}_{t} \Vert_F \le \kappa$. This finishes the proof.
\end{proof}
To define further partitions of $\mathcal{S}_z$, we will need a somewhat more elaborate definition. To do this, consider any $\mathbf{r} \in \mathcal{S}_z$. Recall that $m^{\mathbf{r}}_i = q_{i+1} + \sum_{j=1}^{i} (q_j - 2 r_j)$ and $g_{i+1} \in \mathcal{H}^{\odot m^{\mathbf{r}}_i}$. For any $\ell \in \mathbb{N}$, let $\mathbb{S}_\ell$ denote the symmetric group on $\ell$ elements. Consider a tuple of permutations $\boldsymbol{\sigma}= (\sigma_1, \ldots, \sigma_{t-1})$ such that $\sigma_i \in \mathbb{S}_{m^{\mathbf{r}}_i}$. For $\boldsymbol{\sigma}$ and $\mathbf{r}$, we define $g^{\mathbf{r}, \boldsymbol{\sigma}}_i$ iteratively as 
$ 
g^{\mathbf{r}, \boldsymbol{\sigma}}_1 = h_1 \ \  \textrm{and} \ \ g^{\mathbf{r}, \boldsymbol{\sigma}}_i = \sigma_{i-1}(g^{\mathbf{r}, \boldsymbol{\sigma}}_{i-1} {\otimes}_{r_{i-1}} h_i).
$
Let $\Lambda_{\mathbf{r}} = \{\boldsymbol{\sigma} : \boldsymbol{\sigma} = ( \sigma_1, \ldots, \sigma_{t-1})\}$. With this notation, we have that 
$
g^{\mathbf{r}}_t = \frac{1}{|\Lambda_{\mathbf{r}}|} \cdot \sum_{\boldsymbol{\sigma} \in \Lambda_{\mathbf{r}}} g^{\mathbf{r}, \boldsymbol{\sigma}}_t. 
$
We will now partition pairs $(\mathbf{r}, \boldsymbol{\sigma})$ (where $\mathbf{r} \in \mathcal{S}_z$). 
For this subsequent partitioning, we will need to associate a graph with the pair $(\mathbf{r}, \boldsymbol{\sigma})$ where $\mathbf{r} \in \mathcal{S}_z$ and then partition based on this graph. To understand what this graph is meant to capture, note that to obtain $g^{\mathbf{r}, \boldsymbol{\sigma}}_{i+1}$ from $g^{\mathbf{r}, \boldsymbol{\sigma}}_{i}$, one of the following two events happen: (a) If $r_i=0$, then we first compute $g^{\mathbf{r}, \boldsymbol{\sigma}}_{i} \otimes h_{i+1}$ and then permute the indices of the resulting tensor by $\sigma_i$. Note that each time an index is added, we can associate a unique $1 \le j \le t$ with this index. (b) If $r_i =q_{i+1}$, we compute the contraction product of $g^{\mathbf{r}, \boldsymbol{\sigma}}_{i}$ with $h_{i+1}$ and subsequently permute the indices of the resulting tensor by $\sigma_i$. Note that crucially, the ``contraction" is along the last $r_i$ indices of $g^{\mathbf{r}, \boldsymbol{\sigma}}_{i}$. If we do a contraction at $i=i_0$, we can associate $i_0$ with all the indices that got collapsed at $i=i_0$. In other words, when an index gets created, we associate an element of $[t]$ and another one, when it gets annihilated. The graph structure is meant to capture this relation as formalized below. 

\begin{figure}[htb]
\hrule
\vline
\begin{minipage}[t]{0.98\linewidth}
\vspace{10 pt}
\begin{center}
\begin{minipage}[h]{0.95\linewidth}


\vspace{3 pt}
\underline{\textsf{Description of graph}}
\begin{enumerate}
\item Initialize bipartite graph with $L = R = \phi$. The vertex set $L$ will be ordered at all points in the algorithm. 

\item Add $q_1$ vertices (in order), each colored as `$1$' to $L$. 

\item For $i = 2$ to $i = t$, 
\begin{enumerate}
\item If $r_{i-1} =0$, then add $q_i$ vertices to $L$ with color `$i$' (in order). Permute the `unmatched' vertices in $L$ with permutation $\sigma_{i-1}$. 
\item If $r_{i-1} = q_i$, then add $q_i$ vertices to $R$ with color `$i$' (in order) and match them to the `last' $q_i$ vertices in $L$. Permute the unmatched vertices in $L$ with permutation $\sigma_{i-1}$. 
\end{enumerate}
\end{enumerate}

\vspace{5 pt}

\end{minipage}
\end{center}

\end{minipage}
\hfill \vline
\hrule
\label{fig:DS}
\end{figure}
~The correspondence between the graph constructed above and the contraction process used to generate $g^{\mathbf{r}, \boldsymbol{\sigma}}_{i+1}$ is obvious. For example, we have the following easy observation. 
\begin{obs}
The dimension of the tensor $g^{\mathbf{r}, \boldsymbol{\sigma}}_t$ is precisely $n^{|U|}$ where $U$ is the set of unmatched vertices in $L$. Consequently, 
$\mathbf{E}[g^{\mathbf{r}, \boldsymbol{\sigma}}_t] \not =0$ only if $U \not = \phi$. 
\end{obs} 
Since we are interested in the expectation of $\mathbf{E}[\prod_{i=1}^t p_i]$, thus from now onwards we only focus on $(\mathbf{r}, \boldsymbol{\sigma})$ such that there are no unmatched vertices in the graph generated above. Note that the property of having no unmatched vertices is just a property of $\mathbf{r}$ and not $\boldsymbol{\sigma}$. The next definition forms the crux  of our basis of partitioning the pairs $(\mathbf{r}, \boldsymbol{\sigma})$. 
\begin{definition}
For a pair $(\mathbf{r}, \boldsymbol{\sigma})$ such that the corresponding graph has no unmatched vertices, we say that $(\mathbf{r}, \boldsymbol{\sigma})$ is ``aligned" if and only each color `$i$', there is a unique color `$j$' such that vertices of color $i$ are only matched to vertices of color `$j$' and vice-versa. Else, it is said to be ``unaligned". 
\end{definition}
We will now prove the following claim. 
\begin{claim}\label{clm:approx-error2}
Let $(\mathbf{r}, \boldsymbol{\sigma})$ be an unaligned pair. Then,  $\Vert g^{\mathbf{r}, \boldsymbol{\sigma}}_t \Vert \le \kappa$.  \end{claim}
\begin{proof}
Note that we are assuming that there are no unmatched vertices in the graph. If the pair is $(\mathbf{r}, \boldsymbol{\sigma})$ is unaligned, then there is a unique smallest integer $1 < N\le t$ such that at the $N^{th}$ step of the algorithm to form the graph, the pair becomes unaligned. Namely, $N$ is the smallest integer such that one of the following events happen: ~\\
\textbf{Event A:} At the beginning of Step (3), when $i=N$, we have $r_{i-1} = q_{i}$ and the last $q_i$ vertices in $L$ are not of the same color. 
~\\ \textbf{Event B:} At the beginning of Step (3), when $i=N$, we have $r_{i-1} = q_{i}$, the last $q_i$ vertices in $L$ are of the same color (say $j$) but there is also another unmatched vertex in $L$ colored $j$. 
~\\ We ask the reader to verify that  $(\mathbf{r}, \boldsymbol{\sigma})$ is unaligned if and only if one of Event A or B happen. We will now prove that in both cases $A$ and $B$, our conclusion holds. ~\\
\textbf{Proof for Event A: } Let $\mathcal{I}_1 = \{1\} \cup \{N>i > 1: r_{i-1} =0 \}$ and $\mathcal{I}_2 =  \{N>i > 1: r_{i-1}  \not = 0 \}$. By definition of $N$, it is easy to verify that there is a one-one mapping $\nu$ from $\mathcal{I}_2$ to $\mathcal{I}_1$ such that for any $a \in \mathcal{I}_1$, vertices of color $a$ are matched with vertices of color $\nu(a)$ and vice-versa. 
Let us define $\mathcal{I}_3 = \mathcal{I}_1 \setminus \nu(\mathcal{I}_2)$. 
As a consequence, we have $$g^{\mathbf{r}, \boldsymbol{\sigma}}_{N-1} =\prod_{a \in \mathcal{I}_2} \langle h_a , h_{\nu(a)} \rangle \cdot \mu (\otimes_{j \in \mathcal{I}_3} h_j),$$
where $\mu$ is a permutation of the indices of  $\otimes_{j \in \mathcal{I}_3} h_j$. Further, since at $i=N$, the last $q_N$ vertices of $L$ are of at least two different colors, we have that 
the last $q_N$ indices of $\mu (\otimes_{j \in \mathcal{I}_3} h_j)$ belong to more than one $h_j$ (for $j \in \mathcal{I}_3$). Let $A$ be the set of all indices of $\otimes_{j \in \mathcal{I}_3} h_j$ which do not map to the last $q_N$ positions under $\mu$. Consider some assignment $\boldsymbol{\rho}$ for these indices. Then, we have that 
\[
\Vert g^{\mathbf{r}, \boldsymbol{\sigma}}_{N} \Vert_F^2 = \sum_{\boldsymbol{\rho}} \Vert \prod_{a \in \mathcal{I}_2} \langle h_a , h_{\nu(a)} \rangle \cdot \mu (\otimes_{j \in \mathcal{I}_3} h_{j, \boldsymbol{\rho}}) \otimes_{q_N} h_N \Vert_F^2.\]
Here $h_{j, \boldsymbol{\rho}}$ is tensor obtained by restricting the indices in $A$ to $\rho$. Since $\lambda_{\max}(h_N) \le \kappa$, we get
\[
\Vert \mu (\otimes_{j \in \mathcal{I}_3} h_{j, \boldsymbol{\rho}}) \otimes_{q_N} h_N \Vert_F^2 \le \kappa^2 \cdot \Vert \otimes_{j \in \mathcal{I}_3} h_{j, \boldsymbol{\rho}}\Vert_F^2. 
\]
This implies that 
\[
\Vert g^{\mathbf{r}, \boldsymbol{\sigma}}_{N} \Vert_F^2 \le \kappa^2  \sum_{\boldsymbol{\rho}}  \prod_{a \in \mathcal{I}_2} \langle h_a , h_{\nu(a)} \rangle^2 \cdot  \Vert \otimes_{j \in \mathcal{I}_3} h_{j, \boldsymbol{\rho}}\Vert_F^2 \le  \kappa^2 \prod_{a \in \mathcal{I}_2} \langle h_a , h_{\nu(a)} \rangle^2  \le \kappa^2. 
\]
The penultimate inequality uses the fact that $\sum_{\boldsymbol{\rho}}  \Vert \otimes_{j \in \mathcal{I}_3} h_{j, \boldsymbol{\rho}}\Vert_F^2=1$. Using the fact that $\Vert g_j \Vert_F=1$ for all $j>N$ and applying Fact~\ref{fact:contraction}, we get that $|g^{\mathbf{r}, \boldsymbol{\sigma}}_{t}|\le \kappa$. This finishes the proof. 
~\\
\textbf{Proof for Event B: } Let us define $\mathcal{I}_1$, $\mathcal{I}_2$ and $\mathcal{I}_3$ as in Event A. As before, we have
\[
g^{\mathbf{r}, \boldsymbol{\sigma}}_{N-1} =\prod_{a \in \mathcal{I}_2} \langle h_a , h_{\nu(a)} \rangle \cdot \mu (\otimes_{j \in \mathcal{I}_3} h_j),
\]
where $\mu$ is a permutation of the indices of  $\otimes_{j \in \mathcal{I}_3} h_j$. The crucial difference is that there exists $j_0 \in \mathcal{I}_3$ such that if $\mathcal{B}$ denotes the set of positions where the indices of $h_{j_0}$ get mapped under $\mu$, then this includes the last $q_N$ indices as a proper subset. 
Using $\lambda_{\max}(h_{j_0}) \le \kappa$, we get that 
$\Vert \mu (\otimes_{j \in \mathcal{I}_3} h_{j}) \otimes_{q_N} h_N \Vert_F^2 \le \kappa^2 \cdot \Vert \otimes_{j \in \mathcal{I}_3 \setminus \{j_0\}} h_{j}\Vert_F^2. $
However, $\Vert \otimes_{j \in \mathcal{I}_3 \setminus \{j_0\}} h_{j}\Vert_F^2. =1$, and hence we get $\Vert g^{\mathbf{r}, \boldsymbol{\sigma}}_{N} \Vert_F^2 \le \kappa^2$.
 Using the fact that $\Vert g_j \Vert_F=1$ for all $j>N$ and applying Fact~\ref{fact:contraction}, we get that $|g^{\mathbf{r}, \boldsymbol{\sigma}}_{t}|\le \kappa$. This finishes the proof.
\end{proof}

\begin{claim}
Let $(\mathbf{r}, \boldsymbol{\sigma})$ be an aligned pair. Then, $g^{\mathbf{r}, \boldsymbol{\sigma}}_{N}$ is just dependent on the covariance matrix $C$ (and the pair $(\mathbf{r}, \boldsymbol{\sigma})$). 
\end{claim}
\begin{proof}
Since $(\mathbf{r}, \boldsymbol{\sigma})$ is an aligned pair, it means that the set $\{1, \ldots, t\}$ can be partitioned into two sets given by $L$ and $R$ and a bijection $\pi : L \rightarrow R$  such that 
$g^{\mathbf{r}, \boldsymbol{\sigma}}_{t} = \prod_{a \in L} \langle h_a, h_{\pi(a)} \rangle$. Note that the mapping $\pi$ is just a function of $(\mathbf{r}, \boldsymbol{\sigma})$ and once $\pi$ is fixed, $g^{\mathbf{r}, \boldsymbol{\sigma}}_{t} $ is just a function of the matrix $C$. 
\end{proof}
\begin{proofof}{Lemma~\ref{lem:monomial}}
Let $B = \{(\mathbf{r}, \boldsymbol{\sigma}): \mathbf{r} \in \mathcal{S}_z \ \textrm{and} \ (\mathbf{r}, \boldsymbol{\sigma}) \ \textrm{ is aligned}\}$.
Applying (\ref{eq:Ito1}), we get that
\[
\bigg| \mathbf{E}[\prod_{i=1}^t p_i] - \frac{1}{|\Lambda_{\mathbf{r}}|}\sum_{(\mathbf{r}, \boldsymbol{\sigma}) \in B} \gamma_{\mathbf{r}} g^{\mathbf{r}, \boldsymbol{\sigma}}_{t}\bigg| \leq
\sum_{\mathbf{r} \not \in \mathcal{S}_z} |\gamma_{\mathbf{r}} \cdot g^{\mathbf{r}}_{t} | +  \frac{1}{|\Lambda_{\mathbf{r}}|}\sum_{(\mathbf{r}, \boldsymbol{\sigma}) \not \in B :\mathbf{r} \in \mathcal{S}_z} |\gamma_{\mathbf{r}} \cdot g^{\mathbf{r}, \boldsymbol{\sigma}}_{t} |.
\]
Note that $\frac{1}{|\Lambda_{\mathbf{r}}|}\sum_{(\mathbf{r}, \boldsymbol{\sigma}) \in B} \gamma_{\mathbf{r}} g^{\mathbf{r}, \boldsymbol{\sigma}}_{t}$ is dependent just on $C$ and $\mathbf{r}$. Further, applying Claim~\ref{clm:approx-error2}, Proposition~\ref{prop:approx-error2} and Claim~\ref{clm:bound-easy1}, we see that the right hand side can be bound by
$2^{t(q_1 + \ldots + q_t +1)} \cdot \kappa$. This finishes the proof. 
\end{proofof}
\end{document}